\documentclass{article}
\usepackage{preprint}          
\usepackage[utf8]{inputenc}\usepackage[T1]{fontenc}
\usepackage{amsmath,amssymb,amsthm,mathtools}
\usepackage{graphicx,booktabs,microtype}
\usepackage{enumitem}
\usepackage[round,authoryear]{natbib}
\usepackage{hyperref}\usepackage{cleveref}
\usepackage{xcolor}
\usepackage{bm}

\theoremstyle{plain}
\newtheorem{theorem}{Theorem}\newtheorem{lemma}{Lemma}
\newtheorem{proposition}{Proposition}\newtheorem{corollary}{Corollary}
\theoremstyle{definition}
\newtheorem{definition}{Definition}
\theoremstyle{remark}\newtheorem{remark}{Remark}

\newcommand{\V}{\bm V}\newcommand{\U}{\bm U}\newcommand{\X}{\bm X}
\newcommand{\dosep}{\perp_{d}}\newcommand{\ecgsep}{\perp_{\mathrm{ECG}}}
\newcommand{\indep}{\perp\!\!\!\perp}
\newcommand{\doop}{\mathrm{do}}\newcommand{\Prob}{\mathbb{P}}\newcommand{\E}{\mathbb{E}}
\newcommand{\pa}{\mathrm{pa}}
\newcommand{\R}{\mathbb{R}}\newcommand{\eqd}{\overset{d}{=}}
\newcommand{\GL}{\mathrm{GL}}\newcommand{\Ortho}{\mathrm{O}}\newcommand{\SO}{\mathrm{SO}}
\newcommand{\rank}{\operatorname{rank}}\newcommand{\diag}{\operatorname{diag}}\newcommand{\corank}{\operatorname{corank}}

\newcommand{\interpretation}[1]{\par\nopagebreak[4]\noindent #1\par}

\title{Equilibrium Causal Games:\\ Separation, Identification, and the Identifiability of Cyclic Latent States}
\author{
Faraz Dadgostari\\
Department of Mechanical \& Industrial Engineering, Montana State University
\and
Neda Nazemi\\
Gianforte School of Computing, Montana State University
}
\date{}
\begin{document}\maketitle

\begin{abstract}
Many complex systems, including power grids, markets, and interacting populations, settle into feedback driven equilibria observed only through unknown sensors. Our Equilibrium Causal Game (ECG) joins a game to its cyclic causal model, hidden inputs, sensor map, and rules for interventions and equilibrium selection; interventions edit declared objects and recompute equilibrium.
Under stated conditions, ECG-separation is sound but incomplete in our finite/smooth examples. Back-door/half-trek routes identify observed queries. Yet for an untouched rotationally symmetric Gaussian block, second moments determine only a source-frame rotation, across which distinct-variable effects generically change.
Unknown sensing creates a separate ambiguity. In passive stable linear models without self-effects, unknown wiring and full-rank unknown sensing leave $B$ completely unidentified for $d\ge2$. Under LiNG, non-Gaussianity removes the source rotation; mechanism interventions separate sensing from interactions. With unknown support, invariant sensing, aligned responses, and well-posed single-target interventions identify $(H,B)$ up to declared equivalence. Of $d$ targets, $d-1$ suffice exactly when the sole untargeted node directly parents all others; otherwise $d$ are needed. Acquisition probes are excluded; known wiring gives no universal count.
With nonlinear sensing, isotropic Gaussian source blocks admit hidden twists within and across blocks in labelled environments preserving required radial laws. Conversely, under stated positivity, informative one-block changes, rank, and irreducibility conditions, the finest independent source-block representation is identified within the stated alternative class up to block permutation and blockwise coordinate changes, but not downstream mechanisms or the sensor/interaction split. Together, these results show which causal conclusions equilibrium data support and which require targeted experiments.

\end{abstract}

\section{Introduction}
\label{sec:intro}
\paragraph{Motivation and setting.}
We ask how a declared intervention changes an interacting system's selected equilibrium. In the ECG
intervention algebra, an intervention edits a declared ECG object and re-solves
the structural system; for linear treatment inputs, the total equilibrium effect is
$e_Y^\top(I-B)^{-1}\Gamma$. Cyclic equilibrium and latent observation through an unknown sensor map create
distinct obstacles. Cyclic
equilibrium mechanisms do not admit an acyclic recursive ordering: acyclic $d$-separation can be unsound, and
acyclification can delete feedback contributions carried by $(I-B)^{-1}$ (\Cref{prop:counter,prop:resolvent}).
At the latent layer, observing $\X=H\V$ with unknown full-column-rank $H$ need not determine $B$ in the passive,
stable, zero-diagonal unknown-support class; known support, anchors, or calibration can instead shrink the legal
fibre (\Cref{thm:gauge}).

\paragraph{Two transverse obstructions.}
The source-frame obstruction is a rotation on an un-intervened full-rotational source component. In the
Gaussian or second-moment regime it preserves observational covariance while generically changing effects
between distinct component variables (\Cref{thm:hedge,thm:gauge}). In the noiseless LiNG submodel---or
conditionally on a separate noisy-ICA identifiability theorem for the declared sensor-noise law---non-Gaussian
source acquisition reduces this rotation to permutation and scaling, but independent environment-wise ICA
neither aligns response maps nor separates $H$ from $B$ (\Cref{prop:nongauss}).

The latent-frame obstruction is the legal fixed-$M$ chart in the ambient zero-diagonal unknown-support,
unknown-$H$ class. For $d\ge2$, it persists for any fixed admissible source/noise law. Shift-only environments
likewise leave $B$ unidentified at the population structural-response level and, observationally, under the
calibrated, full-rank, mean-stable sensing conditions of \Cref{thm:linear} (\Cref{thm:gauge,prop:shift}). The two actions are transverse rather than nested
(\Cref{rem:transverse}) and require different information to constrain. In the LiNG identification branch,
aligned population response maps and well-posed single-target mechanism interventions recover the latent
factorization up to the declared equivalence under the conditions of \Cref{thm:linear}; non-Gaussian source
acquisition therefore cannot substitute for the aligned mechanism information needed to split $H$ from $B$.
For full structural recovery in this ambient unknown-support class, $K=d-1$ mechanism targets suffice if and
only if the sole untargeted node directly parents every other node in the whole graph; otherwise $K=d$, assuming
aligned response maps and well-posed targeted edits, and excluding the acquisition probes from $K$
(\Cref{thm:linear,prop:completion}). With known support, on a smooth quotient stratum where
$D\mathcal C_{G,T}$ has constant rank, local point identification is equivalent to full column rank on the actual
quotient tangent. Global point identification holds if and only if the complete legal fibre is one
declared-equivalence class, while a query is globally identified if and only if that query is constant on the
complete legal fibre; no universal target count follows (\Cref{prop:knownsupport}).

\paragraph{Contributions and scope.}
The paper makes six contributions within these declared regimes.
\begin{enumerate}[leftmargin=1.6em,itemsep=2pt,topsep=2pt]
\item It defines the ECG, its equilibrium graph, and its intervention algebra (\Cref{sec:formalism}). SCC-wise
contraction gives a jointly measurable unique equilibrium map under the existence theorem's conditions, while
regular selected branches require the separate hypotheses of \Cref{thm:regbranch}.
\item Under represented exogeneity and unique solvability with respect to every strongly connected subset,
ECG-separation is sound, while the stated finite and smooth constructions show that it is not complete
(\Cref{thm:sound,prop:complete}). For linear ECGs, the trek results provide the stated generic Gaussian-CI and
zero-partial-covariance criteria, while functional separation gives an exact structural sufficient condition
for full conditional independence under every noise law (\Cref{thm:trek,thm:trekng,thm:funcsep}).
\item For observed-layer queries, the stated back-door and valid structural half-trek routes are sufficient,
while the full-rotational Gaussian hedge gives a scoped non-identification construction
(\Cref{thm:backdoor,thm:soundid,thm:hedge}). Once $B$ is identified, weighted-potential structure supplies
testable symmetrizability and cycle restrictions. Separately, when passive information leaves a hedge fibre, the
restriction is generically gauge-reducing on the stated regular and support strata, with point identification
only in the stated cases. The codimension-based gain is not
game-specific (\Cref{thm:game}).
\item Passive linear data in the stable zero-diagonal unknown-support class with unknown full-column-rank $H$
admit the legal chart of \Cref{thm:gauge}. Given aligned response maps and the acquisition and well-posedness conditions of
\Cref{thm:linear}, full mechanism-target coverage identifies $(H,B)$ up to the declared equivalence, and
\Cref{prop:completion} gives the graph-global $K=d-1$ condition. With known support, the constrained-rank and
complete-fibre conditions replace any universal target count (\Cref{prop:knownsupport}).
\item For diffeomorphic mixing, an isotropic-Gaussian source block $S\subset\R^m$, $m\ge2$, admits the
within-block twist collision; two isotropic blocks $A\in\R^{m_A}$ and $B\in\R^{m_B}$, $m_A\ge2$, $m_B\ge1$,
admit the cross-block collision, when the labelled available environments preserve the radial/isotropic laws
used by the constructions (\Cref{thm:block}). For
independent source blocks with strictly positive $C^1$ densities and nonconstant labelled soft $C^1$
single-block reweightings, score rank and irreducibility identify the finest source-block representation, up to
the declared block gauge, among $C^1$-diffeomorphic alternatives with independent blocks, matched dimension and
image, and single-alternative-block reweightings, but not downstream components or a general
equilibrium-component decomposition and, for $d\ge2$, not the $h/B$ split (\Cref{thm:blockid-source}). For a strongly connected
standardized-Gaussian source component of size $m\ge2$ whose intervened systems are well posed, one perfect
intervention $\doop(V_j\!\leftarrow\!W_j)$ per node with $W_j\sim\mathcal N(0,s_j^2)$ supplies score rank outside
a proper algebraic measure-zero set of positive intervention variances and irreducibility at every positive
variance, whereas per-node shifts on at least two distinct targets violate irreducibility
(\Cref{rem:cycleshelp,lem:cyclesRI}).
\item Experiments 1--14 provide finite-instance evidence in the named designs without replacing the proofs or
establishing operational effectiveness (\Cref{sec:exp,sec:limits}).
\end{enumerate}

\paragraph{Roadmap.}
Related work is reviewed in \Cref{sec:related}. \Cref{sec:formalism,sec:theory} define the ECG and develop
observed-layer separation and identification. \Cref{sec:bridge,sec:model,sec:latent} distinguish the two
ambiguity actions and develop the linear and nonlinear latent results; \Cref{sec:algo} records the algorithms,
and \Cref{sec:exp} reports Experiments 1--14. The final main-text sections record limitations and computational
information (\Cref{sec:limits,sec:computational}), followed by proofs and the conditional-cumulant criterion
(\Cref{app:proofs,app:cumulant}).

\section{Related work}
\label{sec:related}
\textbf{Cyclic SCMs and separation.} \citet{Bongers2021} give the measure-theoretic theory of cyclic
SCMs---solvability, marginalization, and the generalized directed global Markov property w.r.t.\
$\sigma$-separation \citep{Forre2018}; \citet{Bongers2018theoretical} treat marginalization, and
\citet{ForreMooij2019calculus} give a do-calculus and identification algorithm for cyclic latent (io)SCMs.
\citet{Ferradini2025} introduce a sound-and-complete $p$-separation for finite consistent cyclic models, and
\citet{Dai2026equiv} characterize distributional equivalence for linear non-Gaussian cyclic latent models.
For \emph{linear} SEMs, \citet{Sullivant2010trek} prove the (acyclic) trek-separation theorem and
\citet{FoygelDraismaDrton2012} the half-trek identifiability criterion. Closest to our separation result,
\citet{Spirtes2013rank} \emph{already} treats trek-separation / entailed rank constraints for
\emph{partially non-linear and cyclic} latent models, licensing constraint-based structure search over a fixed
observed variable set; we do not extend that structure but \emph{use and re-derive} it inside a
\emph{query-relative} identification target under an \emph{unknown} observation map (a different problem), with
a self-contained proof. ECGs \emph{use} these
criteria but add the game layer.

\textbf{Equilibrium, settable systems, and potential games.} \citet{White2009settable} re-engineer causal
primitives for optimizing, interacting agents and allow multiple fixed points, with graphical separation for
settable systems developed by \citet{ChalakWhite2012settable}; \citet{Lauritzen2002chain}
read chain graphs as feedback equilibria; \citet{PetersHalpern2025gsem} propose generalized SEMs as a
super-class; \citet{Hammond2023games} give a causal-graphical semantics for games via mechanised causal game
graphs; equilibrium relations with conservation laws need causal-constraints models \citep{Blom2019ccm}.
None supplies a graphical CI/identification oracle for a \emph{cyclic equilibrium} fixed point under an unknown
observation map \emph{or} exploits potential-game structure for identification. The structure we exploit
is that of \emph{potential games} \citep{MondererShapley1996}: a smooth game is a (weighted) potential game
iff its payoff-gradient Jacobian is (diagonally) symmetric/symmetrizable---the symmetric-Hessian (Poincar\'e)
characterization---which for the linear-quadratic network games of \citet{Bramoulle2014networks} and the
variational-inequality framework of \citet{Parise2019variational} is a condition on the interaction matrix. We
turn this classical characterization into a \emph{causal} specification test (and a codimension-based
identification consequence; \Cref{thm:game}).

\textbf{Multi-agent IRL and inverse games.} Demonstrations from several agents are an equilibrium, so
``optimality'' becomes Nash/correlated/logit-QRE \citep{Yu2019maairl}; MA-AIRL learns rewards from a
logit-equilibrium data-generating process, and \citet{Liao2026inversegame} establish reward identifiability
under quantal response. These works \emph{presuppose} the equilibrium is the data-generating process; ECGs
give that process a cyclic-SCM semantics so ``intervention on the game'' is well defined and confounding can
be reasoned about \citep{Zhang2020causalimitation}.

\textbf{Electricity-market Cournot models.} Linear-demand, constant-marginal-cost Cournot competition yields
linear reaction functions; calibrated oligopoly models are standard for wholesale power
\citep{BorensteinBushnell1999}. Our semi-real experiment (\Cref{sec:exp}, Exp.~5) uses such a calibration as a
ground-truth feedback system.

\textbf{Interventional CRL (acyclic).} Identifiability from interventions is now well understood for
\emph{acyclic} latent graphs: nonparametric guarantees under a pair of distinct \emph{perfect-intervention
domains} per node \citep{vonKugelgen2023nonparam} (with a single-domain statement at $d{=}2$), linear theory
\citep{Squires2023linear,Buchholz2023linear}, soft
interventions \citep{Zhang2023identifiability,Ahuja2022interventional}, and weak supervision
\citep{Brehmer2022weakly}. \citet{Squires2023linear} prove that for linear latents a single intervention per
node is necessary and sufficient---our \Cref{thm:linear} is the cyclic generalization, where the equilibrium
map $(I-B)^{-1}$ entangles nodes and a residual hedge gauge appears. \textbf{Nonlinear ICA / world models.}
Identifiable nonlinear ICA uses auxiliary variables \citep{Hyvarinen2019nonlinear,Khemakhem2020ivae};
sparse mechanism shift \citep{Lachapelle2023additive} and grouping \citep{Morioka2023grouping} relax
independence; CITRIS \citep{Lippe2022citris} handles temporal latents but assumes \emph{no} instantaneous
relations, and iCITRIS \citep{Lippe2022icitris} relaxes this to instantaneous effects under observed
intervention targets---but its instantaneous graph is \emph{acyclic}, so neither addresses the cyclic
(equilibrium) relations central here. \textbf{Linear non-Gaussian and cyclic models.} LiNGAM
\citep{Shimizu2006lingam} and its cyclic generalization LiNG \citep{Lacerda2008cyclic} identify the
\emph{observed}-variable cyclic structure up to a distributional-equivalence class; \citet{Dai2026equiv}
characterize this class for \emph{latent}-variable linear non-Gaussian cyclic models---the closest prior to
ours, which we extend by adding interventions that collapse the equivalence class and by establishing the
hedge boundary. \textbf{Shift interventions.} backShift \citep{Rothenhausler2015backshift} identifies a
cyclic $B$ over \emph{observed causal variables} (latent confounders are allowed; identification uses second
moments across $\ge3$ shift environments); the variables themselves, however, are not a learned
representation. \citet{Hyttinen2012learning} give a complete interventional-identifiability characterization
(the pair condition) for linear cyclic models with latent variables from surgical interventions across
experiments---again over observed variables, without an unknown observation map. We show (\Cref{prop:shift}) that,
for $d\ge2$, once an \emph{unknown mixing} $\X=H\V$ sits between the cyclic state and the data, calibrated shifts identify only $H(I-B)^{-1}$ at the structural-response level
(and observationally under the mean-stable sensing conditions of \Cref{thm:linear}) and are \emph{insufficient} to
split $H$ from $B$ in the ambient unknown-support, unknown-$H$ class. The mechanisms of \Cref{thm:linear} are
one sufficient source of the missing information; known support or sensor/anchor calibration can instead shrink
the admissible fibre. At $d=1$, zero diagonality fixes $B=0$ and calibrated shifts recover $H=M_0$.
\textbf{Structural identification.} Our linear recovery uses the trek-separation oracle and the half-trek
criterion \citep{Sullivant2010trek,FoygelDraismaDrton2012} as we lift it to cyclic ECGs in \Cref{thm:trek}.

\section{The Equilibrium Causal Game (the object)}
\label{sec:formalism}
We work with the SCM tuple $\mathcal{M}=(\V,\U,\bm f,\Prob_{\U})$ of \citet{Bongers2021}: endogenous
$\V=(V_1,\dots,V_n)$, exogenous $\U$ with law $\Prob_{\U}$, structural equations
$V_i=f_i(\V_{\pa(i)},\U_i)$, and graph $\mathcal{G}(\mathcal{M})$ with $V_j\to V_i$ iff $j\in\pa(i)$. A
\emph{solution} for $\U=\bm u$ is a fixed point $\bm v=\bm f(\bm v,\bm u)$. $\mathcal{M}$ is \emph{uniquely
solvable w.r.t.}\ $O$ if for $\Prob_{\U}$-a.e.\ $\bm u$ and every value of $\V_{O^c}$ the subsystem
$\bm v_O=\bm f_O(\bm v_O,\bm v_{O^c},\bm u)$ has a unique measurable solution
$\bm v_O=\bm g_O(\bm v_{O^c},\bm u)$, the \emph{equilibrium (solution) map}.

\begin{definition}[Equilibrium Causal Game]\label{def:ecg}
An \emph{equilibrium causal game} is the tuple
\[
\begin{aligned}
\mathcal{E}=\bigl(&N,(\mathcal{A}_i),(u_i),\mathsf{Eq},\mathcal{I},\Prob_{\U},\\
                  &\mathsf{Sel},\mathcal{O},\mathcal{C}\bigr)
\end{aligned}
\]
with players $i\in\{1,\dots,N\}$, action spaces $\mathcal{A}_i$, payoffs $u_i(a_i,\bm a_{-i};\bm w)$, an
\emph{equilibrium concept} $\mathsf{Eq}$ (Nash, correlated, or logit/quantal-response with rationality
$\lambda$), an \emph{information structure} $\mathcal{I}$, exogenous law $\Prob_{\U}$, an
\emph{equilibrium-selection rule} $\mathsf{Sel}$, an \emph{observation regime} $\mathcal{O}$, and (possibly
empty) causal constraints $\mathcal{C}$ \citep{Blom2019ccm}. Its \emph{induced SCM} $\mathcal{M}(\mathcal{E})$
has endogenous $\bm a=(a_1,\dots,a_N)$ plus any coupling variables and any explicit latent source nodes required
by the represented-exogeneity convention below, exogenous primitive inputs $\U$, and structural equations
given by the $\mathsf{Eq}$-fixed-point (best-response) conditions, e.g.\ for logit-QRE
\begin{equation}\label{eq:qre}
\pi_i(a_i\mid \bm w)\;\propto\;\exp\!\bigl(\lambda\,\E_{\bm a_{-i}\sim\bm\pi_{-i}}[\,u_i(a_i,\bm a_{-i};\bm w)\,]\bigr).
\end{equation}
An SCC of $\mathcal{G}(\mathcal{M}(\mathcal{E}))$ is a \emph{cyclic equilibrium component} iff it contains a
directed cycle: either it has $\ge2$ mutually dependent variables or it is a singleton with a self-loop. A
singleton SCC without a self-loop is an \emph{acyclic singleton component}; its structural assignment is
evaluated directly after its parent components.
\end{definition}

\paragraph{Represented exogeneity.}\label{par:represented-exogeneity}
Every graphical Markov claim below uses the following convention. The stochastic inputs admit a representation
by mutually independent primitive source blocks; whenever one primitive source enters more than one structural
equation it is included as an explicit latent common parent, and $\sigma$-separation is evaluated in that augmented
graph. Thus every exogenous dependence relevant to the structural equations is present
in the graph. An arbitrary dependent $\Prob_{\U}$ whose common causes are omitted from
$\mathcal{G}^{\mathrm{eq}}$ is not covered by a graphical conditional-independence claim.

\begin{definition}[ECG-separation]\label{def:ecgsep}
Let $\mathcal{G}^{\mathrm{eq}}(\mathcal{E})\coloneqq\mathcal{G}(\mathcal{M}(\mathcal{E}))$ be the equilibrium
graph. $\bm A,\bm B$ are \emph{ECG-separated} given $\bm C$ ($\bm A\ecgsep\bm B\mid\bm C$) iff they are
$\sigma$-separated in $\mathcal{G}^{\mathrm{eq}}(\mathcal{E})$, using all of that graph's SCCs in the
$\sigma$-criterion of \citet{Forre2018,Bongers2021}; its cyclic SCCs are the equilibrium components.
\end{definition}

\begin{definition}[Intervention algebra]\label{def:intervention}
For an ECG $\mathcal{E}$, four intervention families re-solve the structural system after the edit. The SCCs of
the edited graph are recomputed, and the intervention is well-defined whenever the edited system is uniquely
solvable w.r.t.\ every recomputed SCC, including each singleton self-cycle; acyclic singleton assignments are
measurable and evaluated directly, and the component maps compose to a unique measurable global solution map:
(i) \emph{mechanism} $\doop(V_i{=}v)$; (ii) \emph{payoff} $\doop(u_i{\leftarrow}\tilde u_i)$;
(iii) \emph{player-set} $\doop(N{\leftarrow}N')$; (iv) \emph{selection}
$\doop(\mathsf{Sel}{\leftarrow}\widetilde{\mathsf{Sel}})$. All results are stated for selection-stable ECGs (in
the uniquely-solvable case, $\mathsf{Sel}$ is vacuous).
\end{definition}

\section{Separation and identification at the observed layer}
\label{sec:theory}
We first record the vehicle's soundness (a reduction), then give the three principal results.

\begin{theorem}[Existence/uniqueness of the equilibrium map]\label{thm:exist}
Suppose every cyclic equilibrium component, including every singleton self-cycle, has a fixed nonempty complete
separable metric state space with its Borel $\sigma$-field and a fixed seed point. Suppose also that there is a
Borel set $\mathcal U_0$ with $\Prob_{\U}(\mathcal U_0)=1$ such that, for every $\bm u\in\mathcal U_0$, each
cyclic component best-response map is jointly Borel in the component state, held-fixed coordinates, and
$\bm u$, maps that state space into itself, and is a contraction in the fixed complete metric defining that state space, with a
component-specific modulus $q_S<1$ uniform in the held-fixed coordinates (e.g.\ $\rho(B)<1$ in the linear case
$\bm v=B\bm v+\bm u$ on the component).
Suppose every acyclic singleton component has a structural assignment that is jointly Borel in its parents and
$\bm u$ and maps into its declared Borel state space. Then $\mathcal{M}(\mathcal{E})$ is uniquely solvable
w.r.t.\ every SCC. Composing the component solution maps in a topological order of the condensation DAG gives a
jointly measurable and a.s.-unique global equilibrium map $\bm g$.
\emph{Algebraic} unique solvability is weaker: in the linear case it holds iff $\det(I-B)\neq0$, which can
hold even when $\rho(B)\ge1$; when coefficients or forcing terms vary, their dependence on the declared
parameters is assumed Borel so that the explicit inverse solution is measurable.
\end{theorem}

\begin{theorem}[Regular selected branches of admissible ECGs]\label{thm:regbranch}
Write the equilibrium conditions as $F(v,p)=0$, where $p=(u,\eta)$ collects exogenous and intervention
parameters. Suppose $F$ is jointly Borel, $F(\cdot,p)$ is $C^1$ in the endogenous coordinate, and a declared
branch $g(p)$ is jointly Borel, satisfies $F(g(p),p)=0$ for $\Prob_{\U}$-a.e.\ $u$ and every admissible $\eta$,
and has square nonsingular $D_vF(g(p),p)$. Then $\mathcal E$ induces a cyclic SCM with a measurable selected
equilibrium map; for fixed $p$ the selected root is locally unique in $v$, and each declared post-intervention
branch is well defined. No global root uniqueness or cross-stratum continuity follows. On any open
finite-dimensional parameter chart where $F$ is jointly $C^1$ in $(v,p)$ and $g(p)$ is the unique regular root,
$g$ is locally $C^1$ and
\[
 D_pg(p)=-D_vF(g(p),p)^{-1}D_pF(g(p),p).
\]
Thus measurability is supplied by the declared Borel branch, whereas parameter smoothness follows only on the
stated unique-regular-root charts. Nonsingularity of $D_vF$ is exactly algebraic unique solvability of the
linearization ($\det(I-B)\neq0$ in the linear case).
\emph{This covers admissible systems that contraction does not}: for $c>1/2$ the grid component has
$\rho(B)=\sqrt{2c}>1$ (naive iteration diverges) yet $\det(I-B)=1+2c\neq0$ and a nonsingular branch Jacobian,
so its equilibrium law is valid even though the fixed-point iteration is not contractive.
\end{theorem}
\begin{corollary}[Admissible game classes]\label{cor:games}
(i) \emph{Logit-QRE finite games}: existence by Brouwer on the compact simplex, a regular isolated branch
under nonsingular equilibrium Jacobian (uniqueness for small rationality/diagonal dominance). (ii)
\emph{Entropy-selected correlated equilibrium}: a nonempty CE polytope with a unique entropy maximizer yields
one measurable branch. (iii) \emph{Smooth constrained games}: on the actual active manifold, write the coupled
KKT conditions as a generalized equation and require locally Lipschitz dependence of its residual and constraint
data on $p$, the relevant constraint qualification, playerwise SSOSC on the true critical cones, strict
complementarity wherever an active-set reduction uses it, and strong regularity. The result is a locally
single-valued Lipschitz constrained \emph{local}-Nash/KKT branch (locally $C^1$ when the data are jointly $C^1$
and the reduced active system is nonsingular). It is a unique \emph{global} Nash branch only under a separate
global best-response sufficiency and an existence/uniqueness condition, such as an appropriately specified
strongly monotone variational inequality. (iv) \emph{Causal constraints}: including conservation or
market-clearing equations in $F$ with constant rank gives a local constraint/solution manifold, not an isolated
equilibrium. Admissibility additionally requires a declared Borel selection and a nonsingular or strongly
regular reduced system on the actual tangent/active manifold. Smooth nonsingular active-set reductions are
instances of \Cref{thm:regbranch}; a strongly regular nonsmooth generalized equation instead supplies the
analogous local Lipschitz admissibility result directly, with a declared Borel selection outside that local
chart. The local/global and manifold/isolation distinctions are load-bearing.
\end{corollary}

\begin{theorem}[Soundness of ECG-separation; reduction]\label{thm:sound}
Assume represented exogeneity (\S\ref{par:represented-exogeneity}) and the following sufficient
loop-solvability premise for the generalized directed global Markov property:
$\mathcal{M}(\mathcal{E})$ is uniquely solvable w.r.t.\ every
strongly connected subset (in particular, a simple SCM satisfies this premise). Then
$\bm A\ecgsep\bm B\mid \bm C \Rightarrow \V_{\bm A}\indep \V_{\bm B}\mid \V_{\bm C}$. The content is the
\emph{reduction}: an equilibrium game with the algebra of \Cref{def:intervention} is a cyclic SCM whose
equilibrium components are all of its cyclic SCCs (including singleton self-cycles) and whose augmented graph
represents every common source; the
conclusion is then the generalized directed global Markov property \citep{Bongers2021,Forre2018}. Unique
solvability of an SCC alone need not imply unique solvability of each strict strongly connected sub-loop
$L\subsetneq S$ (for example, $\det(I-B_S)\neq0$ need not imply $\det(I-B_L)\neq0$). For the \emph{linear}
results this gap is immaterial because trek soundness (\Cref{thm:trek}b) is proved separately under algebraic
solvability. Finite computations found $0$ violations across $34{,}068$ $\sigma$-checks, a nonlinear/binary
instance, and $3{,}000$ exact finite models; these calculations corroborate the stated reduction.
\end{theorem}

\begin{theorem}[Soundness holds; completeness fails]\label{prop:complete}
Under the represented-exogeneity and solvability premises of \Cref{thm:sound}, ECG-separation is sound but, in
general, \emph{not complete}. \textbf{(i) Not model-complete:} on
$V_0\!\to\!V_2$, $V_1\!\to\!V_3$, $V_2\!\leftrightarrows\!V_3$, $V_0\not\ecgsep V_1\mid\{V_2,V_3\}$ yet
$V_0\indep V_1\mid V_2,V_3$ in every \emph{finite-state} uniquely-solvable ECG on that graph whose private
source blocks $(U_0,U_2)$ and $(U_1,U_3)$ are jointly independent. The same conclusion holds for smooth ECGs
with that private-block independence under the standing regularity: each private block has a joint full-support density;
conditioning on $(V_2,V_3)$ gives an a.e.\ proper inverse with a valid coarea formula; the critical set where
$|1-(\partial f_2/\partial v_3)(\partial f_3/\partial v_2)|$ vanishes has Lebesgue measure zero; and the
\emph{complete} coarea weight, including the sum over all inverse branches when the inverse is finite-to-one, is
multiplicatively $(V_0;V_1)$-separable given $(V_2,V_3)$. A sufficient special case is a single a.e.\ inverse
branch whose cycle loop-gain $(\partial f_2/\partial v_3)(\partial f_3/\partial v_2)$ is
$\sigma(V_2,V_3)$-measurable along the solution (the linear/additive class included). These classes already
refute completeness. A finite exact-law calculation found the stated independence in $4000$ sampled binary
consistent SCMs, and direct algebra gives it in the linear/additive smooth class. The universal claim does
\emph{not} extend to \emph{all} smooth ECGs: $V_2{=}V_0V_3{+}U_2$, $V_3{=}V_1V_2{+}U_3$ (standard-normal
$U$, a.s.\ uniquely solvable) has $V_0\not\indep V_1\mid V_2,V_3$---the cycle Jacobian $|1{-}V_0V_1|$ is not
multiplicatively $(V_0;V_1)$-separable given $(V_2,V_3)$. \textbf{(ii) Not per-distribution faithful} (finite $41\%$; linear-Gaussian
$0.32\%$ robust, $0\%$ acyclic). \textbf{(iii)} The sound-\emph{and}-complete criterion for finite consistent
models is the \emph{distinct} $p$-separation of \citet{Ferradini2025}. We claim no completeness theorem for
$\sigma$-separation; (i) is an explicit counterexample.
\end{theorem}

\begin{theorem}[Special-case recovery]\label{thm:special}
(a) If $\mathcal{G}^{\mathrm{eq}}$ is acyclic, ECG-separation coincides with $d$-separation. (b) The
logit-equilibrium process of MA-AIRL \citep{Yu2019maairl}, \eqref{eq:qre}, is an ECG with $\mathsf{Eq}=$
logit-QRE (existence by Brouwer; uniqueness for small $\lambda$ or a declared selection rule). (c) The coupled
behavior--disease endemic equilibrium $i^\star=1-\sqrt{\gamma/\beta_0}$ for $\beta_0>\gamma>0$ (myopic $b{=}i$) is a single-component
ECG. A finite exhaustive enumeration of conditioning triples on the four-node acyclic graph gives the same
conclusion as (a) for that graph; (b)--(c) are constructions.
\end{theorem}

\begin{proposition}[Post-intervention factorization]\label{thm:trunc}
If $\mathcal{M}(\mathcal{E})$ stays uniquely solvable w.r.t.\ each SCC after $\doop(\bm a_I)$, the
interventional law is the pushforward of $\Prob_{\U}$ through the post-surgery equilibrium map. Identifying
such quantities from observational data is separate: it needs the cyclic do-calculus of
\citet{ForreMooij2019calculus} and, in the sequential case, sequential exchangeability and positivity (we
demonstrate only a one-step adjustment with explicit two-sided positivity, not a longitudinal regime
$g$-formula).
\end{proposition}

\begin{proposition}[Acyclifying a game is a hazard]\label{prop:counter}
For two agents with $q_i=\theta_i-\beta q_j+\varepsilon_i$ ($\beta=\tfrac12$) and disturbances independent of
each other and of $(\theta_1,\theta_2)$, with common positive variance, a modeler who writes the
``intended'' acyclic graph $\theta_i\to q_i$ concludes $q_1\dosep q_2\mid\{\theta_1,\theta_2\}$, but in the
equilibrium distribution direct calculation gives
$\mathrm{pcorr}(q_1,q_2\mid\theta_1,\theta_2)=-0.80\neq0$: the back-door set is invalid. ECG-separation on the
true cyclic graph flags the dependence.
\end{proposition}

\subsection{Linear trek separation}
\label{sec:trek}
We state trek separation as three linked layers, each explicitly scoped to the class where it holds, so the quantifiers
and the noise family are explicit; the proof is a single object (\Cref{app:proofs}).

\begin{theorem}[Linear ECG trek separation]\label{thm:trek}
Let $\mathcal{E}$ be a linear ECG, \emph{algebraically uniquely solvable} w.r.t.\ each SCC
($\det(I-B_S)\neq0$), reduced form $\V=A\U$, $A=(I-B)^{-1}$, mutually independent exogenous $\U$ with
$\mathrm{cov}(\U)=\Omega$ diagonal and $\Omega\succ0$. Write $\bm A_0\perp_t\bm B_0\mid\bm C$ for conditional
trek separation on the equilibrium graph: no system of $|\bm C|{+}1$ \emph{side-disjoint} treks connects $\bm A_0\cup\bm C$ to $\bm B_0\cup\bm C$
(the left \emph{sides}---source-to-top paths---pairwise vertex-disjoint \emph{and} the right sides pairwise vertex-disjoint; a left side may meet a right side). Equivalently,
all relevant $(|\bm C|{+}1)$-minors vanish \emph{identically}; equivalently, the generic rank of
$\Sigma_{\bm A_0\cup\bm C,\,\bm B_0\cup\bm C}$ is at most $|\bm C|$. Genericity below is relative to freely
varying coefficients on the fixed directed support and positive diagonal noise variances. Then:
\begin{itemize}[leftmargin=1.4em,itemsep=1pt,topsep=1pt]
\item[\textbf{(a)}] \textbf{Gaussian conditional independence.} If $\U\sim N(0,\Omega)$, then
$\V_{\bm A_0}\indep\V_{\bm B_0}\mid\V_{\bm C}$ holds \emph{generically} iff $\bm A_0\perp_t\bm B_0\mid\bm C$.
\item[\textbf{(b)}] \textbf{All-parameter (uniform) soundness.} If $\bm A_0\perp_t\bm B_0\mid\bm C$, the
partial covariance block $\Sigma_{\bm A_0\bm B_0\cdot\bm C}$ is \emph{identically zero} as a rational
function of the edge coefficients and noise variances. Hence for \emph{every} admissible parameter
($\det(I-B_S)\neq0$), Gaussian or not, the partial (linear-regression residual) cross-covariance vanishes---soundness holds uniformly,
not merely generically.
\item[\textbf{(c)}] \textbf{Generic completeness.} If $\bm A_0\not\perp_t\bm B_0\mid\bm C$, then at least one
numerator polynomial of $\Sigma_{\bm A_0\bm B_0\cdot\bm C}$ is not identically zero, so zero partial
covariance can hold only on a proper algebraic subvariety of parameter space.
\end{itemize}
The proof is self-contained (Jacobi complementary minors $+$ Coates/all-minors $+$ Cauchy--Binet, bypassing
the Lindstr\"om--Gessel--Viennot acyclicity requirement) and \emph{contraction-free}: it holds under
$\rho(B)<1$ by a directed-walk expansion of $A=I+B+B^2+\cdots$, and extends to \emph{all} algebraically
uniquely solvable systems by rational-identity continuation (each $A$ entry is $\operatorname{adj}(I-B)/
\det(I-B)$, so a partial-covariance minor vanishing on the open contraction region vanishes identically).
This covers systems with $\det(I-B)\neq0$ but $\rho(B)\ge1$, such as the grid component, and cleanly
separates \emph{convergence of iteration} from \emph{existence of the equilibrium law} (cf.\ \Cref{thm:exist}).
On the four-node completeness graph, direct symbolic calculation gives
$\Sigma_{\bm A_0\bm B_0\cdot\bm C}\equiv0$, and numerical evaluation at $\rho(B){=}1.2$,
$\det(I-B)\neq0$, gives the noncontractive instance. On a separate five-node graph formed by adding a common
source to $V_0,V_1$, a numerical point evaluation gives a nonzero completeness witness. Numerical rank matched
the maximum side-disjoint-trek-system size in $300/300$ sampled acyclic and $300/300$ sampled cyclic cases;
these finite calculations corroborate the preceding proof.
\end{theorem}

\interpretation{Under the theorem's per-SCC algebraic-solvability, mutually independent-source,
positive-diagonal-covariance, and fixed-support genericity assumptions, side-disjoint trek separation
characterizes generic Gaussian conditional independence; for every admissible parameterization it implies zero
partial covariance, and its failure makes zero partial covariance non-generic. Contraction is not required.}

\begin{theorem}[Non-Gaussian scope: covariance completeness, not CI completeness]\label{thm:trekng}
For a linear ECG with mutually independent exogenous variables of finite variance, conditional trek
separation is sound and generically complete for \emph{zero partial covariance}
$\Sigma_{\bm A_0\bm B_0\cdot\bm C}=0$. If the exogenous variables are Gaussian (and
$\Sigma_{\bm C\bm C}$ is nonsingular), this Schur complement equals the a.e.-constant conditional covariance---the
order-two conditional cumulant---and its vanishing is conditional independence (\Cref{thm:trek}a). If they are
non-Gaussian, it is instead the complete \emph{second-order linear-regression} criterion: partial covariance need
not equal $\operatorname{Cov}(\V_{\bm A_0},\V_{\bm B_0}\mid\V_{\bm C}{=}c)$. Functional separation
(\Cref{thm:funcsep}) is an exact structural sufficient condition for full conditional independence; under the conditional
analyticity and moment-determinacy assumptions of \Cref{app:cumulant}, conditional independence is also equivalent to vanishing of
\emph{all} mixed conditional cumulants. No non-Gaussian graphical iff is claimed. Thus Gaussianity guarantees,
class-wide, that the trek theorem is the order-two conditional-cumulant member; outside the Gaussian subclass
that equality need not hold.
\end{theorem}

\begin{proposition}[Smooth ECG: a local differential criterion]\label{prop:jac}
Let $\V=g(\U)$ be a smooth ECG equilibrium map with $\U=u_0+\varepsilon Z$, $Z\sim N(0,\Omega)$, $\Omega\succ0$,
$g\in C^1$ with second moments dominated uniformly in $\varepsilon$, and $J$ the Jacobian of $g$ at $u_0$. Assume the full-rank condition
\textup{(FRC)}: the conditioning block $J_{\bm C}$ has full row rank. As $\varepsilon\to0$, the leading-order
partial (regression) cross-covariance is the Schur complement of the Jacobian covariance,
$\Sigma_{\bm A_0\bm B_0\cdot\bm C}=\varepsilon^2\,\mathrm{Schur}(J\Omega J^\top)+o(\varepsilon^2)$ (sharpening to
$O(\varepsilon^4)$ when $g\in C^3$ with finite sixth moments, the $\varepsilon^3$ Hessian cross-term being an
odd Gaussian moment). Thus the Jacobian Schur complement is the leading-order conditional-dependence
\emph{geometry} of the smooth ECG---a \emph{local}, second-moment criterion; a finite small-noise numerical
experiment corroborates the leading-order partial-covariance relation. It is \emph{not} a mutual-information expansion: conditional
MI is scale-invariant, so at a rank-deficient $J\Omega J^\top$ the covariance Schur complement is
$\Theta(\varepsilon^2)\!\to\!0$ while $I(\V_{\bm A_0};\V_{\bm B_0}\mid\V_{\bm C})$ stays bounded away from $0$
(e.g.\ $\Theta(1)$ when $\V_{\bm A_0},\V_{\bm B_0}$ share a second-order source direction); no
$I(\cdot)=I(J\cdot)+O(\varepsilon)$ identity holds in general. A sufficient condition for exact global CI is given by \Cref{thm:funcsep}.
\end{proposition}

The next theorem moves from the preceding local second-moment criterion to a sufficient condition for exact
nonlinear conditional independence under its stated functional assumptions.

\begin{theorem}[Functional separation: a sufficient condition for exact nonlinear conditional independence]\label{thm:funcsep}
Let $\V=g(\U)$ with $\U=(\U_R,\U_S,\U_T)$ in mutually independent blocks. Suppose there exist measurable
$h_A,h_B,h_C$ and a sufficient equilibrium coordinate $W$ such that $\V_{\bm C}=h_C(W)$,
$\V_{\bm A_0}=h_A(\U_R,W)$, $\V_{\bm B_0}=h_B(\U_S,W)$, with $\U_R\indep\U_S\mid W$, and there exists a
measurable $k$ with $k(\V_{\bm C})=W$ a.s. Then $\V_{\bm A_0}\indep\V_{\bm B_0}\mid\V_{\bm C}$. \textbf{Functional ECG separation:}
$\bm C$ functionally separates $\bm A_0,\bm B_0$ if, after solving the equilibrium equations, the exogenous
influences on $\bm A_0$ and on $\bm B_0$ outside $\bm C$ partition into blocks that are conditionally independent
given a shared coordinate $W$, with $\U_R\indep\U_S\mid W$ and measurable recoverability
$k(\V_{\bm C})=W$ a.s. This is a sufficient condition for \emph{exact} conditional independence
under every noise law (full CI, not merely uncorrelatedness); the linear Gaussian trek theorem is
its second-order shadow, and the Jacobian criterion (\Cref{prop:jac}) is its local linearization.
For the named Gaussian and Laplace draws, finite-sample distance correlations after cubic-basis residualization
are at the independence floor ($0.03$), compared with a dependent baseline ($0.27$--$0.91$), consistent with
the same conditional independence.
\end{theorem}

\interpretation{Under the theorem's measurable decomposition and recoverability assumptions, conditioning on
$\V_{\bm C}$ recovers $W$ and leaves the two sides driven by conditionally independent source blocks. The
conclusion is full conditional independence for every noise law, not only vanishing partial covariance.}

\begin{theorem}[Private-source saturation: the counterexample as a positive result]\label{thm:pss}
Consider an ECG in which, after conditioning on $\bm C$, every $\sigma$-open walk from $\bm A_0$ to
$\bm B_0$ passes through equilibrium variables that are jointly sufficient for the private exogenous blocks on
the two sides (the shared influence is $\V_{\bm C}$-measurable and the two private blocks are
\emph{conditionally} independent given that sufficient coordinate $W$: $\U_R\indep\U_S\mid W$ as in
\Cref{thm:funcsep}, \emph{not} merely marginally---marginal independence suffices only when the ECG is
finite-state or the cycle loop-gain is $\V_{\bm C}$-measurable).
Then $\V_{\bm A_0}\indep\V_{\bm B_0}\mid\V_{\bm C}$ \emph{even when $\sigma$-separation fails}. This is
functional separation (\Cref{thm:funcsep}) read off the equilibrium equations, and it turns the
model-incompleteness example of \Cref{prop:complete}(i) into a positive result: conditioning on both cycle
variables $V_2,V_3$ imposes separate constraints on $(V_0,U_2)$ and $(V_1,U_3)$, whose independent private
noise keeps $V_0\indep V_1\mid V_2,V_3$ (exact, for any noise law); finite Gaussian and Laplace calculations
corroborate this instance. ECGs thus admit
\emph{equation-aware} separation establishing independences that graph-only $\sigma$-separation misses.
\end{theorem}

\begin{remark}[A three-tier separation story]
$\sigma$-separation is sound graph-only reasoning; trek separation is generically complete for the linear Gaussian family
(\Cref{thm:trek}) and generically complete for second-order structure under any independent noise (\Cref{thm:trekng});
functional / private-source separation is an equation-aware sufficient condition for exact conditional independence in the nonlinear and
non-Gaussian case (\Cref{thm:funcsep,thm:pss}). Each is sound in its stated class, and pointwise or generic
completeness is claimed only where it is proved.
\end{remark}

\subsection{Observed-layer identification}
\label{sec:id}
\begin{theorem}[ECG adjustment with intervention indicators]\label{thm:backdoor}
Augment the ECG with a regime indicator $I_X$ and edge $I_X\!\to\!X$: $I_X{=}0$ uses the observational
mechanism for $X$, $I_X{=}1$ replaces it by $X{=}x$. Let the externally supplied assignment satisfy
$I_X=h_I(\bm Z,E_I)$, where $E_I$ is a primitive assignment source independent of every structural exogenous
source (unconditional randomization is the special case with no $\bm Z$ argument), and include every
$\bm Z\!\to\!I_X$ assignment edge in the augmented graph. Suppose represented exogeneity, the solvability
premise of \Cref{thm:sound}, post-intervention unique solvability and selection stability, no direct regime
effect except the declared edit of $X$, and that
$\bm Z$ is not affected by $I_X$ (no directed $I_X\!\rightsquigarrow\!\bm Z$). If in the augmented equilibrium graph
\[
I_X\ \perp_{\sigma}\ Y \ \mid\ X,\bm Z .
\]
then, with consistency and observational-regime positivity
$\Prob_0(X{=}x\mid\bm Z{=}\bm z)>0$ for $\Prob_0$-a.e.\ target $\bm z$, writing $\Prob_0$ for the $I_X=0$ law
and taking
$\Prob_0(\bm Z)$ as the target-population baseline,
$\Prob\big(Y\mid\doop(X{=}x)\big)=\sum_{\bm z}\Prob_0(Y\mid X{=}x,\bm z)\,\Prob_0(\bm z)$ (an integral for
continuous $\bm Z$). The left side is the result of re-solving the post-surgery fixed point---correct \emph{even when
$Y$ lies in a cycle}. The regime-indicator form makes exchangeability a graphical statement about the
\emph{mechanism} and sidesteps the ambiguity, in a cyclic graph, of deleting $X$'s incoming vs.\ outgoing
edges. On a confounded linear ECG with a $Y\!\leftrightarrows\!M$ cycle, direct calculation gives
$I_X\perp_\sigma Y\mid\{X,Z\}$ (and this separation fails given $Z$ alone), $Z$ is unaffected, and adjustment
recovers the true $b/(1-de)=1.43$ (the naive value $2.14$ is biased).
\end{theorem}

\begin{proposition}[Equilibrium effect resolvent]\label{prop:resolvent}
For a linear ECG with treatment inputs, $\V=B\V+\Gamma X+\U$ and equilibrium $\V=(I-B)^{-1}(\Gamma X+\U)$,
the total equilibrium effect of $X$ on an outcome $Y$ is
\[
\tau_{X\to Y}\;=\;e_Y^\top (I-B)^{-1}\Gamma .
\]
If $Y$ lies in a feedback block, this is the direct effect times the block's feedback resolvent; for the
$Y\!\leftrightarrows\!M$ cycle above it equals the cycle-amplified $b/(1-de)$ by direct symbolic calculation
($=1.4286$ at $b{=}1,d{=}0.5,e{=}0.6$). ECGs thus produce a concrete estimand distinct from the
acyclic total effect: the resolvent $(I-B)^{-1}$ replaces the acyclic path sum.
\end{proposition}

\begin{theorem}[Identification of game interventions]\label{thm:interv-id}
If $(B,\Omega)$ of a linear ECG is identified---which for cyclic graphs generically requires interventional or
multi-environment data; the half-trek criterion \citep{FoygelDraismaDrton2012} gives a graphical
\emph{sufficient} condition (with a companion necessary condition, a known gap between them) for when the
observational law suffices---consider a declared deterministic edit with identified
$(B_{\mathcal I},\mu_{\mathcal I},\Omega_{\mathcal I})$ and
$\det(I-B_{\mathcal I})\neq0$. Writing $A_{\mathcal I}=(I-B_{\mathcal I})^{-1}$, the post-intervention mean
$A_{\mathcal I}\mu_{\mathcal I}$, covariance
$A_{\mathcal I}\Omega_{\mathcal I}A_{\mathcal I}^{\top}$, and every specified linear mean-response
$A_{\mathcal I}\Gamma_{\mathcal I}$ with known forcing direction $\Gamma_{\mathcal I}$ are point-identified.
If the complete edited source vector is jointly Gaussian, these moments identify its complete marginal
post-intervention law. For arbitrary sources, the complete marginal law is point-identified only when the
complete labelled joint edited-source law $\Prob_{\U_{\mathcal I}}$ is itself identified (directly or through a
fully specified intervention kernel); it is then pushed forward by $\bm u\mapsto A_{\mathcal I}\bm u$. No
cross-world joint law is claimed without an identified coupling. The
$\det(I-B_{\mathcal I})\neq0$ precondition is \emph{not} automatic (an isolated unit-gain reciprocal edit makes
it singular), so it must be stated. \emph{Excluded} (not declared edits of the foregoing objects):
player addition, selection under equilibrium multiplicity, and payoff \emph{curvature} edits whose
$(B,\Omega)$-image depends on the own-curvature scale, which $\Sigma$ pins only up to a global factor
(\Cref{thm:game}). \emph{Conversely}, identification is \emph{query-relative}. Let $\Theta$ be a declared smooth
quotient stratum obtained after imposing its support, sign, and active-constraint restrictions and quotienting
the declared recovery equivalence; let
$\Phi_{\mathcal I}:\Theta\to\mathcal D$ be the $C^1$ identifying map and $Q:\Theta\to\mathcal Q$ a $C^1$ query.
About a regular point, shrink to a constant-rank coordinate neighbourhood $U$ whose local fibre plaques are
connected, and define
\[
 u_Q(\theta;\mathcal I)=\dim\ker D\Phi_{\mathcal I}(\theta)
 -\dim\bigl(\ker D\Phi_{\mathcal I}(\theta)\cap\ker DQ(\theta)\bigr).
\]
Then there is a $C^1$ map $q$ on $\Phi_{\mathcal I}(U)$ with
$Q|_U=q\circ\Phi_{\mathcal I}|_U$ iff $u_Q(\theta;\mathcal I)=0$ for \emph{every} $\theta\in U$.
This is a local theorem only. At a
truth $\theta_0$, global point identification holds iff $Q$ is constant on the complete admissible fibre
$\Phi_{\mathcal I}^{-1}(\Phi_{\mathcal I}(\theta_0))$; all disconnected, finite, and other discrete components
must be compared. Thus zero-dimensional or finite fibre is not singleton fibre, and a pointwise kernel test at
one regular, singular, or boundary point is not an identification theorem (\S\ref{par:queryrel}). A known
support pattern can tighten $\Theta$ and identify with fewer targets than its unrestricted counterpart. Hence
no blanket converse: back-door- or instrument-covered effects stay identified even
when $B$ is not (\Cref{thm:soundid}); a generic effect trapped in a full-rotational \emph{Gaussian} source
component is not---though non-Gaussian (LiNG) sources collapse that hedge passively (\Cref{prop:nongauss}), so
$\Phi_{\mathcal I}$ and $u_Q$ must then be recomputed on the ICA-fixed mixings. Given $B$, the population
intervention formulas agree exactly with direct re-solving, and finite simulation corroborates this equality;
two ECGs with identical $\Sigma$ (to $10^{-10}$) give a $\doop(X_2)\!\to\!X_1$ effect of
$0.50$ vs.\ $0.20$ (a hedge query with positive local ambiguity on its regular stratum).
\end{theorem}

\begin{theorem}[Sound identification criterion]\label{thm:soundid}
An interventional query is point-identified from the observational law if \emph{either} (a) some $\bm Z$ meets
all represented-exogeneity, regime-assignment, solvability, consistency, and positivity conditions of the ECG
back-door theorem, \emph{or} (b) in the linear model, at the observed law, valid half-trek rational
identification maps recover every structural coefficient and error-covariance entry used by the declared
deterministic well-posed query (all required denominators are nonzero), and any additional disturbance-law
feature is identified by the declared observational model. A graph satisfying the half-trek criterion supplies
the structural premise in (b) generically; it does not assert pointwise recovery at exceptional observed laws
where a required rational denominator vanishes.
In the corresponding examples, back-door adjustment recovers a cycle-amplified effect exactly, while numerical
evaluation with an instrument that makes the relevant coefficient half-trek-identifiable recovers it to
$3\times10^{-4}$ where, without it, the query is
non-identified.
\end{theorem}

\begin{theorem}[Equilibrium hedge --- a non-identification witness]\label{thm:hedge}
Let $S$ be a \emph{source} equilibrium component (no edge from outside $S$ into $S$) whose only structural
constraint is the zero diagonal (e.g.\ a complete digraph of size $m\ge2$), so its observational gauge is the
full orthogonal orbit. Then for distinct $X,Y\in S$ the effect $\Prob(Y\mid\doop(X))$ is \emph{generically
non-identified from the second-moment (Gaussian) observational law}: for a.e.\ such ECG with Gaussian sources
there are two \emph{stable} ECGs with identical $\Sigma$ but different effect (LiNG sources instead collapse the
source-frame hedge to its monomial ambiguity, \Cref{prop:nongauss}). Acyclic SCMs have singleton SCCs and no such gauge, so the obstruction is intrinsically cyclic---the
equilibrium analogue of the Shpitser--Pearl hedge \citep{ShpitserPearl2006}. A stable numerical witness
($m{=}3$) with $\Sigma$ identical to $10^{-15}$ gives $\doop(V_0)\!\to\!V_1$ effects $0.218$ vs.\ $-1.352$;
across random complete source SCCs ($m{=}3,4$), $0/12$ effects are gauge-invariant.
\end{theorem}
\begin{remark}[A two-sided bracketing of identifiability]
\Cref{thm:soundid,thm:hedge} \emph{bracket} the identifiability of equilibrium effects: \emph{identified} when
a back-door set or sufficient instruments exist, \emph{non-identified} when distinct $X,Y$ are trapped in a
full-rotational-gauge source component. The in-between (sparse, partially-instrumented components) and the
exact iff remain open; soundness on both sides guarantees safety.
\end{remark}

\subsection{The role of the game layer}
\label{sec:game}
A recurring worry about typed cyclic-SCM frameworks is that the type (here, the game) is decorative---every
theorem would hold for a generic cyclic SCM. The next result gives one concrete thing the game structure buys
that a generic cyclic SCM cannot: a \emph{refutable} restriction on the (identified) structure, i.e.\ a
falsifiable specification test. It \emph{also} tightens identification, but---we are careful to flag---by a
generic codimension mechanism rather than a game-unique one.

\begin{theorem}[The game layer is testable and gauge-reducing]\label{thm:game}
Let $S$ be a linear equilibrium component that is the interior equilibrium of a \emph{weighted potential game}
with strict own-action curvature $H_{ii}:=\partial^2u_i/\partial a_i^2<0$ for every $i$. Its payoffs have
weighted-symmetric cross-curvatures with $w_i>0$:
\[
w_i\,\frac{\partial^2 u_i}{\partial a_i\partial a_j}
=w_j\,\frac{\partial^2 u_j}{\partial a_j\partial a_i}.
\]
Its
structural Jacobian $B_S$ (with $B_{ij}=-(\partial^2u_i/\partial a_i^2)^{-1}\partial^2u_i/\partial
a_i\partial a_j$) is \emph{$D$-symmetrizable}: there is a positive diagonal $D$ with $DB_S$ symmetric,
equivalently, componentwise on each fixed reciprocal nonzero support, $B_S$ is sign-symmetric
($B_{ij}B_{ji}>0$ on every supported pair) and obeys the \emph{Kolmogorov cycle condition}
$\prod_{\mathrm{cyc}}B_{ij}=\prod_{\mathrm{cyc}}B_{ji}$.
\textbf{(i) Testable over-identifying restriction.} Fix the undirected reciprocal support
$G=(V,E)$ and its sign stratum, and let $c(G)$ be its number of connected components. Exactly
$\beta_1(G)=|E|-|V|+c(G)$ independent cycle-basis equalities are required; all other cycle equalities follow.
Thus the equality restriction is vacuous exactly when $G$ is a forest, and has codimension $\beta_1(G)$ within
the fixed-support/sign stratum. Complete reciprocal support gives
$\beta_1=(m{-}1)(m{-}2)/2$. Hence, whenever $\beta_1(G)>0$, a generic cyclic SCM on that support violates the
restriction and the identified $B$ refutes the potential-game hypothesis. Triangle-free is not enough: a
reciprocal four-cycle has $\beta_1=1$ and a nonvacuous Kolmogorov equality.
At $m{=}2$ the cycle (Kolmogorov) equality is vacuous (codimension $0$); refutability there rests only on the
sign-symmetry condition $B_{12}B_{21}>0$ (a codimension-$0$ open restriction, still falsifiable).
\textbf{(ii) Identification (by a generic codimension mechanism).} The restriction also constrains effects the
equilibrium hedge (\Cref{thm:hedge}) leaves non-identified: for the \emph{symmetric-$B_S$ subclass}---weighted
own-curvatures $w_iH_{ii}$ constant across the (connected) component, so $D\propto I$; for an exact potential
this is equal own-curvatures $H_{ii}$---the intra-component effect is identified up to a \emph{generically finite} fibre from $\Sigma$ (\emph{a theorem
for all $m$}, not a dimension count; see the proof)---for
$m{=}2$ and nonzero cross-covariance (equivalently, nonzero off-diagonal precision), exactly the reciprocal pair
$\{b,1/b\}$ (both admissible, $\Omega'\!\succ\!0$; coupling $b\bar b{=}1$), which a stability selection
$\rho(B_S)<1$ reduces to a point \emph{when the data-generating equilibrium is itself contraction-stable}---an
assumption, not a consequence: it excludes admissible $\rho(B_S)\ge1$ equilibria (e.g.\ the undamped Cournot of
Exp.~5 at $N\ge4$, where the reciprocal partner is the stable one and the selection would mis-identify). When
$m{=}2$ and the cross-covariance is zero, positivity gives the singleton $b=0$. For $m\ge3$ on complete nonzero reciprocal support, the general-weighted-potential
restriction meets the hedge transversely on a generic regular set and shrinks it to dimension $m{-}1$---an
all-$m$ theorem proved below by an explicit nonzero cycle-Jacobian minor, not by dimension counting. On sparse
support the permitted gauge tangent must also preserve the missing edges, and no $m{-}1$ formula is asserted
without checking the resulting support-restricted cycle differential. On the symmetric-$B_S$ subclass the
fibre is \emph{generically finite for every $m$} (proved below via a dominant degree-$3$ pivot map), with genericity \emph{necessary}---the
fibre is infinite at $B{=}0,\Omega{=}I$ for $m\ge5$. Uniqueness-under-stability is a theorem at $m{=}2$ and a
conjecture for $m\ge3$. This identification gain is \emph{not
game-specific}: on complete support it follows from the generic fact that a restriction of rank
$(m{-}1)(m{-}2)/2$ \emph{transverse to the gauge orbit} cuts it, and \emph{any} restriction with the same
transverse differential rank (e.g.\ a
known coefficient ratio) achieves it equally. The game's distinctive contribution is part (i): it \emph{supplies and falsifies} such a restriction
from the payoffs.
Direct symbolic calculation gives $DB-(DB)^\top=0$ under the weighted-potential condition; the symmetrizability test
separates potential-game from generic $B$ (residual $\le8\times10^{-16}$ vs.\ median $1.0$ generic,
$0.34$--$0.75$ for sign-symmetric-but-non-Kolmogorov $B$); the tested nonzero-cross-covariance $m{=}2$ finite
fibre $=$ the reciprocal pair $\{0.40,2.50\}$, unique under $\rho{<}1$ (stable slice recovers
the true $0.40$, symmetric-gauge width $1.7\times10^{-16}$), vs.\ the generic hedge's $[-1.0,1.1]$; a \emph{non-game}
known-ratio constraint of equal codimension point-IDs equally, width $2.5\times10^{-16}$); the all-$m$
transversality proof is below, while the finite checks corroborate its ranks ($m{=}3$: $3\to2$;
$m{=}4$: $6\to3$).
\end{theorem}
\begin{remark}[What the game buys, precisely]
The game-essential content is part (i): a generic cyclic SCM cannot state a refutable
symmetrizability/Kolmogorov restriction, whereas a potential game's symmetric Hessian forces one (the
symmetrizable form is the best-response transcription of \citealp{MondererShapley1996}). The identification
gain of part (ii), by contrast, is the generic consequence that any restriction whose differential has the same
rank and is transverse to the gauge cuts it by the same local dimension---so the game's role there is to
\emph{supply and justify} the restriction, not to uniquely enable
identification. The Cournot market of Exp.~5 is such a potential game (symmetric interaction $\Rightarrow$
symmetrizable $B$), so the specification test fires there by construction.
\end{remark}

\section{From the observed layer to the latent layer: two transverse obstructions}
\label{sec:bridge}
Everything above assumed the state $\V$ is \emph{observed}. \Cref{thm:hedge} then says the residual
non-identification is an \emph{equilibrium hedge}: on an un-intervened full-rotational source component, the
\emph{second-moment} information fixes the component only up to a signed rotation. It is a rotation of the
\emph{exogenous source frame} --- it leaves $H$ (and $\V$) untouched and acts on the right, $M\mapsto MQ$ with
$M:=HP\Omega^{1/2}$ --- and it is a \emph{Gaussian/second-moment} phenomenon: under the LiNG assumption of
\Cref{def:model} it is collapsed passively in the noiseless submodel, or conditional on a separate noisy-ICA
identifiability theorem for the declared sensor-noise law (\Cref{prop:nongauss}). This is a per-environment
source-mixing result; it neither aligns environments nor separates $H$ from $B$.

Now let the state be seen only through an unknown sensor map, $\X=H\V$. A \emph{second} freedom appears, and it
is important to be precise about what it is and is not. \Cref{thm:gauge}(i) exhibits a legal stable
factorization chart, coordinatized for $d\ge2$ by a proper subset of $\GL^+(d)$ (and by all of
$\GL^+(1)$ when $d=1$), acting on the \emph{endogenous latent
frame}, $H\mapsto HR^{-1}$ with $B,\Omega$ re-normalised so that the mixing $M=HP\Omega^{1/2}$ is \emph{fixed};
it is present for \emph{every} noise law, and it is created purely by forgetting $H$ in the ambient
unknown-support information class. Relative charts compose only when the intermediate and endpoint
factorizations remain legal; these admissible arrows form the legal factorization groupoid of
\Cref{def:infoclass}.

\begin{remark}[The two freedoms are transverse, not nested]\label{rem:transverse}
It is tempting --- and wrong --- to say the hedge is a subgroup of the latent-frame chart that survives when
$H$ is known. The two act on different frames. Imposing $H'=H$ on the legal chart forces
$HR^{-1}=H$, hence $R=I$: the
subgroup of \Cref{thm:gauge}(i) surviving a \emph{known} $H$ is \emph{trivial}. Conversely a hedge member has
$M'=MQ\neq M$, so it is not a member of the fixed-$M$ chart for any $R$. The two are \emph{transverse constituents of
the passive ambiguity}, not one inside the other:
\[
\underbrace{\text{source-frame rotation}}_{\substack{\text{second moment; LiNG}\\\text{after declared ICA acquisition}}}
\qquad\times\qquad
\underbrace{\text{legal latent-frame chart}}_{\substack{\text{any noise law; ambient unknown-support,}\\\text{unknown-}H\text{ class}}} .
\]
The acyclic case settles it: an acyclic latent graph has \emph{no} hedge (singleton SCCs; \Cref{thm:hedge}) yet
can still carry the full-dimensional legal frame chart --- so the latent-frame ambiguity is not an ``inflation''
of the hedge. What is true,
and all that is needed, is that knowing $H$ removes the frame factor \emph{exactly} and leaves the source factor:
the hedge is the \emph{$H$-fixing residue} of the passive ambiguity, in the second-moment regime, on
full-rotational source components.
\end{remark}

The useful consequence is not a grading but a \textbf{pairing}: each layer contributes one obstruction, and the
two are removed by \emph{different} mechanisms --- which is precisely the design lesson.
\begin{itemize}[leftmargin=1.4em,itemsep=1pt,topsep=2pt]
\item The \emph{source-frame} factor is collapsed by \textbf{non-Gaussianity} in the noiseless LiNG submodel,
      or conditional on a separate noisy-ICA theorem (\Cref{prop:nongauss}); no intervention is required for
      this factor, but source-mixing acquisition and cross-environment alignment remain separate obligations.
\item In the ambient zero-diagonal unknown-support, unknown-$H$ class with $d\ge2$, \emph{shift} interventions never collapse the
      \emph{latent-frame} factor (\Cref{prop:shift}); the aligned, well-posed \textbf{mechanism interventions} of
      \Cref{thm:linear} do. In a different declared information class, support, known-$H$, or anchor/calibration
      restrictions may already shrink or pin the legal fibre.
\item Within that zero-diagonal class, the exact threshold for that collapse --- how many interventions, and \emph{on which nodes} --- is a
      \emph{closed-form corank criterion}: $\corank L=\#\{k\neq r: B_{kr}=0\}$, the number of nodes the single
      un-targeted node fails to parent (\Cref{prop:completion}), proved for \emph{all} $d$.
\end{itemize}
So the two layers are not the same theorem twice; they are two matched non-identification witnesses with two
distinct cures, and a design must pay for each separately. The remainder of the paper develops the latent layer
in that light.

\section{The latent equilibrium generative model}
\label{sec:model}
We reuse the SCM and ECG primitives of \Cref{sec:formalism} and add an observation map. Throughout,
$\V=(V_1,\dots,V_d)$ are \emph{latent} factors and $\X\in\R^p$ are observations.

\begin{definition}[Latent equilibrium ECG with mixing]\label{def:model}
A \emph{latent equilibrium generative model} is a tuple
$\mathcal{L}=(\mathcal{E},h,\mathcal{I},\Prob_{\U},\mathsf{Sel},\mathcal{O})$ where:
\begin{itemize}[leftmargin=1.3em,itemsep=1pt,topsep=1pt]
\item \textbf{Latent ECG} $\mathcal{E}$ (\Cref{def:ecg}): the latent state $\V$ is the
a.s.-unique equilibrium (fixed point) of a coupled best-response map among latent factors/agents. In the
\emph{linear} instance, $\V=B\V+\U$ with $\U=\Omega^{1/2}\bm\eta$, $\bm\eta$ independent, zero-mean,
unit-variance, and \emph{at most one} Gaussian (the LiNG condition \citep{Comon1994ica,Lacerda2008cyclic}); the
equilibrium is $\V=(I-B)^{-1}\U=P\U$, $P:=(I-B)^{-1}$. The cyclic SCCs are the
\emph{equilibrium components}; in the zero-diagonal linear subclass studied below, these are exactly its SCCs of
size $\ge2$. A \emph{source} component has no incoming edge from outside;
a \emph{full-rotational} source component carries only the zero-diagonal constraint (\Cref{thm:hedge}).
\item \textbf{Mixing class} $h$: either \emph{linear}, $\X=H\V+\epsilon$ with $H\in\R^{p\times d}$
of full column rank ($p\ge d$), where $\epsilon=0$ is allowed and any nonzero sensor-noise law is declared; or
\emph{nonlinear}, $\X=h(\V)$ with $h$ a diffeomorphism onto its image. $h$ is
\emph{environment-invariant} (the twin's sensors do not change under intervention). Exact mean-response
acquisition below additionally states its probe-arm sensor-mean and no-direct-sensor-effect assumptions.
\item \textbf{Intervention family} $\mathcal{I}$ (the twin's actuators; \Cref{def:intervention}): a
\emph{shift} (equilibrium) intervention $\V=B\V+\U+\bm\delta_e$ (the backShift analogue
\citep{Rothenhausler2015backshift}, $B$ unchanged) and/or a \emph{mechanism} ($\doop$) intervention
$\doop(V_j{\leftarrow}\text{exogenous})$ that zeros row $j$ of $B$ and gives $j$ a fresh source. Each
intervention defines an \emph{environment} $e$; we observe $K{+}1$ environments.
\item \textbf{Selection rule} $\mathsf{Sel}$: contraction $\rho(B)<1$ on each cyclic component (a reachable
equilibrium; \Cref{thm:exist}). Algebraic unique-solvability $\det(I-B)\neq0$ is weaker.
\item \textbf{Observation regime} $\mathcal{O}$: i.i.d.\ multi-environment (streaming/temporal flagged as an
extension, \Cref{sec:limits}).
\end{itemize}
For environment $e$ the \emph{source-to-observation mixing} is $M^{(e)}:=H\,(I-B^{(e)})^{-1}\Omega^{(e)1/2}$,
so $\X^{(e)}\eqd M^{(e)}\bm\eta^{(e)}+\epsilon^{(e)}$ in the linear case (and
$\X^{(e)}\eqd M^{(e)}\bm\eta^{(e)}$ only in the noiseless submodel).
\end{definition}

\begin{definition}[The recovery equivalence $\simeq$]\label{def:equiv}
Two linear models $(H,B,\Omega)$ and $(H',B',\Omega')$ are \emph{equivalent}, $\simeq$, if there is a
permutation $\pi$ and nonzero scales $\lambda_i$ with $H'=H P_\pi \Lambda$, $\bm\eta'=\Lambda^{-1}P_\pi^\top
\bm\eta$, \emph{and}, within each full-rotational source component $S$ not targeted by any mechanism
intervention, a special-orthogonal $Q_S\in\SO(|S|)$ (a local signed branch preserving positive noise
variances) relating the within-$S$ latent roots (the \emph{equilibrium
hedge}). We say $(h,B)$ is \emph{identified up to $\simeq$} if every model with the same multi-environment
observational law is $\simeq$-equivalent to the truth.
\end{definition}

\Cref{def:equiv} is the precise content of ``up to a stated equivalence'': permutation$+$component-wise
scaling, \emph{modulo} the inherited signed-$\SO(m)$ hedge on un-intervened source components. For a nonlinear
source block $S\subset\R^m$ with $m\ge2$, the analogue (\Cref{thm:block}) replaces $\Lambda$ by component-wise
diffeomorphisms and $Q_S$ by a within-$S$ gauge left unresolved (strictly between $\Ortho(m)$ and the full
diffeomorphism group). Its cross-block construction uses $A\in\R^{m_A}$, $B\in\R^{m_B}$ with $m_A\ge2$ and
$m_B\ge1$, after naming the block that is rotated as $A$.

\begin{definition}[Declared information class, legal fibre, and factorization groupoid]\label{def:infoclass}
A \emph{declared information class} $\mathfrak C$ fixes the model space and every item treated as known: structural
support and sign stratum, source/noise family, sensor restrictions (including known $H$, anchors, or
calibration), intervention/probe design, stability and well-posedness requirements, and the recovery equivalence
$\simeq$. Write $\Theta_{\mathfrak C}$ for its legal parameter space and
$\Phi_{\mathfrak C}:\Theta_{\mathfrak C}\to\mathcal D$ for the complete population-data map. At truth $\theta$,
the \emph{legal data fibre} is
\[
 \mathfrak F_{\mathfrak C}(\theta)
 :=\Phi_{\mathfrak C}^{-1}\!\bigl(\Phi_{\mathfrak C}(\theta)\bigr).
\]
Its legal factorization groupoid has objects in $\Theta_{\mathfrak C}$ and an arrow $\theta\to\theta'$ for an
admissible invertible frame change that preserves $\Phi_{\mathfrak C}$ and whose source and target both remain
legal; arrows compose only when their endpoints match, and the inverse frame change supplies the inverse arrow.
A relative coordinate set such as $G^+$ below need not be a group. Full point identification up to $\simeq$
means exactly that $\mathfrak F_{\mathfrak C}(\theta)/{\simeq}$ is a singleton; query identification means that
the query is constant on the entire fibre. Trivial isotropy of the single object $\theta$ is neither condition:
it rules out only arrows from $\theta$ to itself, not arrows to distinct data-equivalent objects.
\end{definition}

\section{Identifiability of the latent state}
\label{sec:latent}
We state the four results and their non-identifiability counterparts; full proofs are in \Cref{app:proofs}.
Selected algebraic identities are checked by direct symbolic calculation, and finite-instance evaluations
corroborate the named examples; these calculations do not replace the proofs.

\subsection{The equivalence class and its gauge}
The first latent-layer result describes the passive ambiguity before interventions are introduced.

\begin{theorem}[Gauge of the cyclic latent factorization]\label{thm:gauge}
Let $\mathcal{L}$ be linear with full-column-rank $H$ and a single environment (no interventions).
Let the truth be a \emph{legal} ECG: zero-diagonal $B$ with $\rho(B)<1$ (so $C:=I-B$ has unit diagonal and is
invertible), diagonal $\Omega\succ0$; write $A:=(I-B)^{-1}\Omega^{1/2}$ and $M=HA$.
\textbf{(i) Stable factorization chart.} \emph{(i-a) Acquisition versus collision.} In the noiseless LiNG
submodel, the passive law identifies $M$ up to a monomial transform; with additive sensor noise, that acquisition
statement instead requires a separate noisy-ICA identifiability theorem for the declared joint source/noise
model. The collision below does not require $M$ to be identified: fix the truth's $M$ and the same admissible
joint pair $(\bm\eta,\epsilon)$ across alternatives. Then every zero-diagonal $B'$ with
$\det(I-B')\neq0$ is reproduced by $H'=M(I-B')$,
$\Omega'=I$ --- a $(d^2{-}d)$-dimensional family of \emph{factorizations} (not all legal ECGs). \emph{(i-b) In-class.}
With the positive row-normalizer $\Omega'^{1/2}:=\diag(1/(RA)^{-1}_{ii})$, $B':=I-\Omega'^{1/2}(RA)^{-1}$ and
$H'=HR^{-1}$, the map $R\mapsto(B',\Omega')$ is a \emph{diffeomorphism} from the \emph{stable chart}
$G^+:=\{R\in\GL(d):(RA)^{-1}\text{ has positive diagonal},\ \rho(B')<1\}$ onto the \emph{entire} legal space
$\{$zero-diagonal $B'':\rho(B'')<1\}\times\{$diagonal $\Omega''\succ0\}$, with inverse
$R=(I-B'')^{-1}\Omega''^{1/2}A^{-1}$; because every alternative has
$\X'=M\bm\eta+\epsilon=\X$, the complete law is preserved and $H'=HR^{-1}$ keeps full column rank. Hence
\emph{every} legal $(B,\Omega)$ is observationally equivalent to the truth. For $d\ge2$, $B$ is \emph{completely
non-identified within the class}, over the full $(d^2{-}d)$-dimensional $B$-space: the contraction constraint
carves the open chart $G^+$ but does \emph{not} reduce the non-identification dimension (differential
$\dot B'=CE-\diag(CE)C$ at $R=I$, surjective onto the zero-diagonal space with $d$-dimensional kernel
$\{C^{-1}DC:D\text{ diagonal}\}$ --- the noise-rescaling fibre). At $d=1$, zero diagonality forces $B=0$, so
$B$ is trivially identified and $G^+=\GL^+(1)$ varies only the sensor/noise scaling. The chart is open and
connected for every $d$. For $d\ge2$ it is a proper subset of $\GL^+(d)$ containing $I$---a chart, not a
subgroup: the orientation-preserving $C'=[\begin{smallmatrix}1&2\\-2&1\end{smallmatrix}]$ has
$\det C'=5$ but gives $\rho(B')=2$, and composing two legal relative coordinates can exit, with $\rho$ jumping
$0.9\!\to\!9.47$. There is no canonical volume fraction because its value depends on the chosen measure. This
is the \emph{unknown-$H$}, unknown-support freedom in this declared information class: constraints tying $H$ or
restricting support replace $G^+$ by the corresponding legal subfibre and can shrink or pin it.
\textbf{(ii) Hedge residual.} In the second-moment (Gaussian) regime, the
residual gauge on a full-rotational source component $S$ is the local signed-$\SO(|S|)$ branch (covariance
is invariant to $\Ortho(|S|)$; the admissible positive-variance orbit is the connected signed-$\SO(|S|)$)---exactly the
observed-layer equilibrium hedge (\Cref{thm:hedge})---of dimension $|S|(|S|{-}1)/2$. Acyclic latents have singleton
components and no such gauge.
\end{theorem}
\interpretation{In the theorem's ambient unknown-support, unknown-$H$ class, the single-environment linear
passive law does not identify $B$ for $d\ge2$: every legal stable zero-diagonal $B$ is represented within the
stable chart, while the $d=1$ case fixes $B=0$. In the second-moment regime, a full-rotational component retains
the stated signed-$\SO$ gauge; acyclic singleton components do not.}
\noindent\emph{Reading.} In the theorem's ambient unknown-support, unknown-$H$ information class and for $d\ge2$,
passive cyclic latents are maximally non-identified: \emph{every} legal $B'$ (with $\rho(B')<1$) is reproduced
by holding the truth's source mixing $M$ and complete joint source/noise law fixed and setting
$H'=M\Omega'^{-1/2}(I-B')$ --- a genuine stable chart of alternatives, not the full
$\GL(d)$ orbit (which exits the contraction class). At $d=1$, $B=0$ is the unique legal zero-diagonal structural
matrix; only the sensor/noise scaling fibre remains passive. Aligned mechanisms are one way to cut the chart down to
$\simeq$; support, $H$-anchors, or calibration restrictions define different subfibres. In particular, object
isotropy cannot replace the complete-fibre test of \Cref{def:infoclass}: for
\[
 B_t=\begin{pmatrix}0&t\\ t&0\end{pmatrix},\qquad H_t=I-B_t,\qquad 0<t<1,
\]
all legal models have $H_t(I-B_t)^{-1}=I$ and hence the same passive law (take $\Omega=I$ and the same sources),
although the continuum of distinct $(H_t,B_t)$ need not arise from a nontrivial stabilizer of any one object.

\subsection{Linear identifiability and intervention requirements}
The next theorem states what aligned response maps identify in the linear subclass, the intervention counts
under the named support assumptions, and the fixed-dimensional recovery rate.

\begin{theorem}[Linear identifiability, minimal interventions, and sample complexity]\label{thm:linear}
Let $\mathcal{L}$ belong to the zero-diagonal linear subclass, with every legal alternative likewise constrained
to have zero-diagonal $B$, full-column-rank environment-invariant $H$, LiNG sources, and a family of
single-target mechanism ($\doop$) interventions with target set $T=\bigcup_e T_e$; $K:=|T|$ counts
\emph{mechanism} environments. The recovery consumes each environment's aligned population response matrix
$M_0^{(e)}=H(I-B^{(e)})^{-1}$. These matrices may be supplied directly or acquired as follows. Within every
environment $e$, apply the same $q$ labelled, known latent-coordinate shift probes with design
$\mathsf D=[\delta^{(1)},\ldots,\delta^{(q)}]\in\R^{d\times q}$, $\rank\mathsf D=d$. Require that the actuator
affect $\X$ only through $\V$ and that the sensor-noise conditional mean be the same in the baseline and every
probe arm of that environment (or that its known arm-specific difference be removed). If
$R_e=[\E\X^{(e,\ell)}-\E\X^{(e,0)}]_{\ell=1}^q$, then
\[
 M_0^{(e)}=R_e\mathsf D^\top(\mathsf D\mathsf D^\top)^{-1}.
\]
The special design $\mathsf D=\Delta I_d$ uses $d$ known-magnitude coordinate probes. These acquisition probes
are not counted in $K$. In the noiseless submodel (or under a separately stated noisy-ICA identifiability
theorem for the declared sensor-noise law), independent per-environment ICA instead identifies
$M^{(e)}=M_0^{(e)}\Omega^{(e)1/2}$ only up to an environment-specific monomial transform; it supplies the aligned
$M_0^{(e)}$ input only under a separate common-label/common-scale and source-scale alignment theorem or
equivalent anchors, not from independent ICA runs alone. The counts below are therefore \emph{given aligned
$M_0^{(e)}$}, with the calibrated full-rank mean-response design above as the theorem-backed acquisition route.
\begin{enumerate}[leftmargin=1.5em,itemsep=1pt,topsep=1pt]
\item[\textbf{(a)}] \textbf{Sufficiency.} If $T=[d]$ (full coverage), \Cref{lem:recover} reconstructs every
row of $P$, so $(H,B)$ is identified up to $\simeq$ (\Cref{def:equiv}) for \emph{every} admissible instance.
If exactly one node $r$ is un-targeted --- with every $\doop$ well posed, $P_{jj}\neq0$ for $j\neq r$
(\Cref{lem:recover}) --- the other $d{-}1$ rows are recovered as above and row $r$ is pinned by the linear
cofactor system of \Cref{lem:recover} \emph{iff} the resulting $d\times d$ completion matrix $L$ is
nonsingular. \Cref{prop:completion} gives $L$ in closed form and settles nonsingularity for \emph{all} $d$
(no strong-connectivity hypothesis): with $C:=I-B$,
$L=-\diag(Ce_r)\,C^\top$, so $\det L=-\big(\prod_{k\neq r}B_{kr}\big)\det C$ (equivalently, in the
adjugate normalisation of \Cref{lem:recover}, $\det\propto(\det P)^{d-1}\prod_{k\neq r}B_{kr}$, matching
$B_{1r}B_{2r}(\det P)^2$ at $d=3$) and, more sharply, $\corank L=\#\{k\neq r:B_{kr}=0\}$. Hence, at generic
parameters, a single un-targeted node identifies \emph{iff it directly parents every other node of the
graph} (no missing child) --- a \emph{graph-global} condition; the un-identified dimension equals the number
of nodes $r$ fails to parent (its missing children). A node with a missing child is thus \emph{not} identified at $K=d-1$: a
directed ring of size $d\ge3$ (out-degree one) collides in $d-2$ directions (a single direction at $d{=}3$), so $K_{\min}=d$; a \emph{sink}
(out-degree zero) collides in $d-1$ directions. The latter example shows why a component-local reading of the
\emph{iff} is insufficient: a sink lies in no source component and therefore vacuously ``parents its component.''
\item[\textbf{(b)}] \textbf{Necessity.} In the ambient unknown-support class, if \emph{two or more} nodes
$r,s$ are never $\doop$-targets, then $(H,B)$ is \emph{non-identified}. With $C=I-B$ and
$E_t=t(e_s-C_{rs}e_r)e_r^\top$, the explicit family
\[
 C_t=C(I+E_t),\quad P_t=(I+E_t)^{-1}P,\quad B_t=I-C_t,\quad H_t=H(I+E_t)
\]
has $\diag C_t=\mathbf1$, shares every targeted row of $P$, and therefore reproduces every environment response.
For all sufficiently small $t\ne0$ it is invertible and stable, while $I+E_t$ is non-monomial, so full-column-rank
$H$ makes it non-$\simeq$-equivalent. This gives a continuous collision without invoking a full group action.
\item[\textbf{(c)}] \textbf{Minimal $K$.} The criterion is query-relative, with fixed counts only as
corollaries. The operative quantity is
$K_{\min}(Q,G)=\min_{\mathcal I}|\mathcal I|$ subject to $Q$ being constant on the complete admissible legal
fibre of the identifying map. On a regular constant-rank neighbourhood, the condition
$u_Q(\theta;\mathcal I)=0$ at every nearby $\theta$ is the local differential screen; global identification still
requires clearance of every disconnected/discrete fibre component and the gauge-pinning layer. The fixed counts
below are corollaries under named supports. For a single full-rotational source component of size $d$: $K=d-2$ is non-identified (b), and
$K=d-1$ is sufficient \emph{iff} the un-targeted node directly parents every other node of the graph (a) --- the \emph{iff} proved for all $d$ by the closed-form completion criterion of \Cref{prop:completion} ($\corank L=\#$ missing children; the ring $K=d$ and sink $K_{\min}$ counterparts are corollaries). Hence
$K_{\min}=d-1$ when the component has such a node (e.g.\ complete support) and $K_{\min}=d$ otherwise (e.g.\ a
directed ring of size $d\ge3$, whose every node has out-degree one). For a \emph{general} graph with \emph{unknown}
mixing support, targeting all but one node in the whole graph is \emph{necessary}: any two un-targeted rows
leave a positive-dimensional family of $(H',B')$ reproducing every environment's law (the half-trek transfer
\citep{FoygelDraismaDrton2012} that would localise interventions component-by-component is an
\emph{observed-variable} argument and does not survive an unknown mixing). With the sole un-targeted
row $r$, it is \emph{sufficient} at $K=d-1$ exactly when the graph-global condition of (a) holds:
$B_{kr}\ne0$ for every $k\ne r$, so $r$ directly parents every other node. A completion-deficient choice of
$r$ (e.g.\ any node of a source ring of size at least $3$) still requires full coverage $K=d$. With
\emph{known} support, the exact support-constrained completion map
$\mathcal C_{G,T}$ of \Cref{prop:knownsupport} supplies the criterion: on a regular constant-rank quotient
stratum, local identification is equivalent to full column rank on the actual constrained tangent, while global
identification additionally requires its complete legal zero-fibre to be one declared-equivalence class (or the
query to be constant there). There is no universal scalar target count, SCC sum, or lower bound: a known directed
$3$-cycle is identified from one target although $\sum_S(|S|-1)=2$ (\Cref{prop:knownsupport}).
\item[\textbf{(d)}] \textbf{Sample complexity.} Work at a regular instance in an identified regime of (a)/(c):
after quotienting by $\simeq$, the relevant recovery branch is locally single-valued, every denominator used by
the row recovery is nonzero, and the applicable completion matrix is nonsingular (for the one-missing-row case,
$\sigma_{\min}(L)>0$ in the row-$P$ normalization of \Cref{prop:completion}). At \emph{fixed} $(p,d,K,q)$,
suppose every baseline/probe arm has $n$ independent (or otherwise CLT-valid) samples with finite covariance,
the sensor-mean/no-direct-effect conditions above hold, and the fixed probe design has $\rank\mathsf D=d$.
Then the plug-in recovery on that branch is
$\sqrt n$-consistent,
$\|\widehat B-B\|_F=O_p(n^{-1/2})$; Exp.~9 is consistent with this predicted rate on its $d=3$ simulated
instances. The constant grows with $\|(I-B)^{-1}\|$, $\|I-B\|^2$, $\sigma_{\min}(H)^{-1}$,
$\max_j|P_{jj}|$ (equivalently $(\min_j|1-w_{jj}|)^{-1}$: larger loop gain \emph{worsens} conditioning, since
$P_{jj}=1/(1-w_{jj})$ gives $dP_{jj}/dw_{jj}=P_{jj}^2$), the completion-system conditioning
$\sigma_{\min}(L)^{-1}$ when that solve is used, and the inverse probe-design
norm $\|\mathsf D^\top(\mathsf D\mathsf D^\top)^{-1}\|$ (equal to $\Delta^{-1}$ for
$\mathsf D=\Delta I_d$), and the source/sensor second moments (or the corresponding moments required by the
declared CLT); in adjugate coordinates the completion factor is instead that of the actual cleared matrix
$\widetilde L=(\det P)L$, including its declared scale. We do \emph{not} bound $\|(I-B)^{-1}\|$ by
$(1-\rho(B))^{-1}$ because an explicit non-normal counterexample violates this bound. We \emph{do not} claim a dimension-explicit rate:
for a fixed matrix, uniform entrywise $O_p(n^{-1/2})$ bounds give only the loose Frobenius upper order
$O_p(d/\sqrt n)$ across about $d^2$ entries; a lower order or a uniform-in-$d$ statement needs nondegenerate
covariance and matrix-concentration / operator-norm assumptions, which we leave open. A finite-sample
error-ratio calculation at $d=3$ is only a diagnostic and does not establish the exponent or dimensional scaling.
\end{enumerate}
The closed form of \Cref{lem:recover} uses the full coverage $T=[d]$ for simplicity; the $K=d-1$ threshold
in (c) is the \emph{information-theoretic} boundary, established for \emph{all} $d$ by the closed-form completion criterion (\Cref{prop:completion}: $\corank L=\#$ missing children) and corroborated on universal-parent supports by a seeded parametrized collision search on random instances
(Exp.~8); the explicit family in part~(b) proves the general-graph necessity clause
(unknown support), and a separate finite-dimensional calculation corroborates a named multi-component instance.
\end{theorem}

\interpretation{Given the aligned response matrices and the theorem's acquisition and model assumptions, full
target coverage identifies $(H,B)$ up to $\simeq$. In the ambient unknown-support class, leaving one node
un-targeted identifies exactly when that node directly parents every other node, while leaving two un-targeted
nodes gives non-identification. The fixed-dimensional plug-in rate applies only on the stated regular identified
branch, and no dimension-explicit rate is claimed.}

\begin{remark}[Spectral stability versus resolvent conditioning]\label{rem:rho-resolvent}
For the signed two-cycle
$B=\left(\begin{smallmatrix}0&\beta\\-\delta&0\end{smallmatrix}\right)$ with $\beta,\delta>0$,
$\det(I-B)=1+\beta\delta$ and $\rho(B)=\sqrt{\beta\delta}$.  On the balanced path
$\beta=\delta=t$, $(I-B)^\top(I-B)=(1+t^2)I$, so
$\|(I-B)^{-1}\|_2=(1+t^2)^{-1/2}$ and $\kappa_2(I-B)=1$ even as $\rho(B)=t\to1$.
Thus the contraction margin and resolvent conditioning are distinct: $\rho(B)<1$ governs the declared
contraction regime, whereas $\sigma_{\min}(I-B)$ governs the resolvent norm.  This distinction does not exclude
slow finite iteration, solver tolerances, dynamics, active-set changes, or other optimizer mechanisms.
\end{remark}

Given the aligned baseline and target-environment response matrices under the preceding full-column-rank
assumptions, a well-posed single-node mechanism intervention recovers the corresponding row of the equilibrium
resolvent.

\begin{lemma}[Closed-form cyclic recovery]\label{lem:recover}
Let $M_0:=H(I-B)^{-1}$ and, for a $\doop$ on node $j$ that is well posed ($\det(I-B^{(j)})\neq0$, equivalently $P_{jj}\neq0$), $M_0^{(j)}:=H(I-B^{(j)})^{-1}$. Then
$D_j:=M_0^{(j)}-M_0$ has rank at most one (it is zero exactly when row $j$ of $B$ is zero),
\[
D_j=-\,\frac{(M_0\, e_j)\,(b_j^\top P)}{P_{jj}},\qquad b_j^\top P=(P-I)_{j,:},\quad P:=(I-B)^{-1},
\]
where $b_j^\top$ is row $j$ of $B$. Hence $w_j^\top:=-(M_0e_j)^\top D_j/\|M_0e_j\|^2=(P-I)_{j,:}/P_{jj}$ is
exact; $P_{jj}=1/(1-w_{jj})$ and row $j$ of $P$ equals $e_j^\top+P_{jj}w_j^\top$. Stacking all rows recovers
$P$, then $B=I-P^{-1}$ and $H=M_0P^{-1}$, up to the labelling/scale of $\simeq$.
\end{lemma}

In the ambient unknown-support class, under the proposition's zero-diagonal, nonsingularity, and
targeted-intervention well-posedness assumptions, the next result determines when the remaining row is
identified.

\begin{proposition}[Closed-form completion criterion; general-$d$ universal-parent characterisation]\label{prop:completion}
Let $B$ have zero diagonal, $C:=I-B$ nonsingular, $P:=C^{-1}$, and un-target a single node $r$ with every
$\doop$ on $j\neq r$ well posed ($P_{jj}\neq0$). With the $d{-}1$ rows $j\neq r$ of $P$ pinned by
\Cref{lem:recover} and row $r$ the unknown $x$, matching every environment mixing is automatic (Lemma~0
below), so the only constraints are the $d$ zero-diagonal equations $(P'^{-1})_{jj}=1$, whose Jacobian at the
truth is the completion matrix
\[
  L=-\diag(Ce_r)\,C^\top,\qquad (Ce_r)_r=1,\ (Ce_r)_k=-B_{kr}\ (k\neq r).
\]
Consequently $\det L=-\big(\textstyle\prod_{k\neq r}B_{kr}\big)\det C$ and, for the adjugate normalisation
$\widetilde L=(\det P)\,L$ of the cofactor system, $\det\widetilde L=-(\det P)^{d-1}\prod_{k\neq r}B_{kr}$;
moreover $\corank L=\#\{k\neq r:B_{kr}=0\}$ with $\ker L=P^\top\!\operatorname{span}\{e_k:k\neq r,\ B_{kr}=0\}$.
Hence row $r$ is identified iff $B_{kr}\neq0$ for every $k\neq r$ (\emph{$r$ directly parents every other node
of the graph}); this holds for \emph{all} $d$ and uses no strong-connectivity/full-rotational hypothesis.
\end{proposition}
\begin{proof}
\emph{Lemma~0 (environment matching automatic).} For any $x$ with $\det P'\neq0$, set $B'=I-P'^{-1}$,
$H'=M_0P'^{-1}$ (same sources). For $e=\doop(j)$, $j\neq r$, the matrix $(I-B^{(j)})P$ has rows $e_i^\top$
($i\neq j$) and row $j$ equal to $P_{j,:}$; the same holds for $(I-B'^{(j)})P'$, and since $j\neq r$ and $P'$
differs from $P$ only in row $r$, $P'_{j,:}=P_{j,:}$. Thus
$K_j:=(I-B'^{(j)})P'=(I-B^{(j)})P$ exactly, whence
$\det(I-B'^{(j)})=\det(I-B^{(j)})\det P/\det P'\neq0$ (well posed) and
$H'(I-B'^{(j)})^{-1}=M_0K_j^{-1}=H(I-B^{(j)})^{-1}$; the observational case is immediate. So the alternative reproduces
every environment for \emph{every} $x$, and the only membership constraints are $B'_{jj}=0$, i.e.
$(P'^{-1})_{jj}=1$.
\emph{Closed form.} Write $P'=P+e_r\delta^\top$, $\delta:=x-P_{r,:}^\top$. By Sherman--Morrison
$P'^{-1}=C-\dfrac{(Ce_r)(\delta^\top C)}{1+\delta^\top Ce_r}$, so $(P'^{-1})_{jj}-1=-\dfrac{(Ce_r)_j(C^\top\delta)_j}{1+\delta^\top Ce_r}$
(using $C_{jj}=1$). Differentiating at $\delta=0$ gives $L=-\diag(Ce_r)C^\top$ exactly; equivalently, in the
cleared (adjugate) form the $(1+\delta^\top Ce_r)$ factor cancels because $C_{jj}=1$, leaving an exactly
linear system with matrix $\widetilde L=(\det P)L$.
\emph{Rank and determinant.} $\det L=(-1)^d\big(\prod_j (Ce_r)_j\big)\det C^\top=(-1)^d(-1)^{d-1}\big(\prod_{k\neq r}B_{kr}\big)\det C=-\big(\prod_{k\neq r}B_{kr}\big)\det C$; the change of variables $y=C^\top\delta$ gives
$L\delta=0\iff (Ce_r)_jy_j=0\ \forall j\iff y\in\operatorname{span}\{e_k:B_{kr}=0\}$ (index $r$ excluded since
$(Ce_r)_r=1$), so $\corank L=\#\{k\neq r:B_{kr}=0\}$ and $\ker L=P^\top\!\operatorname{span}\{e_k:k\neq r,\ B_{kr}=0\}$.
\emph{Genuineness of the collision.} On $\ker L$ the family $B'(y)=B+(Ce_r)y^\top$, $H'(y)=H(I-e_ry^\top)$ is
zero-diagonal ($(Ce_r)_ky_k=-B_{kr}y_k=0$ on the support), keeps every $\det(I-B'^{(j)})$ fixed (so all
environments stay well posed), is admissible ($\rho(B'(y))<1$ on a neighbourhood of $y=0$ by continuity), and
is \emph{not} $\simeq$-equivalent to truth: $\simeq$ forces $H'=H\Pi\Lambda$ with $\Pi\Lambda$ monomial, hence
$I-e_ry^\top=\Pi\Lambda$, but row $r$ of $I-e_ry^\top$ is $e_r^\top-y^\top$, which keeps its diagonal $1$
(as $y_r=0$) and has an extra nonzero at each missing child, so it is monomial only when $y=0$ (here $H$
full column rank is load-bearing). Thus $\corank L>0$ gives a genuine positive-dimensional non-$\simeq$
collision, and $\corank L=0$ pins row $r$ uniquely.
\end{proof}
This is the \emph{unknown/complete-support} statement (the collision family adds edges, so a known support
pattern may exclude the whole flat and identify at $K=d{-}1$ regardless; cf.\ part (c)). The sink construction
is the special case $Be_r=0\Rightarrow Ce_r=e_r\Rightarrow B'=B+e_ry^\top$, a $(d{-}1)$-dimensional collision.

\begin{proposition}[Exact known-support completion map and fibre criterion]\label{prop:knownsupport}
Fix a declared directed support $G=(V,E)$, with $B_{ij}$ allowed to be nonzero exactly when $j\to i\in E$,
and a target set $T$. From the aligned response maps of \Cref{thm:linear}, recover the rows
$\bar P_{j,:}$ for $j\in T$. Let $x$ collect all entries of the un-targeted rows and let $P(x)$ be the matrix
whose target rows equal $\bar P_{j,:}$ and whose remaining rows are $x$. On the open domain where $P(x)$ is
invertible, every target mechanism remains well posed, $B(x):=I-P(x)^{-1}$ is stable, and every declared
present edge is nonzero (with any declared sign), set
\[
 Z_G:=\{(i,j):i\ne j,\ j\to i\notin E\},\qquad
 \mathcal C_{G,T}(x):=
 \bigl((P(x)^{-1})_{ii}-1\bigr)_{i=1}^d
 \oplus\bigl((P(x)^{-1})_{ij}\bigr)_{(i,j)\in Z_G}.
\]
For the aligned population-response data consumed by \Cref{thm:linear}, with the same declared source laws,
the complete known-support legal data fibre is exactly
\[
 \left\{\bigl(H(x),B(x)\bigr):\mathcal C_{G,T}(x)=0,\quad
 H(x):=M_0P(x)^{-1}\right\}\big/{\simeq}.
\]
Consequently, on a smooth quotient stratum on which $D\mathcal C_{G,T}$ has constant rank, local point
identification is equivalent to $D\mathcal C_{G,T}$ having full column rank on the actual quotient tangent.
Global point identification is the separate condition that the entire legal zero-fibre above is one
$\simeq$-class; query identification requires query constancy on that whole fibre. Neither the number of
displayed constraints nor the SCC scalar $\sum_S(|S|-1)$ is a rank or fibre-cardinality theorem.

For example, on the known directed $3$-cycle
\[
 B=\begin{pmatrix}0&0&a\\ b&0&0\\0&c&0\end{pmatrix},\qquad 0<|abc|<1,
\]
targeting node $1$ recovers
$P_{1,:}=(1,ac,a)/(1-abc)$, and therefore
\[
 a=\frac{P_{13}}{P_{11}},\qquad
 c=\frac{P_{12}}{P_{13}},\qquad
 b=\frac{1-P_{11}^{-1}}{ac}.
\]
Thus one target identifies this known support, while $\sum_S(|S|-1)=2$; the latter is not even a universal
lower bound.
\end{proposition}
\begin{proof}
For any candidate $P(x)$ sharing every target row, put $C(x)=P(x)^{-1}$ and $H(x)=M_0C(x)$. For a target $j$,
$K_j:=(I-B^{(j)})P$ has every row equal to the corresponding identity row except row $j$, which equals
$P_{j,:}$; the same description holds for $K_j(x):=(I-B(x)^{(j)})P(x)$. Since the target rows agree,
$K_j(x)=K_j$. Hence
\[
 H(x)(I-B(x)^{(j)})^{-1}=M_0K_j^{-1}=H(I-B^{(j)})^{-1},
\]
and the baseline equality is immediate. Conversely, zero diagonal is exactly
$(P(x)^{-1})_{ii}=1$, while every absent edge is exactly $(P(x)^{-1})_{ij}=0$; the remaining support, sign,
stability, invertibility, and well-posedness requirements are precisely the open-domain restrictions stated
above. This proves the fibre identity. The constant-rank theorem gives local fibre dimension
$\dim x-\rank D\mathcal C_{G,T}$ on the quotient stratum, proving the local equivalence; the global statement is
the definition of complete-fibre identification. For the displayed cycle, direct inversion gives the stated row,
and the three ratios recover $(a,b,c)$ exactly.
\end{proof}

The two non-identifiability counterparts give a genuine two-sided bracket (cf.\ the observed-layer hedge).

\paragraph{The rank criterion is query-relative and local; the fibre criterion is global.}\label{par:queryrel}
Use the declared smooth quotient stratum $\Theta$, identifying map $\Phi_{\mathcal I}$, and query $Q$ of
\Cref{thm:interv-id}. At a regular truth $\theta_0$, shrink to a constant-rank coordinate neighbourhood $U$
whose local fibre plaques are connected. The constant-rank theorem gives local coordinates
$\Phi_{\mathcal I}(s,t)=(s,0)$, with the $t$-directions tangent to its connected local fibre plaques. Hence
\[
 \ker D\Phi_{\mathcal I}(\theta)\subseteq\ker DQ(\theta)
 \quad\text{for every }\theta\in U
 \quad\Longleftrightarrow\quad
 Q=q\circ\Phi_{\mathcal I}\ \text{locally on }U,
\]
equivalently $u_Q(\theta;\mathcal I)=0$ throughout $U$. Inclusion only at $\theta_0$ is insufficient, and this
differential screen is not applied across rank-changing, singular, or boundary strata. Global point
identification at $\theta_0$ is the separate exact condition that $Q$ be constant on the entire admissible fibre
$\Phi_{\mathcal I}^{-1}(\Phi_{\mathcal I}(\theta_0))$ after quotienting the declared equivalence: constancy on
each connected component is not enough unless the component values also agree. Thus a finite or
zero-dimensional fibre can still contain several query-distinct points. This is the query-relative,
interventional refinement of identifying-Jacobian reasoning \citep{FoygelDraismaDrton2012}. Full structural
recovery is $Q(B)=B$; lower-dimensional queries may use fewer interventions only when the local kernel condition
holds and the complete global fibre is query-constant. Accordingly, the cheapest design minimizes
$|\mathcal I|$ subject to full-fibre query constancy, using $u_Q=0$ on regular neighbourhoods as a local screen,
not as a pointwise global criterion.

\begin{proposition}[Shift interventions are insufficient under an unknown mixing for $d\ge2$]\label{prop:shift}
Let $d\ge2$. In the ambient unknown-support, unknown-$H$ information class, at the population structural-response level---and
observationally under the calibrated, full-rank, mean-stable sensing conditions of \Cref{thm:linear}---shift
(backShift) interventions, which preserve $B$,
identify only $M_0=H(I-B)^{-1}$ (up to unknown probe scale when the shifts are not calibrated).
For any legal stable zero-diagonal $B'$, setting $H'=M_0(I-B')$ gives $H'(I-B')^{-1}=M_0$; a distinct nearby
$B'$ exists and reproduces the entire shift-response. Thus $B$ is non-identified from shifts alone in that class; the aligned mechanisms
of \Cref{thm:linear} supply sufficient additional information. This is not an ``only mechanism'' theorem across
information classes: known support, known/anchored $H$, or calibration restrictions can shrink or pin the fibre.
(At $d=1$, zero diagonality forces $B=0$, so $B$ is trivially identified even when $H$ is unknown. When $H=I$
is known, $B=I-M_0^{-1}$ is identified---this recovers backShift
\citep{Rothenhausler2015backshift} and the observed-variable case.)
\end{proposition}

\begin{proposition}[Non-Gaussianity collapses the passive hedge after source-mixing acquisition]\label{prop:nongauss}
In the passive (unknown-magnitude) \emph{noiseless linear} regime second moments fix $M^{(e)}$ only up
to right multiplication by $\Ortho(d)$. Under the LiNG condition of \Cref{def:model},
per-environment independent-component analysis fixes $M^{(e)}$ up to permutation$+$scale
\citep{Comon1994ica,Lacerda2008cyclic}; with additive sensor noise, this step instead requires a separately
stated noisy-ICA identifiability theorem for the declared noise law. In either case,
the residual within-source-component rotation $\Ortho(m)$ (the hedge) is admissible \emph{iff} the rotated
sources $Q^\top\bm\eta$ remain independent. An isotropic Gaussian component permits every orthogonal $Q$;
under LiNG (independent nondegenerate components, with at most one Gaussian), ICA uniqueness permits only the
monomial ambiguity. Hence the Gaussian component stays hedged, whereas LiNG collapses $\Ortho(m)$ to
permutation$+$scale. Mechanism interventions then split $H$ from $B$ as in
\Cref{thm:linear} once its required cross-environment response alignment is supplied; independent
per-environment ICA alone does not supply that alignment.
\end{proposition}

\subsection{Nonlinear block-identifiability: a sharp impossibility, a scoped source-block positive, and an open general case}
The next pair separates the nonlinear obstruction from the scoped positive result: first the
isotropy-preserving collisions, then the sufficient score-rank and irreducibility condition.

\begin{theorem}[Diffeomorphic mixing: dimension-scoped isotropy collisions; the general positive is open]\label{thm:block}
Let $\mathcal{L}$ have a diffeomorphic mixing $h$.
\textbf{Negative (the equilibrium hedge acts within \emph{and} across components).}
\emph{(Within)} Let a source block be $S\subset\R^m$ with $m\ge2$, with isotropic-Gaussian sources. If every
labelled available environment preserves the radial/isotropic law needed by the construction (observation and
block-isotropic/radial interventions), fix a nontrivial one-parameter rotation $R_t\in\SO(m)$ and choose smooth
$s$, constant near $0$ for smoothness, such that $r\mapsto R_{s(r)}$ is nonconstant on the radial support. The
\emph{twist} $\sigma:(r,\theta)\mapsto(r,R_{s(r)}\theta)$ of the source polar coordinates
preserves the baseline isotropic-Gaussian law---and, by radiality, every labelled available law---while
changing the within-component factorization, an infinite-dimensional gauge $\supsetneq\Ortho(m)$.
\emph{(Across)} Let two isotropic source blocks be $A\in\R^{m_A}$ and $B\in\R^{m_B}$, with $m_A\ge2$ and
$m_B\ge1$ (equivalently, for two blocks at least one has dimension at least two, and that rotated block is named
$A$). If every labelled available environment preserves the radial/isotropic product law needed below, the
\emph{$B$-dependent rotation}, for a fixed nontrivial one-parameter $R_t\in\SO(m_A)$ and smooth $s$ such that
$b\mapsto R_{s(\|b\|^2)}$ is nonconstant on the support of $B$,
$T:(A,B)\mapsto(R_{s(\|B\|^2)}A,\,B)$ preserves every such law (conditional on $B$ it only rotates the isotropic
$A$) yet \emph{entangles} the components---its first block depends on $B$---so absorbing $T^{-1}$ into $h$ gives
an observationally- and block-interventionally-equivalent representation with a \emph{different component
partition}. Thus under isotropy-preserving interventions neither the within-component factorization
\emph{nor the component partition} is identified; ``interventional faithfulness'' (a nonzero response) does not
exclude these collisions.
\textbf{Positive (linear, or isotropy-breaking; otherwise open).} The linear specialization is fully identified
(\Cref{thm:linear}). For diffeomorphic $h$, an isotropy-\emph{breaking} intervention---e.g.\ a per-coordinate
perfect $\doop$ inside a component---breaks the local isotropy and collapses the gauge in that direction (in
the \emph{linear} specialization this is the rank-at-most-one difference of \Cref{lem:recover}; for general $h$ it is a
local score-gradient constraint, \emph{not} a linear rank-at-most-one difference). But a \emph{general} positive block-identification theorem for diffeomorphic $h$
(recovering the partition and component-wise factors up to permutation from observed interventional laws) is
\emph{not} established here and, by the collisions above, is \emph{false} without further restriction;
block-identification for general diffeomorphic $h$ therefore remains \textbf{open}. Of the viable routes, (i)
full-rank, block-\emph{distinct} interventional score fields establish the \emph{source}-block factorization in
\Cref{thm:blockid-source} below (a sufficient score-rank$+$irreducibility condition; for $d\ge2$ it does not separate
$h$ from $B$, while at $d=1$ zero diagonality fixes $B=0$ and $h=g$; it does not reach non-source components);
routes (ii) a block-\emph{separable} decoder $h=\bigoplus_S h_S$ and
(iii) paired/counterfactual intervention information remain open.
\end{theorem}

\interpretation{With diffeomorphic mixing, if a source block $S\subset\R^m$ has $m\ge2$, isotropic-Gaussian
sources, and labelled available interventions preserving the radial/isotropic law used by the construction, the
displayed twist prevents identification of its within-component factorization. If two source blocks are
$A\in\R^{m_A}$ and $B\in\R^{m_B}$ with $m_A\ge2$, $m_B\ge1$, and their labelled available interventions
preserve the required radial/isotropic product law, the displayed cross-component transformation also prevents
identification of the component partition. Isotropy-breaking interventions impose only a local
score-gradient constraint in the general nonlinear case; the general positive block-identification result
remains open, with the scoped source-block result below.}

\begin{theorem}[A sufficient source-block identification condition: score rank$+$irreducibility (the positive, scoped)]\label{thm:blockid-source}
Read the ECG in the \emph{source frame}: $X=g(Z)$ with $g=h\circ(I-B)^{-1}$ a $C^1$ diffeomorphism onto its
image and $Z=(Z_1,\dots,Z_m)$ the \emph{independent} source-component blocks ($Z_i\in\R^{n_i}$, $C^1$
strictly-positive densities). Environments $E=\bigsqcup_i E_i$ are soft/stochastic single-block interventions:
each $e\in E_i$ reweights block $i$ only, $p_e=p_0\,e^{\rho_e(z_i)}$ with $\rho_e\in C^1$ non-constant. Assume,
for every block $i$: \textup{(R)} $\mathrm{span}\{\nabla\rho_e(z_i):e\in E_i\}=\R^{n_i}$ at a.e.\ $z_i$ (hence
$|E_i|\ge n_i$); and \textup{(I$_i$)} \emph{no} bipartition $E_i=E^1\sqcup E^2$ (both nonempty) makes the tuples
$(\rho_e(Z_i))_{e\in E^1}$ and $(\rho_f(Z_i))_{f\in E^2}$ independent under $p_0$. Call an alternative
$X=g'(Z')$ \emph{admissible} if $g'$ is a $C^1$ diffeomorphism onto its image with independent blocks $Z'_j$,
matched latent dimension $d'=d$ and matched image $g'(\R^{d})=g(\R^{d})$, and each environment reweights a single
alternative block. Then, among admissible alternatives, the source-block
factorization is the \emph{unique finest} representation: every admissible alternative reproducing the
observational and all interventional laws is a \emph{coarsening} (merge) of the source blocks, and any alternative
with the same (maximal) block count agrees with the truth after a \emph{size-respecting block permutation}, with
$g'^{-1}\!\circ g=\bigoplus_i\phi_i$ a direct sum of block diffeomorphisms. The residual gauge is
exactly \emph{size-respecting block permutations} $\times\prod_i\mathrm{Diff}(\R^{n_i})$ (the realization
direction uses $C^2$ block diffeomorphisms, or $C^0$ block densities, to stay within the $C^1$-density class).
\textbf{Scope (what this does \emph{not} do).} It identifies the source blocks of $Z$---equivalently the
composite decoder $g$---up to block-diffeomorphism; for $d\ge2$ it does \emph{not} separate $h$ from $B$
(reweighting environments obey \Cref{prop:shift}); at $d=1$, $B=0$ is fixed while $h=g$ remains identified only
up to the declared block diffeomorphism. It does not identify non-source (downstream) components (whose $V$-factors are not
independent, $V=(I-B)^{-1}U$), nor within-block structure. A \textup{(R)}-deficient score span is \emph{one}
source of collision. More precisely, if the score span has constant rank $k$ on an $n_i$-block and there exist
$q=n_i-k$ pointwise independent, pairwise commuting, complete $C^1$ vector fields $w_1,\ldots,w_q$ with
$w_a\perp\mathrm{span}\{\nabla\rho_e:e\in E_i\}$ and $\mathrm{div}(p_0w_a)=0$, their joint flows generate a
$q$-dimensional measure-preserving $\R^q$ action that preserves every environment law. One qualifying field
gives only a one-parameter flow; rank deficiency alone does not supply the frame, commutation, or completeness.
(The isotropic rotation of \Cref{thm:block} is the $n_i{=}2$, $k{=}1$ witness.) Failure of \textup{(I$_i$)} is a
\emph{separate} obstruction. Thus \textup{(R)}$+$\textup{(I$_i$)} \emph{jointly} is a sufficient source-block identification
condition---\textup{(R)} the nonlinear analogue of the completion corank of \Cref{prop:completion}, and
\textup{(I$_i$)} its irreducibility complement; we prove joint \emph{sufficiency}, not an \emph{iff}
characterization of every failure of identification.
\end{theorem}
\interpretation{Under the theorem's source-frame model with a $C^1$ diffeomorphic composite decoder,
independent source blocks with $C^1$ strictly positive densities, labelled soft single-block reweightings by
non-constant $C^1$ log-density ratios, matched latent dimension and image, and conditions \textup{(R)} and
\textup{(I$_i$)}, every admissible alternative reproducing the observational and all interventional laws is a
coarsening of the source blocks. Any such alternative with the same maximal block count agrees up to a
size-respecting block permutation and block diffeomorphisms. This identifies the source blocks and composite
decoder $g$ only to that gauge; it does not separate $h$ from $B$ for $d\ge2$ or identify downstream or
within-block structure, and the theorem proves sufficiency rather than an \emph{iff} characterization.}
\begin{remark}[Cycles can supply irreducibility excluded by multi-target per-node shifts]\label{rem:cycleshelp}
Condition \textup{(I$_i$)} is the operative and \emph{non}-generic one. A per-node \emph{shift} intervention
$u_j\mapsto u_j+\delta$ on node $j$ has score $\rho_{e_j}(u)=\log q_j(u_j-\delta)-\log q_j(u_j)$---a function of the
single exogenous scalar $u_j$; the $u_j$ are independent, so any per-node shift design with $\ge2$ \emph{distinct}
target nodes in a block \emph{violates} \textup{(I$_i$)} for \emph{every} source law (its scores are separable
across the independent $u_j$). Feedback \emph{can} restore it---but not automatically. For a
\emph{standardized-Gaussian} source SCC ($U\sim\mathcal N(0,I)$) a perfect $\doop(V_j\!\leftarrow\!W_j)$,
$W_j\sim\mathcal N(0,s_j^2)$, has source-frame score $\rho_{\doop(j)}(u)=\tfrac12 u_j^2-\tfrac{1}{2s_j^2}
(P_j\!\cdot\!u)^2$ (up to a constant), gradient $\nabla\rho_{\doop(j)}=H_ju$, $H_j=e_je_j^\top-s_j^{-2}P_j^\top P_j$;
for a \emph{general} source law the score is $\log q_{W_j}(P_j\!\cdot\!u)-\log q_j(u_j)+\mathrm{const}$, \emph{not}
this quadratic. These couple across the block ($\beta=P^\top b_j\neq0$, the \emph{baseline} $P$), \emph{but coupling
is not sufficient}: whether $\{\nabla\rho_e\}$ spans (condition \textup{(R)}) depends on the intervention variances.
For a $2$-cycle $B=\big(\begin{smallmatrix}0&a\\ b&0\end{smallmatrix}\big)$ the pair is rank-deficient \emph{exactly}
at the resonance $s_1^2=s_2^2=c:=(1-ab)^{-1}$ (i.e.\ $s_1=s_2=\sqrt c$)---there
$\rho_{\doop(1)}=\tfrac{a}{b}\,\rho_{\doop(2)}$ and $\mathrm{span}\{\nabla\rho\}$ collapses to a line---and full
rank for every other choice. \Cref{lem:cyclesRI} makes this precise: on a
\emph{strongly-connected standardized-Gaussian} source SCC, a per-node perfect-intervention design supplies
\textup{(R)}$+$\textup{(I$_i$)} at all intervention variances outside a proper measure-zero algebraic set.
\emph{The cycle thus dissolves the equilibrium hedge off a measure-zero set of intervention variances, not
automatically; full rank of the coefficient vectors $\alpha_j=e_j+\beta$ alone does not imply \textup{(R)},
which depends on the actual gradients $H_ju$.} Finite-dimensional calculations on the $2$- and $3$-cycle
(actual gradient rank $+$ a distance-covariance dependence statistic $+$ the resonance witness) corroborate
these instances; they are finite-instance evidence, not a substitute for \Cref{lem:cyclesRI}.
\end{remark}

\begin{lemma}[Gaussian-generic realizability of \textup{(R)}$+$\textup{(I$_i$)} by feedback]\label{lem:cyclesRI}
Let a source SCC of size $m\ge2$ carry standardized-Gaussian sources $U\sim\mathcal N(0,I_m)$ with
\emph{strongly-connected} coupling $B$ whose intervened systems are well-posed (each $I-\bar B^{(j)}$ invertible,
equivalently $P_{jj}\neq0$---an assumption in its own right: it holds if every \emph{post-intervention} map also
satisfies the contraction condition, but is \emph{not} implied by contraction of the baseline alone, as
$\rho(B)<1$ permits $P_{jj}=0$), and equilibrium map $P=(I-B)^{-1}$. Apply one perfect intervention per node, $\doop(V_j\!\leftarrow\!W_j)$, $W_j\sim\mathcal N(0,s_j^2)$,
giving scores $\rho_j(u)=\tfrac12 u_j^2-(P_j\!\cdot\!u)^2/(2s_j^2)$ (up to a constant) and Hessians
$H_j=e_je_j^\top-s_j^{-2}P_j^\top P_j$; each $\rho_j$ is non-constant, as $P_j\not\parallel e_j$ in a
strongly-connected SCC. Then \textup{(R)} $\mathrm{span}\{\nabla\rho_j(u)\}=\R^m$ at a.e.\ $u$ holds for all
$(s_j^2)$ outside a proper algebraic (measure-zero) set of \emph{variance} space, and \textup{(I$_i$)} holds at
\emph{every} positive variance; hence a non-resonant per-node design realizes \textup{(R)}$+$\textup{(I$_i$)}.
\end{lemma}
\emph{Proof.} \emph{(R).} With $t_j:=s_j^{-2}$, column $j$ of the gradient matrix is $H_ju=u_je_j-t_j(P_j\!\cdot\!u)
P_j^\top$, affine in $t_j$, so $F(u,t):=\det[H_1u,\dots,H_mu]$ is a polynomial in $(u,t)$ with $F(u,0)=
\det[u_1e_1,\dots,u_me_m]=\prod_j u_j$. Hence the coefficient of $\prod_j u_j$ in $F$ equals $1$ at $t=0$ and is a
\emph{nonzero} polynomial in $t$; the variances for which $F(\cdot,t)\equiv0$ in $u$ lie in its zero set, a proper
algebraic (measure-zero) subset of variance space (clear denominators to make it polynomial in the $s_j^2$). Off
it, $F(\cdot,t)$ is a nonzero degree-$m$ polynomial in $u$, so $\mathrm{span}\{H_ju\}=\R^m$ at a.e.\ $u$.
\emph{(I$_i$).} Each $\rho_j$ is a quadratic form (plus constant) with matrix $\tfrac12 H_j$ and \emph{no} linear
term; by the Craig--Sakamoto theorem \citep{Li2000craigsakamoto}, under $U\sim\mathcal N(0,I)$,
$\rho_j\perp\rho_k$ iff $H_jH_k=0$. For
$j\neq k$, writing $G_{jk}=\langle P_j,P_k\rangle$, expansion gives
$H_jH_k=e_j\otimes(-t_kP_{kj}P_k)+P_j^\top\otimes(t_jt_kG_{jk}P_k-t_jP_{jk}e_k^\top)$, whose column space lies in
$\mathrm{span}\{e_j,P_j^\top\}$; strong connectivity makes those two generators independent ($P_j\not\parallel
e_j$), so $H_jH_k=0$ forces $P_{kj}=0$ and---as $P_k\not\parallel e_k$---also $G_{jk}=P_{jk}=0$. Thus for $j\neq k$,
$\rho_j\perp\rho_k\iff P_{jk}=P_{kj}=\langle P_j,P_k\rangle=0$, at \emph{every} positive variance. Independence
across a bipartition $E^1\sqcup E^2$ would then force $P_{jk}=P_{kj}=0$ for all $j\in E^1,k\in E^2$, i.e.\ $P$---%
hence $I-B$, hence $B$---block-diagonal along $E^1|E^2$, so no edge crosses the partition: not strongly connected,
a contradiction. The dependence graph (edge $j\!\sim\!k$ when $H_jH_k\neq0$) is therefore connected, and no
bipartition yields independent score tuples: \textup{(I$_i$)} holds. $\qed$

\subsection{Special-case recovery}
\begin{theorem}[Sanity / special cases]\label{thm:special2}
\textbf{(a)} If the latent graph is acyclic, \Cref{thm:linear} reduces to linear interventional CRL
\citep{Squires2023linear} and, for diffeomorphic $h$, \Cref{thm:block} reduces to nonparametric
interventional CRL \citep{vonKugelgen2023nonparam} with no residual hedge (singleton components); the nonlinear
within-block twist requires $m\ge2$, and the cross-block collision requires a rotated block $A$ with
$m_A\ge2$ and a second block $B$ with $m_B\ge1$.
\textbf{(b)} A single full-rotational source component of size $m\ge2$ with no intervention is exactly the observed-layer
equilibrium hedge (\Cref{thm:hedge}): $(h,B)$ is non-identified up to the signed-$\SO(m)$ gauge. \textbf{(c)} The two-prosumer grid market
(price$\rightleftarrows$demand) and an attenuated local behavior--disease SIS response loop
(prevalence$\rightleftarrows$caution) are latent equilibrium generative models. For the latter, let
$f(b)=1-\gamma/[\beta_0(1-b)]$, $i^\star=b^\star=1-\sqrt{\gamma/\beta_0}$ for $\beta_0>\gamma>0$, and
\[
 i=i^\star+\lambda_i\{f(b)-i^\star\}+u_i,\qquad
 b=i^\star+\lambda_b(i-i^\star)+u_b.
\]
Exp.~14 uses centered local deviations in one common declared scale with $(\lambda_i,\lambda_b)=(0.55,0.45)$, whose Jacobian
is $B=\big[\begin{smallmatrix}0&-0.55\\0.45&0\end{smallmatrix}\big]$ because $f'(b^\star)=-1$; its spectral
radius is $\sqrt{0.2475}<1$. For both generators, \Cref{thm:linear}
recovers their cyclic latent state from aligned sensor-response maps with $K=d-1$ and predicts the response map
of a genuinely unobserved mechanism target (\Cref{sec:exp}, Exp.~14).
\end{theorem}

\section{Algorithms}
\label{sec:algo}
\textbf{Testing ECG-separation.} Compute SCCs (Tarjan, $O(n+m)$); evaluate $\sigma$-separation via the
acyclification of \citet{Forre2018} followed by $m$-separation. \textbf{Computing interventional equilibria.}
After a $\doop(\cdot)$ edit, recompute the SCCs and proceed along the condensation DAG: evaluate acyclic
singletons directly and re-solve cyclic components by the Picard best-response iteration of \Cref{thm:exist}.
On a convex normed state space, relaxed updates $x^{k+1}=(1-\alpha)x^k+\alpha T_S(x^k)$,
$0<\alpha\le1$, are also contractions, with modulus $1-\alpha(1-q_S)$.
\textbf{Game specification test.} Given an identified $B$, test $D$-symmetrizability by checking sign-symmetry
($B_{ij}B_{ji}>0$ on the support) and solving $\log d_i-\log d_j=\log|B_{ji}/B_{ij}|$ for the cycle
(Kolmogorov) residual (\Cref{thm:game}). A spanning forest leaves one independent residual for each support
chord, exactly $\beta_1(G)$ in total; a forest has none, whereas a triangle-free graph may still have them. A
sign violation or a nonzero cycle-basis residual \emph{refutes} the potential-game hypothesis.

\Cref{lem:recover} \emph{is} the population estimator; the finite-sample twin route is: (1) within each
environment, estimate the baseline-to-probe mean differences $R_e$ using the labelled calibrated design
$\mathsf D$ of \Cref{thm:linear}, verify $\rank\mathsf D=d$, and set
$\widehat M_0^{(e)}=\widehat R_e\mathsf D^\top(\mathsf D\mathsf D^\top)^{-1}$. This step requires the stated
probe-arm-invariant (or known-corrected) sensor mean and no direct actuator-to-sensor effect; for
$\mathsf D=\Delta I_d$, a shift $\Delta$ on node $k$ moves $\E[\X]$ by
$\Delta M_0^{(e)}e_k$. (2) Form the rank-at-most-one differences $D_j$ and read off each row of $P$ by
\Cref{lem:recover}. (3) Return $B=I-P^{-1}$, $H=M_0P^{-1}$, and
$\widehat\V=\widehat H^{+}\X$. With general $q$, forming the response maps costs
$O((K{+}1)pqd)$ after the fixed right inverse is formed, and the remaining recovery costs $O(Kd^2p)$ plus the
displayed matrix solves (for $q=d$ the acquisition term is again $O((K{+}1)d^2p)$). Passive per-environment ICA
(\Cref{prop:nongauss}) may replace (1) only in the noiseless submodel (or under a separate noisy-ICA theorem for
the declared noise law) and when a separate joint multi-environment theorem or anchors fix a common component
labelling, relative column scales, and the structural source scales; independent ICA plus an outer-product signature
is not itself such an alignment theorem. Once the $M_0^{(e)}$ are aligned, the target set can
be \emph{unknown}: the left factor of $D_j$ equals column $j$ of $M_0$, so each intervened node is identified by
matching.

\section{Experiments}
\label{sec:exp}
The synthetic and calibrated designs for Exp.~1--14 are described below; additional computational information
appears in \Cref{sec:computational}. Experiments are
numbered \emph{Exp.~1}--\emph{Exp.~14}
throughout; unless noted, each reports means with $95\%$ intervals over multiple seeds. Exp.~1--6 corroborate
named finite instances of separation soundness, interventional capability, and the market/epidemic case studies;
Exp.~7--14 form the latent-identifiability sweep.

\subsection{Soundness of the separation criterion}
\textbf{Exp.~1.} \emph{Goal.} Evaluate the finite-instance soundness of ECG/$\sigma$-separation and contrast it
with SCC-blind $d$-separation. \emph{Setup.} We draw $300$ distinct random linear-Gaussian SCMs ($5$--$7$
nodes, randomized SCC count/size; $153$ cyclic and $147$ acyclic), rescale them to
contraction-stable, and enumerate all $>\!84{,}000$ separation triples. Ground truth is the population
covariance from $\Sigma_{\V}=(I-B)^{-1}\Sigma_{\U}(I-B)^{-\top}$. For queried sets $A$ and $B$ conditional
on $C$, we test the $A$-by-$B$ cross-block of the conditional-covariance Schur complement; in this Gaussian
experiment, a zero cross-block is equivalent to conditional independence, not to the full Schur-complement
matrix being zero. A separate randomized linear run makes $34{,}068$ conditional-independence checks.
Distinctly, an exact-law run evaluates $3{,}000$ randomly generated uniquely solvable binary models and the
model-incompleteness counterexample of
\Cref{prop:complete}.
\emph{Key result.} Across all $300$ graphs, ECG/$\sigma$-separation makes \emph{zero} false CIs (precision
$1.00$; \Cref{thm:sound}), whereas $d$-separation errs on $27.3\%$ of graphs (mean false-CI rate $2.1\%$,
precision $0.96$); on the $147$ acyclic graphs the two coincide (\Cref{thm:special}a). The separate randomized
linear run reports $0$ violations. \emph{Theoretical implication.} Together with the named exact-law examples,
these finite calculations corroborate but do not prove the general results.

\begin{table}[t]\centering\footnotesize
\caption{Exp.~1 (separation soundness): $300$ random linear-Gaussian graphs---$153$ cyclic and $147$
acyclic---and $>\!84{,}000$ separation triples. ECG/$\sigma$-separation makes no false
conditional-independence claims, whereas SCC-blind $d$-separation does. The table reports graph-level
incidence, the mean and $[2.5\text{th},97.5\text{th}]$ percentile range of graph-wise false-CI rates, and
pooled precision.}
\label{tab:exp1}
\begin{tabular}{lccc}
\toprule
Criterion & Graphs with $\ge1$ false CI & \shortstack{Graph-wise false-CI rate\\mean [2.5th, 97.5th pct.]} & Pooled precision \\
\midrule
$d$-separation (SCC-blind) & $27.3\%$ & $0.021\ [0,\,0.148]$ & $0.957$ \\
\textbf{ECG / $\sigma$-separation} & $\mathbf{0\%}$ & $\mathbf{0.000}$ & $\mathbf{1.000}$ \\
\bottomrule
\end{tabular}
\end{table}

\subsection{Interventional accuracy and a fair capability test}
\textbf{Exp.~2.} \emph{Goal.} Isolate intervention-awareness and then test it under a same-information
comparison. \emph{Setup.} In the oracle-$B$ setting, a regulator re-wires a quarter of the couplings in an
$N{=}20$ linear-quadratic network game, and the post-intervention $B'$ is given to the ECG. In the fair
setting, both methods see the \emph{same} $K$ interventional environments for $N{=}8$; the ECG must estimate
$B$ by ridge regression of $\bar\V_e$ on $\mu_e$. \emph{Key result.} Given $B'$, the ECG re-solve is exact,
while a learned $\U\!\to\!\V$ map errs by $1.52$ (empirical $2.5$th--$97.5$th percentile range
$[0.49,3.06]$ across networks). Once $K\gtrsim N$ in the fair
setting, the ECG identifies $B$, extrapolates outside the training box (relative error
$0.31\!\to\!0.04$), and answers a new intervention type ($\doop$ remove agent; $0.22\!\to\!0.04$), whereas
an assumption-free k-NN over $\mu\!\to\!\bar\V$ remains at approximately $0.8$--$1.0$. At $K<N$, the
experiment's $K$-by-$N$ shift-design matrix is rank deficient and the structure is under-identified; although
the realized ECG shift error $0.77$ is below k-NN's $1.02$, this finite-run difference is not an
identification-based advantage. \emph{Theoretical implication.} The
oracle result isolates capability rather than a fair contest. In the fair setting, a correctly specified
linear baseline ties on shift queries; the ECG's edge is over assumption-free learners and on new intervention
types. The below-threshold failure here is attributed to the experiment's rank-deficient shift design, not to a
general theorem claim.

\begin{table}[t]\centering\footnotesize
\caption{Exp.~2b (fair same-data contest; $N{=}8$, $n{=}400$ per environment, $15$ seeds): relative error
as $K$ crosses the identification threshold. ECG shift and new-intervention-type errors fall sharply for
$K\ge N$, while the assumption-free k-NN remains inaccurate; the table also reports $\widehat B$ estimation
error and marks the $K<N$ case as under-identified because the $K$-by-$N$ shift-design matrix is rank
deficient.}
\label{tab:exp5}
\begin{tabular}{lcccc}
\toprule
& $K{=}4$ ($<\!N$) & $K{=}8$ ($=\!N$) & $K{=}16$ ($>\!N$) & $K{=}32$ ($\gg\!N$) \\
\midrule
ECG, shift extrapolation & $0.77$ & $0.31$ & $0.077$ & $\mathbf{0.039}$ \\
Flexible learner (k-NN), same data & $1.02$ & $0.97$ & $0.88$ & $0.81$ \\
ECG, new intervention type ($\doop$ remove) & $0.82$ & $0.22$ & $0.087$ & $\mathbf{0.041}$ \\
$\hat B$ relative error (estimation) & under-id.\ & $9.6$ & $0.27$ & $0.15$ \\
\bottomrule
\end{tabular}
\end{table}

\subsection{Fixed-cohort misspecification profile for the linear theory}
\textbf{Exp.~3.} \emph{Goal.} Describe finite-grid misspecification behavior on a fixed common cohort.
\emph{Setup.} Generic completeness for linear trek separation is restricted to Gaussian conditional
independence, or to zero partial covariance under independent finite-variance noise. We let the true DGP be a
nonlinear cyclic equilibrium with quadratic best response $\phi_\alpha(x)=x+\alpha x^2$, fit the linear ECG
from $K=32$ interventional environments, evaluate all $20$ seeds on a declared ten-value grid, and report
new-intervention-type ($\doop$ remove) and $\widehat B$ errors at seven selected values of $\alpha$. A
seed--$\alpha$ run is strictly complete when all $33$ required
fixed-point solves meet the implemented successive-iterate criterion $<10^{-12}$ before $4{,}000$ iterations
and none triggers the nonfinite or $10^8$-norm rejection. Of $20$ attempted seeds, $14$ strictly completed at
every reported $\alpha$; all reported error summaries use this fixed common seven-value cohort.
\emph{Key result.}
Within these $14$ seeds, both mean errors increase monotonically at the sampled grid points. The
new-intervention-type means at $\alpha=.4$ and $.6$, respectively $0.14919\pm0.03054$ and
$0.21548\pm0.04492$ (mean $\pm1.96$ Monte Carlo SE), straddle the declared $.20$ descriptive tolerance at
those evaluated points. Strict-completion counts among all $20$ attempted seeds are reported separately.
At $\alpha=.6$, downstream code accepts $18/20$ outputs, but one is a cap-returned vector and therefore is not
strict, giving $17/20$ here.
\emph{Theoretical implication.} The complete-case summaries and Monte Carlo error bars are conditional on the
selected all-grid-completer cohort. They do not identify a continuous crossing, describe the failed seeds or
full attempted population, establish an advantage or smoothness, diagnose equilibrium nonexistence, or define a
population threshold.

\begin{table}[t]\centering\footnotesize
\caption{Exp.~3 (fixed-cohort misspecification): errors at seven selected values of $\alpha$ for the fixed common cohort
of $14$ seeds that strictly completed at every reported value, reported as mean $\pm1.96$ Monte Carlo SE.
The final column separately gives strict-completion counts among all $20$ attempted seeds. Within the selected
cohort, both mean errors increase monotonically over the sampled grid; the $.4$ and $.6$ new-type means
straddle the declared $.20$ descriptive tolerance at those evaluated points.}
\label{tab:exp6}
\begin{tabular}{cccc}
\toprule
$\alpha$ & \shortstack{new-type relative error\\mean $\pm1.96$ MC SE} &
\shortstack{$\widehat B$ relative error\\mean $\pm1.96$ MC SE} & \shortstack{strict completions\\among $20$ attempts} \\
\midrule
$0$   & $0.04576\pm0.00758$ & $0.14375\pm0.01155$ & $20/20$ \\
$0.1$ & $0.05677\pm0.00924$ & $0.15700\pm0.01649$ & $20/20$ \\
$0.2$ & $0.08456\pm0.01496$ & $0.19097\pm0.02475$ & $20/20$ \\
$0.3$ & $0.11632\pm0.02274$ & $0.23760\pm0.03539$ & $18/20$ \\
$0.4$ & $0.14919\pm0.03054$ & $0.29205\pm0.04816$ & $18/20$ \\
$0.6$ & $0.21548\pm0.04492$ & $0.41683\pm0.08087$ & $17/20$ \\
$1.0$ & $0.34967\pm0.06965$ & $0.71827\pm0.17357$ & $14/20$ \\
\bottomrule
\end{tabular}
\end{table}

\subsection{The game layer is testable and gauge-reducing}
\textbf{Exp.~4.} \emph{Goal.} Instantiate the falsifiable game restriction and the supported identification
scope of \Cref{thm:game}. \emph{Setup.} For test power, we compare potential-game $B$, generic cyclic $B$,
and sign-symmetric but non-Kolmogorov $B$. For identification, the tested labelled two-node subclass combines
exact potentiality with equal own curvatures, hence $B=B^\top$, diagonal positive disturbances, and
contraction-stability selection. We also examine complete-support, regular cases with $m=3,4$ and a tested
equal-codimension non-game restriction given by a known coefficient ratio. \emph{Key result.} The
$D$-symmetrizability/Kolmogorov residual is $\le8\times10^{-16}$ for potential-game $B$. For generic cyclic
$B$, the reported $1.0$ is the implementation's sign/support-failure sentinel rather than a Kolmogorov
residual. Conditional on complete reciprocal sign-symmetric support, the actual Kolmogorov residual medians
are $0.34$ ($m=3$) and $0.75$ ($m=4$).
In the two-node subclass with nonzero cross-covariance, the covariance-compatible symmetric fibre is the
reciprocal pair $\{b,1/b\}=\{0.40,2.50\}$; strict contraction retains $0.40$ and identifies the reported
effect within this subclass (numerical symmetric-gauge width $1.7\times10^{-16}$), versus the generic hedge's
$[-1.0,1.1]$. With zero cross-covariance the compatible solution is $b=0$, not a reciprocal pair. In the
tested complete-support, regular $m=3,4$ cases, the gauge shrinks from $m(m-1)/2$ to $m-1$ (constraint ranks
$1,3$). The tested non-game restriction also identifies the effect (width $2.5\times10^{-16}$), and direct
computer algebra gives $DB-(DB)^\top=0$. \emph{Theoretical implication.} On the tested complete reciprocal
sign-symmetric supports at $m=3,4$, a nonzero cycle-basis residual---not sign symmetry alone---provides
rejection. The cycle test is nonvacuous only when $\beta_1(G)>0$; on forests, including reciprocal $m=2$
support, only support/sign violations can refute the restriction.

\subsection{Market and epidemic case studies}
\textbf{Exp.~5 (Cournot electricity market).} \emph{Goal.} Illustrate structure-based prediction of an unseen
cost intervention in a calibrated Cournot market. \emph{Setup.} We anchor a linear wholesale-power demand
curve to a short-run elasticity of $-0.30$ at a $\$40$/MWh, $300$\,MWh reference point and use heterogeneous
merit-order marginal costs $\{20,25,30,35,45\}$\,\$/MWh \citep{BorensteinBushnell1999}. The resulting Cournot
equilibrium clears at approximately $\$55$/MWh (elasticity approximately $-0.46$ there), so the reference
anchor pins the demand curve, not the equilibrium. With linear demand $P(Q)=A-bQ$ and constant marginal cost,
$q_i=(A-c_i)/(2b)-\tfrac12\sum_{j\neq i}q_j$: an ECG with $B_{ij}=-\tfrac12$ and a potential game. Symmetric
interaction implies that $B$ is symmetrizable, so the Exp.~4 specification test fires; $\det(I-B)\neq0$ gives
a unique Nash even though $\rho(B)=(N-1)/2>1$ for $N\ge4$, illustrating \Cref{thm:exist}. We hold out the
carbon tax $\doop(c_i{\leftarrow}c_i{+}\tau e_i)$ with $\tau=30$/ton and emission intensities $e_i$.
\emph{Key result.} Because total output and price depend on aggregate marginal cost $\sum_i c_i$, while output
shares depend on the cost vector, $\sum_i\tau e_i\neq0$ makes the tax lower total output, raise price, and
reshuffle the merit order. Total emissions fall $194.6\!\to\!151.4$ ($-22.2\%$), with both the
aggregate-output and merit-order channels contributing. The ECG re-solves the Nash equilibrium exactly without
cost-varying data (per-firm output MAE $=0$ to machine precision); a cost-blind correlational twin trained only
on demand shocks predicts essentially no reshuffle (MAE $13.5$ MWh), while a cost-aware learner given
cost-varying training data also recovers it (MAE approximately $10^{-13}$). \emph{Theoretical implication.}
This is an illustrative calibrated simulation, not real data: the functional form and constants are realistic,
but the equilibria are simulated. Its value is structure in lieu of cost-varying data, the capability quantified
in Exp.~2, not superiority over a fairly trained learner.

\begin{figure}[t]\centering
\includegraphics[width=0.42\textwidth]{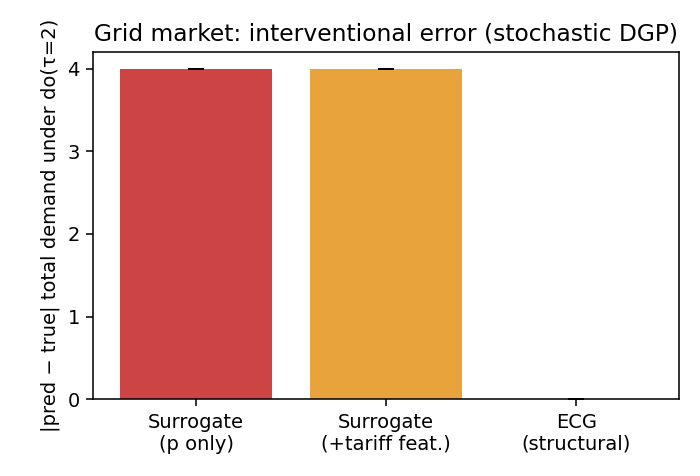}\hfill
\includegraphics[width=0.55\textwidth]{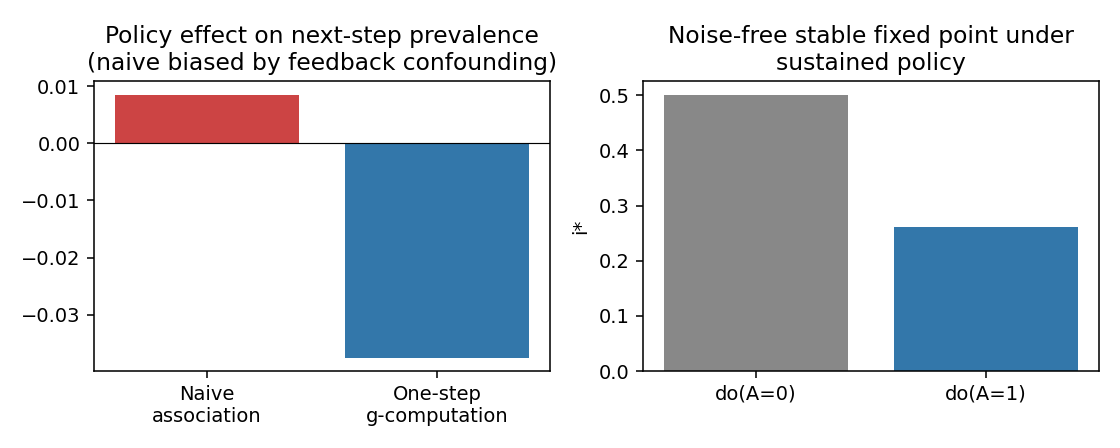}
\caption{Exp.~6 (interventional and counterfactual prediction). \textbf{(a)} In the grid market, the
correlational surrogate errs by $4.0$ demand units under an unseen tariff, whereas the structural ECG is exact.
\textbf{(b)} In the behavior--disease loop, the unadjusted one-step association has the wrong sign relative to
the positivity-supported, time-pooled one-step $g$-computation estimate. \textbf{(c)} Separately, the positive
stable fixed points of the noise-free clipped map reached from $i_0=0.3$ under sustained structural
interventions move $0.50\!\to\!0.26$.}
\label{fig:exp23}
\end{figure}

\paragraph{Exp.~6 (interventional and counterfactual prediction; grid and epidemic).} \emph{Goal.} Compare
structural and associational predictions for unseen interventions in two cyclic case studies. \emph{Setup.}
A two-prosumer grid market recovers $\theta$ from one equilibrium and is queried under the unseen tariff
$\doop(\tau=2)$. A behavior--disease SIS loop with an endogenous policy creates treatment--confounder
feedback. The assignment uses symmetric $5\%$ exploration around its prevalence-threshold policy: pointwise
in the current prevalence,
$\Prob(A_t{=}1\mid i_t)=0.05$ for $i_t\le0.45$ and $0.95$ for $i_t>0.45$, so each action has probability at
least $0.05$. In the $n=18{,}000$ seed-$3$ transitions,
$\widehat\Prob(A_t{=}0\mid i_t{>}0.45)=0.053$ and
$\widehat\Prob(A_t{=}1\mid i_t\le0.45)=0.053$; across $20$ additional seeds with the same $18{,}000$
transitions each, the empirical $2.5$--$97.5\%$ ranges are $[0.045,0.053]$ and $[0.048,0.052]$.
Conditional on the fully observed pre-action $i_t$, the exploration coin is independent of every transition
disturbance and has no direct effect on $i_{t+1}$ except through $A_t$; together with consistency, this identifies
the equal-time-pooled one-step target
\[
 \psi_{\mathrm{pool}}:=\frac1{60}\sum_{t=0}^{59}
 \E_{i_t\sim P_t^{\mathrm{beh}}}\!\left[m_1(i_t)-m_0(i_t)\right],\qquad
 m_a(i):=\E(i_{t+1}\mid A_t{=}a,i_t{=}i).
\]
We fit the coefficients and noise scale of a clipped-Gaussian outcome model using the simulator-known SIS basis
$q(i,a)=[1-\min\{1,i+0.4a\}](1-i)i$ and evaluate the exact clipped conditional mean. No transition is censored
in these runs, so the coefficient part of the Tobit likelihood reduces to ordinary least squares.
\emph{Key result.} The grid ECG predicts $Q_{\mathrm{tot}}=6$ exactly, while the correlational surrogate errs
by $4.0$ (\Cref{fig:exp23}, panel~(a)). In the SIS loop, the naive one-step association is $+0.009$ (wrong
sign), while the back-door/$g$-computation estimate is $-0.0376$ (exact clipped conditional-response oracle
under empirical standardization $-0.0374$; $20$-seed empirical $2.5$--$97.5\%$ range
$[-0.0375,-0.0373]$; \Cref{fig:exp23}, panel~(b)). Separately, iterating the known noise-free clipped
transition from $i_0=0.3$ under sustained $\doop(A_t=0)$ or $\doop(A_t=1)$ reaches positive stable fixed
points $0.50\!\to\!0.26$ (\Cref{fig:exp23}, panel~(c)). \emph{Theoretical implication.} The SIS estimate is
a synthetic, model-informed one-step result, not a generic black-box outcome-regression validation or a full
longitudinal regime effect; the structural fixed-point benchmark is not inferred by the preceding observational
adjustment.

\subsection{Latent identifiability: threshold, complexity, and recovery quality}
The symbolic and finite-instance calculations associated with these experiments are summarized in
\Cref{sec:computational}; they corroborate the named examples and do not prove the general results. Synthetic
cyclic-latent generators use block-structured $B$ rescaled to $\rho(B)=0.6$,
LiNG (Laplace) sources, and random full-column-rank $H$ ($p=d{+}2$). We report means with $95\%$ CIs over
$10$--$20$ seeds; runtime is seconds on a laptop CPU.

\paragraph{Exp.~7: identifiability threshold (the closed-form estimator).} \emph{Goal.} Locate the recovery
threshold for the closed-form estimator. \emph{Setup.} For one full-rotational source component of size $d=4$,
we vary the number $K$ of $\doop$-covered nodes. \emph{Key result.} The closed-form row-wise estimator
(\Cref{lem:recover}, which uses full coverage) has gauge-dominated error $\|\widehat B-B\|_F$ below threshold,
then collapses at $K=d$: $1.244\to1.110\to0.959\to0.706\to\mathbf{0.042}$ for $K=0,\ldots,4$; population
recovery is exact to $7\times10^{-15}$, and the within-component gauge residual collapses to $0$
(\Cref{fig:threshold}, panels~(a) and~(b)). \emph{Theoretical implication.} The finite recovery profile
corroborates the full-coverage construction in \Cref{lem:recover}.

\paragraph{Exp.~8: information-theoretic minimal $K$.} \emph{Goal.} Distinguish the intervention threshold on
universal-parent supports from that on supports without a universal parent. \emph{Setup.} A parametrized
collision search asks whether a non-$\simeq$-equivalent $(H',B')$ reproduces every environment's mixing on
random instances; finite completion-rank checks cover dimensions $3,4,5,6$ and include an explicit
three-dimensional collision family. \emph{Key result.} For a single size-$d$ source component on
universal-parent supports, a continuum of alternatives exists at $K=d-2$ (non-identified; $4/4$ graphs at
$d\in\{3,4\}$), while none exists at $K=d-1$ on the random dense instances searched. By contrast, the
closed-form completion characterization implies that a directed ring of any size $d\ge3$, which has no
universal-parent node, collides at $K=d-1$ and therefore has threshold $K=d$. \emph{Theoretical implication.}
The searches and completion-rank checks corroborate the support-dependent conclusion and the generic result in
\Cref{thm:linear}(c); they are finite calculations rather than proofs.

\paragraph{Exp.~9: sample complexity.} \emph{Goal.} Assess the predicted above-threshold $1/\sqrt n$ recovery
rate. \emph{Setup.} The $d=3$ simulation uses the special full-rank design $\mathsf D=3I_d$, no direct
probe-to-sensor effect, and zero-mean arm-invariant simulated sensor noise. \emph{Key result.} Error falls from
$0.14926$ at $n=500$ to $0.012386$ at $n=32\text{k}$, a $12.05\times$ drop over a $64\times$ range,
calculated from the unrounded seed means. The endpoint-implied exponent is $0.59852$, versus the asymptotic
$1/2$; the difference is finite-sample deviation, while the fixed conditioning constant governs the error
level. \emph{Theoretical implication.} These instances are consistent with \Cref{thm:linear}(d) and the
$1/\sqrt n$ prediction (\Cref{fig:threshold}, panel~(c)), but instantiate only the calibrated design and do not
test robustness to probe-dependent sensor bias.

\begin{figure}[t]\centering
\includegraphics[width=0.66\textwidth]{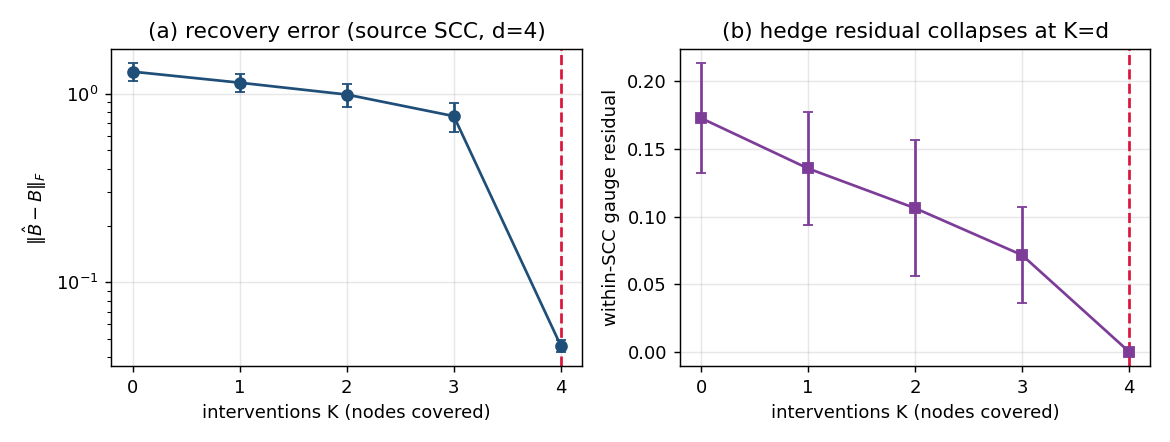}\hfill
\includegraphics[width=0.32\textwidth]{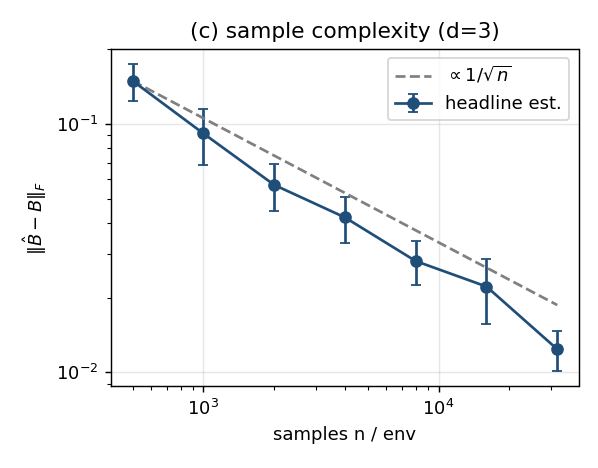}
\caption{Exp.~7 and Exp.~9 (threshold and sample complexity). \textbf{(a)} For Exp.~7, recovery error versus
interventions $K$ in a size-$d=4$ full-rotational source component is gauge-dominated below threshold and
collapses at $K=d$ (dashed). \textbf{(b)} Also for Exp.~7, the within-component hedge residual collapses at
$K=d$. \textbf{(c)} For Exp.~9, above-threshold sampled errors at $d=3$ are consistent with the predicted
$1/\sqrt n$ decay without establishing its asymptotic exponent. All panels show $95\%$ CIs over $12$ seeds
(\Cref{thm:linear}).}
\label{fig:threshold}
\end{figure}

\paragraph{Exp.~10: the hedge and non-Gaussianity.} \emph{Goal.} Test whether non-Gaussian information breaks
the covariance-preserving rotational hedge. \emph{Setup.} For a full-rotational source component, we apply a
signed-$\SO(m)$ rotation under Gaussian and non-Gaussian source regimes. \emph{Key result.} The rotation leaves
the covariance invariant in both regimes ($\|\Delta\Sigma\|\approx0.02$), but the rotated sources
$Q^\top\bm\eta$ have fourth-order dependence $0.009$ in the Gaussian case and $0.346$ in the non-Gaussian
case. \emph{Theoretical implication.} The Gaussian hedge remains non-identifying, whereas non-Gaussianity
rules it out in this finite instance, corroborating \Cref{prop:nongauss}.

\paragraph{Exp.~11: metrics and baselines.} \emph{Goal.} Quantify recovery and isolate the contributions of
cyclic modeling and interventional information. \emph{Setup.} We use $12$ seeded graphs, each with a complete
directed three-node source SCC and a fourth singleton. Two controlled baselines isolate the two ingredients:
(i) an acyclic-but-interventional method receives the same interventional data but may output only a DAG, as in
acyclic CRL; (ii) an observational FastICA surrogate, implemented as PCA+FastICA, uses observational data
without interventions. We also include a
$\sigma$-separation-only ablation. \emph{Key result.} The closed-form estimator attains MCC
$\mathbf{1.000}$, SHD-up-to-equivalence $\mathbf{0.00}$, and gauge residual $\mathbf{0.000}$. Forcing the
correctly recovered cyclic $\widehat B$ to a DAG deletes exactly three of the SCC's six directed edges---one
from each reciprocal pair---and incurs SHD $\mathbf{3.00}$. The observational FastICA surrogate has MCC $0.85$,
and the graph is non-identifiable. \emph{Theoretical implication.} The acyclic baseline measures the cost of
the acyclicity assumption, not a weak implementation; the observational-only FastICA result agrees with
\Cref{thm:gauge}. The sound but incomplete separation-only ablation cannot orient within-component edges,
consistent with the completeness calibration above.

\paragraph{Exp.~12: intervention strength.} \emph{Goal.} Measure sensitivity to intervention magnitude.
\emph{Setup.} We vary the shift magnitude from $\delta=0.5$ to $8$. \emph{Key result.} Recovery error improves
monotonically from $0.149$ to $0.011$. \emph{Theoretical implication.} The pattern is the expected
signal-to-noise behavior.

\paragraph{Exp.~13: nonlinear block-ID collisions and the isotropy-breaking boundary.} \emph{Goal.} Test
partition recovery under isotropy-breaking interventions and contrast it with the isotropy-preserving
collision. \emph{Setup.} The positive setting is deliberately special:
$h(v)=W(v+0.3\sin v)$ is coordinate-block-separable up to an unknown random permutation of its four observed
coordinates, with two independent source blocks, each of dimension $2$. Two labelled per-coordinate variance interventions act inside
one block. A response-profile learner receives only observational and interventional $X$ samples, scores each
observed coordinate by its largest absolute log-variance response, and splits the ordered scores at their
largest gap; the generator partition is withheld until scoring. \emph{Key result.} With $n=800$ per discovery
arm, the learner recovers the exact responsive-block/invariant-complement partition on $20/20$ seeds (minimum
response-score gap $0.956$). On an independent seed-$0$ held-out sample ($n=800$ per arm), the learned
responsive group changes (energy-test $p=0.007$), its learned complement remains invariant ($p=0.159$), and
their distance-correlation is $0.068$. Under labelled isotropy-preserving interventions, a within-component
rotation in either size-$2$ block reproduces the block-interventional law (relative gap approximately $2\%$) with a different within-component
factorization (gap $0.97$), and a $B$-dependent cross-component rotation preserves the joint law while
re-partitioning. \emph{Theoretical implication.} Isotropy-breaking supports truth-blind partition recovery in
this block-separable, strict-response-gap setting, but does not establish recovery under a general dense
diffeomorphic mixing. The collisions of \Cref{thm:block} leave even the partition non-identified under
isotropy preservation; the general positive remains open.

\paragraph{Exp.~14: semi-real cyclic generators (12 seeds, mean $\pm1.96$ Monte Carlo SE).} \emph{Goal.} Test
prediction for a held-out mechanism target---not a new intervention type---in two calibrated cyclic generators.
\emph{Setup.} For each stable two-cycle, training contains the observational aligned response map and exactly
$K=d-1=1$ labelled perfect-mechanism map, each acquired with the same calibrated full-rank shift-probe design;
no response from the held-out target is supplied. The held-out node is price in the grid market
(price$\rightleftarrows$demand, $\rho(B)=0.55$) and caution in the attenuated local behavior--disease SIS
response loop ($i^\star=b^\star=0.5$). The latter's simulated latent coordinates are centered local deviations
in one common declared scale under the response equations of \Cref{thm:special2}, not raw prevalence/caution
proportions. Each held-out node directly parents the other, so the graph-global completion condition holds; the
minimum sampled $\sigma_{\min}/\sigma_{\max}$ of the completion system is $0.222/0.514$. Both predictors
receive identical labelled estimated maps. The declared response-only label-exchange stress comparator fits the
theorem-implied right rank-one signature by least squares and transfers it under an exchange of the two target
labels, without recovering $H$ or $B$. \emph{Key result.} The $K=d-1$ structural completion recovers the grid
and SIS loops to $\|\widehat B-B\|=2.25\times10^{-2}\pm9.0\times10^{-3}$ and
$1.38\times10^{-2}\pm3.8\times10^{-3}$, respectively. The comparator stores the available mechanism map
verbatim, so its zero training lookup error is by construction rather than evidence that the finite-sample
difference is exactly rank one; its mean signature-fit residual is $0.0124\pm0.00735$ in the grid setting and
$0.0151\pm0.00831$ in the SIS setting. On the full response map of the genuinely unobserved target, relative
Frobenius error is $0.0179\pm0.0078$ versus $0.105\pm0.020$ in the grid setting, and
$0.00851\pm0.00252$ versus $0.561\pm0.072$ in the SIS setting. An anchored minimum-norm target table, reported
only as a no-target-sharing diagnostic, gives $0.689\pm0.101$ and $0.288\pm0.0376$; its missing coefficient is
fixed by regularization and is not presented as a learned extrapolator. A clamp of the held-out node, in the
common local-deviation scale for the SIS setting, uses only one map column and is matched by the
response-transfer rule in these two-node loops, so it is not used to claim an advantage.
\emph{Theoretical implication.} These are finite-instance results for the two named calibrated generators and
the specified same-information stress comparators. They show improved full-map prediction over these declared
rules in the named settings, not universal superiority over nonstructural predictors and not a proof of
\Cref{thm:linear}.

\section{Limitations and broader impacts}
\label{sec:limits}
\textbf{Limitations.}
\begin{enumerate}[leftmargin=1.6em,itemsep=3pt,topsep=2pt]
\item \textbf{Well-posedness at both layers.} Every well-posedness check covers all cyclic SCCs (singleton
self-cycles included) after recomputing post-edit SCCs. Graphical soundness additionally requires that every
exogenous dependence be represented and the sufficient strongly-connected-subset solvability premise used by
the global Markov theorem; graphical adjustment and separation-based identification routes inherit these
requirements. At the latent layer contraction $\rho(B)<1$ enforces the selected stable class;
non-uniquely-solvable games need a declared selection rule $\mathsf{Sel}$ and are only partially covered.

\item \textbf{Observed-layer completeness.} ECG-separation is sound but \emph{not} complete (as shown by the
model-incompleteness construction). For practice we provide trek-separation---generically complete for Gaussian
CI, and generically complete for second-order (zero-partial-covariance) structure under non-Gaussian noise, on
\emph{linear} ECGs---and a sound smooth-case criterion. A complete CI criterion beyond the linear class stays
open, as does the exact \emph{iff} between our sufficient-ID and sufficient-non-ID (hedge) sides.

\item \textbf{Latent-layer completeness.} The clean \emph{iff} is for the zero-diagonal linear/LiNG class. In
the nonlinear case, for an isotropic-Gaussian source block $S\subset\R^m$, $m\ge2$, labelled
isotropy-preserving interventions leave its within-block factorization unidentified. For two isotropic blocks
$A\in\R^{m_A}$, $B\in\R^{m_B}$ with $m_A\ge2$, $m_B\ge1$, labelled interventions preserving the required
radial/isotropic product law also leave the block partition unidentified (\Cref{thm:block}). The proved positive identifies only independent source blocks under the
assumptions of \Cref{thm:blockid-source}. Positive recovery of equilibrium-component partitions or factors, or
of within-component structure, under additional restrictions beyond that source-block theorem remains open.

\item \textbf{Information cost and transversality.} In the ambient unknown-support, unknown-$H$ class, the
active route needs aligned, well-posed \emph{mechanism} actuation---for $d\ge2$, shifts never suffice, whatever
their number (\Cref{thm:linear,prop:shift})---and the passive route additionally needs non-Gaussianity
(\Cref{prop:nongauss}). The two obstructions are transverse, and the substitution is \emph{one-way}
(\Cref{rem:transverse}): non-Gaussianity provably cannot collapse the legal fixed-$M$ frame chart---every chart
member fixes the source map, so the laws agree to all orders. The active route uses calibrated full-rank shift
probes under mean-stable sensing to acquire aligned response maps and sufficiently rich well-posed mechanism
environments to split $H$ from $B$; for $d\ge2$, shift probes alone never suffice. At $d=1$, zero diagonality
already fixes $B=0$. Those mechanism responses also collapse the hedge. In other declared information classes,
known support, known/anchored $H$, or calibration restrictions can shrink or pin the legal fibre without
mechanisms. Passive ICA is an alternative acquisition route only in the noiseless submodel (or under a separate
noisy-ICA identifiability theorem for the declared noise law) and under a separate cross-environment alignment
theorem. Non-Gaussianity is the cheaper cure but buys strictly less.

\item \textbf{Game-layer scope.} \Cref{thm:game} covers the \emph{potential-game} subclass. Its game-essential
content is the refutable specification test. The tested two-node point-identification result is for the
symmetric-$B_S$ subclass (equal own-curvatures for an exact potential), with contraction selection when the
data-generating equilibrium is itself contraction-stable. On complete nonzero reciprocal support, generically
on the regular set, the general weighted-potential restriction leaves an $(m{-}1)$-dimensional fibre for
$m\ge3$. No such formula is asserted on sparse support without its actual tangent/rank analysis, and the cycle
restriction is vacuous on forests. The identification gain is a generic codimension consequence, not a
game-unique mechanism.

\item \textbf{Non-Gaussian \emph{iff} scope.} The functional-separation result is a sufficient condition for
exact non-Gaussian CI (\Cref{thm:funcsep}). Under the a.e.\ conditional analyticity and
moment-determinacy assumptions of \Cref{app:cumulant}, vanishing of the complete mixed conditional-cumulant
hierarchy is also an exact \emph{distributional} criterion. No non-Gaussian graphical iff is claimed, and
neither the linear result nor block-ID soundness depends on one.

\item \textbf{External-validation boundary.} Exp.~1--14 are synthetic or calibrated studies. They corroborate
their named finite instances but do not establish operational effectiveness in real energy markets or
public-health systems. No high-fidelity simulator or real-data analysis is used as evidence in this paper; an
applied claim would require a complete, fully specified data-generation or extraction pipeline that an
independent team can rerun end-to-end,
intervention construction, task-level outputs, and declared aggregation and uncertainty analysis.
\end{enumerate}

\paragraph{Open questions and conjectures.}
\begin{enumerate}[leftmargin=1.6em,itemsep=2pt,topsep=2pt]
\item \textbf{Complete observed-layer criteria.} A complete CI criterion beyond the linear class and the exact
\emph{iff} between our sufficient-ID and sufficient-non-ID (hedge) sides remain open. Under the assumptions of
\Cref{app:cumulant}, the complete mixed conditional-cumulant hierarchy is already an exact distributional
criterion; what remains open is a graph characterization beyond the functional-separation sufficient condition.
\item \textbf{General nonlinear block identification.} For an isotropic-Gaussian source block
$S\subset\R^m$, $m\ge2$, labelled isotropy-preserving interventions leave its within-block factorization
unidentified. For two isotropic blocks $A\in\R^{m_A}$, $B\in\R^{m_B}$ with $m_A\ge2$, $m_B\ge1$, labelled
interventions preserving the required radial/isotropic product law also leave the block partition unidentified
(\Cref{thm:block}). The proved positive in
\Cref{thm:blockid-source} identifies independent source blocks only under its assumptions. Positive recovery of
equilibrium-component partitions or factors, or of within-component structure, under additional restrictions
beyond that source-block theorem remains open.
\item \textbf{Game-layer point identification.} On the symmetric-$B_S$ subclass, the fibre is generically
finite for every $m$; stability yields uniqueness at $m{=}2$, while uniqueness for $m\ge3$ remains conjectural.
General weighted-potential models retain the stated $(m{-}1)$-dimensional fibre on complete nonzero reciprocal
support, generically on the regular set, rather than conjectural uniqueness (\Cref{thm:game}).
\item \textbf{High-dimensional finite-sample rates.} A dimension-explicit or uniform-in-$d$ rate requires
additional nondegenerate-covariance and matrix-concentration or operator-norm assumptions and is left open.
\item \textbf{External validation.} Operational effectiveness in real energy markets or public-health systems
requires the applied validation described above and remains open.
\end{enumerate}

\paragraph{Broader impacts.} Sound interventional reasoning on feedback systems---and a criterion establishing that a
closed-loop twin's latent state is genuinely causal rather than merely predictive---can improve energy-market
and public-health decision support. The same machinery could be used to design manipulative interventions, and
equilibrium models can encode biased payoffs; deployments should therefore be paired with validation against
held-out interventions.

\section{Computational details}
\label{sec:computational}
Selected symbolic and finite-sample calculations corroborate the examples in \Cref{sec:exp} and the accompanying theorem discussions; complete proofs are provided in \Cref{app:proofs}. Code, data, and computational artifacts underlying the reported results are available from the author upon reasonable request

\bibliographystyle{plainnat}
\bibliography{ecg-c1-references}

\appendix
\section{Proofs}
\label{app:proofs}

\paragraph{\Cref{thm:exist}.} The condensation graph of the SCCs is a DAG. For a cyclic component, after holding
its external coordinates fixed, Banach gives a unique fixed point on its fixed nonempty complete component
space. Starting from the fixed seed, the Picard iterates are jointly Borel by induction and converge pointwise
to that fixed point on the Borel full-measure set $\mathcal U_0$; defining the map to equal the seed outside
$\mathcal U_0$ gives a jointly measurable a.e.\ component solution map in the held coordinates and exogenous
input. Thus every cyclic SCC is uniquely solvable w.r.t.\ itself; an acyclic singleton equation contains no
self-dependence and is uniquely solved w.r.t.\ itself by direct evaluation of its jointly Borel structural
assignment. Proceeding in a topological order of the condensation DAG, every
external parent of the current component has already been solved; finite composition therefore gives a jointly
Borel global solution on $\mathcal U_0$, and componentwise uniqueness gives global uniqueness there.
Linear: $\bm v=(I-B)^{-1}\bm u$, unique iff
$\det(I-B)\neq0$; Borel coefficient/forcing maps give a Borel inverse solution on the nonsingular domain, while
the simultaneous iteration converges to the fixed point from every initialization iff $\rho(B)<1$. These differ (grid component:
$\det(I-B)=1+2c>0$ but $\rho=\sqrt{2c}$, which exceeds one exactly when $c>1/2$). $\qed$

\paragraph{\Cref{thm:sound}.} Replace every shared exogenous source by an explicit latent common-parent node.
The remaining primitive source blocks are mutually independent, so the augmented graph represents the complete
stochastic factorization. The $\mathsf{Eq}$ fixed-point conditions are structural equations whose mutual
dependence is the equilibrium SCC, and the assumed unique solvability for every strongly connected subset is
the sufficient loop-solvability premise used here for the generalized directed global Markov property
\citep{Bongers2021,Forre2018}. Hence $\sigma$-separation in the augmented graph implies the stated conditional
independence. $\qed$

\paragraph{\Cref{prop:complete}.} Soundness is \Cref{thm:sound} (computationally checked). (i) With $V_0=U_0$,
$V_1=U_1$, $V_2=f_2(V_0,V_3,U_2)$, $V_3=f_3(V_1,V_2,U_3)$ and independent private blocks
$(U_0,U_2)\indep(U_1,U_3)$, conditioning on
$V_2{=}x,V_3{=}y$ solves the $2$-cycle for $(U_2,U_3)$: $U_2=\phi_2(x,y;V_0)$, $U_3=\phi_3(x,y;V_1)$. The change
of variables $(U_2,U_3)\!\mapsto\!(V_2,V_3)$ at fixed $(V_0,V_1)$ carries the cycle Jacobian
\[
 \left|\frac{\partial(U_2,U_3)}{\partial(V_2,V_3)}\right|
 =\frac{|1-\mathfrak L|}{|f_{2,U_2}f_{3,U_3}|},
 \qquad
 \mathfrak L=\frac{\partial f_2}{\partial v_3}\frac{\partial f_3}{\partial v_2}.
\]
For a single inverse branch the joint
density factors as
\begin{align*}
p(V_0,V_1,V_2{=}x,V_3{=}y)
 &=\underbrace{\frac{p_{U_0,U_2}(V_0,\phi_2)}{|f_{2,U_2}|}}_{(V_0)\text{-side}}
   \underbrace{\frac{p_{U_1,U_3}(V_1,\phi_3)}{|f_{3,U_3}|}}_{(V_1)\text{-side}}
   |1-\mathfrak L|.
\end{align*}
(the side factors are $\sigma(V_2,V_3)$-measurable in their own noise), so $V_0\indep V_1\mid V_2,V_3$
\emph{iff $|1-\mathfrak L|$ is multiplicatively $(V_0;V_1)$-separable given $(V_2,V_3)$} --- automatic whenever
$\mathfrak L$ is $\sigma(V_2,V_3)$-measurable along the solution (the linear/additive class, constant
$\mathfrak L$). With finitely many inverse branches, the same conclusion holds only when the complete
branch-summed coarea weight is multiplicatively separable; branchwise separability alone is insufficient. A
\emph{separate} indicator-factorization gives the finite-state conclusion. Indeed, at fixed $(v_2,v_3)$ let
\begin{align*}
 L(v_0;v_2,v_3)&=\sum_{u_0,u_2}p_{02}(u_0,u_2)\mathbf 1\{v_0=u_0,\ v_2=f_2(v_0,v_3,u_2)\},\\
 R(v_1;v_2,v_3)&=\sum_{u_1,u_3}p_{13}(u_1,u_3)\mathbf 1\{v_1=u_1,\ v_3=f_3(v_1,v_2,u_3)\}.
\end{align*}
Private-block independence gives
$p(v_0,v_1,v_2,v_3)=L(v_0;v_2,v_3)R(v_1;v_2,v_3)$; normalizing over $(v_0,v_1)$
therefore gives $V_0\indep V_1\mid V_2,V_3$ at every conditioning value of positive probability. In those
cases the walk
$V_0\!\to\!V_2\!\to\!V_3\!\leftarrow\!V_1$ is $\sigma$-open, refuting completeness. For this four-node
construction, an exact-law finite run finds the stated conditional independence in $4000/4000$ sampled consistent
binary ECGs, and direct calculation gives the same conditional independence for the linear/additive continuous
class; these calculations illustrate the preceding
proof. It \emph{fails} when $\mathfrak L$ couples the sides:
$V_2{=}V_0V_3{+}U_2$, $V_3{=}V_1V_2{+}U_3$ has $\mathfrak L{=}V_0V_1$, $|1-V_0V_1|$ non-separable, so
$V_0\not\indep V_1\mid V_2,V_3$ (a numerical evaluation gives conditional
$\mathrm{corr}\approx-0.19$)---the earlier ``for every $f_2,f_3$'' would be false.
(ii)/(iii) as stated. $\qed$

\paragraph{\Cref{thm:special}.} (a) Acyclic $\Rightarrow$ singleton SCCs $\Rightarrow$ $\sigma$-exception never
fires $\Rightarrow$ $d$-separation; exhaustive finite enumeration gives the same conclusion on the enumerated
class. (b) Setting $\mathsf{Eq}=$ logit-QRE and reading
\eqref{eq:qre} as structural equations reproduces MA-AIRL. (c) Myopic $b{=}i$, $\dot I=0$ gives
$i(\beta_0(1-i)^2-\gamma)=0$, interior root $i^\star=1-\sqrt{\gamma/\beta_0}$ when
$\beta_0>\gamma>0$ (symbolic). $\qed$

\paragraph{\Cref{prop:counter}.} Acyclification places $\{q_1,q_2\}$ in one bidirected component, so
$q_1\not\ecgsep q_2\mid\{\theta_1,\theta_2\}$; because the disturbances are independent of
$(\theta_1,\theta_2)$ and each other with common variance, conditioning on $(\theta_1,\theta_2)$ gives by
direct calculation $\mathrm{pcorr}=-2\beta/(1+\beta^2)=-0.80$ at $\beta=1/2$. $\qed$

\begin{lemma}[No cancellation in the all-minors expansion]\label{lem:nocancel}
Let every present entry $B_{ij}$, including a permitted diagonal self-loop, be an algebraically independent
indeterminate. For index sets with $|R|=|C|$, the all-minors (Coates) expansion
$\det((I-B)_{R,C})=\sum_{\varphi:C\to R}\mathrm{sgn}(\varphi)\prod_{c\in C}(\delta_{\varphi(c),c}-B_{\varphi(c),c})$
runs over bijections $\varphi$. At a moved column the factor is $-B_{\varphi(c),c}$; at a fixed point it is
$1-B_{cc}$. Expand the product further by a subset $F$ of the fixed points, selecting $-B_{cc}$ for $c\in F$
and $1$ otherwise. Each resulting squarefree monomial records both the moved-arc set
$E(\varphi)=\{(\varphi(c),c):\varphi(c)\neq c\}$ and the selected self-loop set $F$. The moved arcs determine
$\varphi$ on every moved column; all remaining columns are fixed, and their diagonal factors determine $F$.
Thus $(\varphi,F)\mapsto$ monomial is injective. Distinct expanded linear subgraphs give distinct monomials,
no two terms cancel, every coefficient is $\pm1$, and $\det((I-B)_{R,C})\not\equiv0$ iff an admissible
permutation uses only present moved arcs (with any permitted self-loops recorded by $F$). Setting absent
diagonal entries to zero recovers the zero-diagonal special case.
\end{lemma}

\paragraph{\Cref{thm:trek} (complete proof).}
\emph{Step 1 (rank form --- airtight).} For Gaussian $\V$, $\V_{\bm A}\indep\V_{\bm B}\mid\V_{\bm C}$ iff
$\operatorname{rank}\Sigma_{A\cup C,\,B\cup C}=|C|$. \emph{Step 2 (Cauchy--Binet).} With $P=(I-B)^{-1}$,
$\Sigma_{A,B}=P_A\Omega P_B^{\!\top}$, and every $k\times k$ minor of $\Sigma_{A,B}$ equals
$\sum_{|S|=k}(\prod_{s\in S}\Omega_{ss})\det(P_{A',S})\det(P_{B',S})$; generically nonzero iff some $S$ has both
factors $\not\equiv0$. \emph{Step 3 (linkage$=$rank, self-contained; the cyclic ``LGV'' step).}
\textbf{(a) Jacobi complementary minors} \citep{HornJohnson2013}: for invertible $M$,
$\det(M^{-1}_{I,J})=(-1)^{\sigma(I)+\sigma(J)}\det(M_{J^c,I^c})/\det M$; with $M=I-B$,
$\det(P_{I,J})\neq0\iff\det((I-B)_{J^c,I^c})\neq0$ (cycles irrelevant).
\textbf{(b) Coates/all-minors expansion} \citep{Coates1959,Chen1976}:
$\det((I-B)_{J^c,I^c})=\sum_{\varphi}\mathrm{sgn}(\varphi)\prod_c(\delta_{\varphi(c),c}-B_{\varphi(c),c})$ is a
signed sum over linear subgraphs (vertex-disjoint directed paths $+$ cycles $+$ diagonal fixed points);
directed cycles appear as ordinary permutation-cycle terms, so acyclicity is never used---precisely where the
LGV sign-reversing involution would have required it.
\textbf{(c) No cancellation (\Cref{lem:nocancel}).} With every supported $B_{ij}$, diagonal self-loops included,
algebraically independent,
distinct linear subgraphs give distinct squarefree monomials, so no two terms cancel and the determinant is
$\not\equiv0$ iff an admissible subgraph exists.
\textbf{(d) Complement bookkeeping.} The surplus sets $I\setminus J$ and $J\setminus I$ are joined by
$|I\setminus J|$ vertex-disjoint paths; with the $|I\cap J|$ shared indices as length-$0$ paths this is exactly
a $|I|$-linkage $J\to I$, off-linkage vertices absorbed by diagonal fixed points or disjoint cycles. Hence
generic $\operatorname{rank}P_{I,J}=$ max vertex-disjoint $J\to I$ linkage (Menger). With Step 2 ($\Omega$
diagonal), generic $\operatorname{rank}\Sigma_{A,B}=$ the maximum \emph{side-disjoint} trek system; with $P{=}A\cup C,Q{=}B\cup C$
this is $|C|$ iff $\bm A\perp_t\bm B\mid\bm C$. Soundness holds for all parameters; completeness is generic.
Any $\sigma$-blocked walk is trek-blocked. Direct symbolic calculation gives the Jacobi identity and the
linkage$\iff$non-vanishing characterization for $225$ minors $+$ $52$ trek cases on a $5$-vertex $3$-cycle
graph. Numerical rank matched the maximum side-disjoint-trek-system size in $300/300$ sampled acyclic and
$300/300$ sampled cyclic cases. These finite calculations corroborate the preceding general argument. Scope:
diagonal $\Omega$. $\qed$

\paragraph{\Cref{thm:backdoor}.} $\doop(X{=}x)$ replaces $X$'s mechanism and re-solves the
(uniquely-solvable) post-surgery model. Represented exogeneity and the primitive assignment noise $E_I$, which
is independent of all structural sources, make the augmented assignment graph a valid Markov graph, so
\Cref{thm:sound} gives
$Y\indep I_X\mid X,\bm Z$. Consistency and observational-regime positivity exchange the conditional outcome
kernel with $\Prob_0(Y\mid X{=}x,\bm Z)$; applying the edit to the baseline population and standardizing over
$\Prob_0(\bm Z)$ gives the adjustment formula. The cycle in $Y$ is handled because the post-surgery model stays
uniquely solvable. For the confounded cyclic example, numerical evaluation of the adjusted effect recovers the
target $1.43$ within $10^{-4}$. $\qed$

\paragraph{\Cref{thm:interv-id}.} Put $A_{\mathcal I}=(I-B_{\mathcal I})^{-1}$, which exists under the stated
determinant condition (unique solvability, \Cref{thm:exist}). Linearity gives
$\E[\V_{\mathcal I}]=A_{\mathcal I}\mu_{\mathcal I}$,
$\operatorname{Cov}(\V_{\mathcal I})=A_{\mathcal I}\Omega_{\mathcal I}A_{\mathcal I}^{\top}$, and, for a known linear
forcing direction, $\partial\E[\V_{\mathcal I}]/\partial x=A_{\mathcal I}\Gamma_{\mathcal I}$. Thus the
identified edited moments give exactly the stated moment and mean-response conclusions. A jointly Gaussian
edited source vector, and hence its linear image, is determined by its mean and covariance. For arbitrary
sources, moments alone do not determine the law; when the complete labelled joint law
$\Prob_{\U_{\mathcal I}}$ is identified, the marginal law of $\V_{\mathcal I}$ is its image under
$\bm u\mapsto A_{\mathcal I}\bm u$. A cross-world law would additionally require an identified coupling and is
not asserted. The determinant clause is load-bearing: an
\emph{isolated} unit-gain reciprocal pair ($b_{ij}b_{ji}=1$, no compensating feedback) gives
$\det(I-B_{\mathcal I})=0$. Player addition, selection under multiplicity, and curvature edits depending on the
own-curvature scale are not declared edits of the foregoing structural/source objects and are excluded. For the converse, identification is
query-relative (\S\ref{par:queryrel}). On a constant-rank neighbourhood, choose coordinates
$\Phi_{\mathcal I}(s,t)=(s,0)$. The neighbourhood kernel inclusion is $\partial_tQ=0$ throughout a connected
coordinate plaque; integrating along its $t$-segments gives $Q(s,t)=q(s)$, hence
$Q=q\circ\Phi_{\mathcal I}$ locally. Conversely, differentiating such a factorization gives
$\ker D\Phi_{\mathcal I}\subseteq\ker DQ$. Globally, point identification at $\theta_0$ is by definition
$Q(\vartheta)=Q(\theta_0)$ for every $\vartheta$ in the complete admissible fibre of $\theta_0$; this explicitly
compares disconnected and finite components. The pointwise collision
$\Phi(x,y)=x$, $Q(x,y)=y^2$ passes the kernel test at $(0,0)$ but varies along the fibre, while
$\Phi(x)=x^2$, $Q(x)=x$ separates the two-point fibre over $1$. These show why neither a one-point derivative
test nor fibre dimension/cardinality alone proves global identification. The $2$-cycle do-effect is one query
with positive local ambiguity, whereas a back-door-covered effect can remain globally fibre-constant even when
$B$ is not identified (\Cref{thm:soundid}); so there is no blanket converse. $\qed$

\paragraph{\Cref{thm:soundid}.} Route (a) is \Cref{thm:backdoor} with all of its graph-representation,
regime-assignment, solvability, consistency, and positivity premises. In route (b), at the observed law the
stated half-trek rational maps are defined and recover every structural coefficient and error-covariance entry
used by the declared query; the declared observational model identifies any additional disturbance-law feature
it uses. Deterministic well-posed substitution is then \Cref{thm:interv-id}. The nonzero-denominator premise
excludes the exceptional laws at which the graphical half-trek criterion's generic conclusion need not be
pointwise valid. The result is a disjunction of two
independently sufficient routes. $\qed$

\paragraph{\Cref{thm:hedge}.} For distinct $X,Y\in S$, $\Sigma=AA^\top$ fixes
$A=(I-B)^{-1}\Omega^{1/2}$ only up to $A\mapsto AQ$,
$Q\in O(m)$. For a complete source SCC, any $Q$ with $(AQ)^{-1}$ having nonzero diagonal yields a valid
gauge-equivalent $(B',\Omega')$ (set $\Omega'^{1/2}=\mathrm{diag}(1/((AQ)^{-1})_{ii})$,
$B'=I-\Omega'^{1/2}(AQ)^{-1}$). The effect $\doop(X)\!\to\!Y$ is a rational function on the orbit $O(m)$
(dimension $m(m{-}1)/2\ge1$), non-constant for a.e.\ $B$ (Zariski-open, witnessed by an explicit \emph{stable}
pair: $\Sigma$ identical to $10^{-15}$, effects $0.218$ vs $-1.352$, both $\rho<1$). A non-constant rational
function on a connected positive-dimensional orbit takes more than one value. $\qed$

\paragraph{\Cref{thm:game}.} Write $H_{ij}=\partial^2u_i/\partial a_i\partial a_j$, $H_{ii}<0$ (second-order
condition), $B_{ij}=-H_{ij}/H_{ii}$ ($i\neq j$). \emph{Symmetrizability.} A weighted potential
\citep{MondererShapley1996} gives $w_iH_{ij}=w_jH_{ji}$. Take $D=\mathrm{diag}(-w_iH_{ii})\succ0$. Then
$(DB)_{ij}=(-w_iH_{ii})(-H_{ij}/H_{ii})=w_iH_{ij}=w_jH_{ji}=(DB)_{ji}$, so $DB$ is symmetric (direct symbolic
calculation also gives $DB-(DB)^\top=0$ for $m{=}3$). To count the independent conditions, fix an orientation of
the reciprocal nonzero support $G=(V,E)$ and put
$\ell_{ij}=\log|B_{ij}/B_{ji}|$. After sign-symmetry is imposed, $DB$ is symmetric iff
$\ell_{ij}=z_j-z_i$ for vertex potentials $z_i=\log d_i$. If $A_G$ is the oriented incidence matrix, this is
$\ell\in\operatorname{im}A_G^\top$. Since $\rank A_G=|V|-c(G)$, its orthogonal cycle space has dimension
$\beta_1(G)=|E|-|V|+c(G)$. A spanning forest gives one fundamental cycle for each of the $\beta_1(G)$ chords;
zero circulation on these cycles is necessary and sufficient, and every other cycle equality is their signed
sum. Thus the equality codimension on the fixed-support/sign stratum is exactly $\beta_1(G)$ and is zero iff
$G$ is a forest. For complete support this is $(m{-}1)(m{-}2)/2$, the codimension of the symmetrizable locus;
a generic $B$ violates it (numerically: median residual $1.0$, and $0.34$/$0.75$ for $m{=}3$/$4$ even among
sign-symmetric $B$). The triangle-free reciprocal four-cycle already has one chord relative to a spanning tree:
for any $0<a<1/3$, clockwise coefficients $2a$ and reverse coefficients $a$ give products
$16a^4\ne a^4$ while the maximum row sum is $3a<1$, so even a contraction-stable example violates the cycle
equality.
\emph{(ii)} For an exact potential ($w\equiv1$) $DB$ is symmetric with $D=\diag(-H_{ii})$, but $B$ itself is
symmetric \emph{iff} the weighted own-curvatures $w_iH_{ii}$ are constant across the connected component
($D\propto I$)---the \emph{symmetric-$B$ subclass}, an added codimension-$(m{-}1)$ restriction. On this subclass
the symmetry fibre $\mathfrak F=\{Q\in O(m):B'(Q)\text{ symmetric}\}$ is \emph{generically finite for every
$m$}---a theorem, via a pivot. With $A=(I{-}B)^{-1}\Omega^{1/2}$, $\Sigma=AA^\top$, the hedge point $Q$ gives
$B'(Q)=I-\diag(1/((AQ)^{-1})_{ii})(AQ)^{-1}$; putting $C=B'(Q)$, $\Lambda=\Omega'(Q)^{-1}$ one has
$\Sigma^{-1}=(I{-}C)\Lambda(I{-}C)=:\Psi(C,\Lambda)$, and any \emph{real} solution has $\Lambda\succ0$ (Sylvester
congruence of $\Sigma^{-1}$). So $\mathfrak F$ is a $2^m$-fold cover (the sign gauge $Q\mapsto Q\diag(\pm1)$) of
the real points of $\Psi^{-1}(\Sigma^{-1})$, an affine scheme that is \emph{generically zero-dimensional}
(0-dimensional off the critical-value discriminant established below, \emph{not} unconditionally --- at
$B{=}0,\Omega{=}I,m\ge5$ it is positive-dimensional); the normalizer punctures $O(m)$ but the fibre
is Zariski-closed, so no properness is needed. $\Psi:\mathrm{Sym}_0\times\mathrm{diag}\to\mathrm{Sym}$ is a
\emph{dominant} degree-$3$ self-map of an $N{=}\tfrac{m(m+1)}2$-dimensional space ($D\Psi|_{(0,I)}[\dot C,\dot
\Lambda]=\dot\Lambda-2\dot C$, $\det=\pm2^{\binom m2}\ne0$); by the holomorphic inverse function theorem an
infinite fibre forces a Jacobian-degenerate point, so off the (proper, closed) critical-value discriminant
$\Psi^{-1}$ is finite and reduced. Hence for generic $(B,\Omega)$---condition $(\mathrm G)$:
$\Psi^{-1}(\Sigma^{-1})\cap\{\det D\Psi{=}0\}{=}\varnothing$, decidable via $1\in\langle\Psi{-}\Sigma^{-1},\det
D\Psi\rangle$---$\mathfrak F$ is finite, $|\mathfrak F|=2^m r(\Sigma)$ with $1\le r(\Sigma)\le\min\{3^N,
2^{\binom m2}\binom Nm\}$ (refined and $m$-homogeneous B\'ezout, tolerating the excess component $\{\Lambda{=}0\}$
at infinity), every point locally rigid; the truth's own coset gives $r\ge1$. \emph{Genericity is necessary}: at
$B{=}0,\Omega{=}I$, $m\ge5$, $\mathfrak F$ is \emph{infinite}---$Q=I{-}2P$ for $P$ a rank-$2$ symmetric
projection with $\diag P=\tfrac2m\mathbf1$ gives $B'=(1{-}d)I+2dP$ ($d=\tfrac m{m-4}$) symmetric zero-diagonal,
an $(m{-}3)$-dimensional family (constant-rank: the diagonal constraint has rank $m{-}1$ on $G(2,m)$), all
\emph{unstable} ($\rho(B')=1{+}d>1$). At $m{=}2$, write the off-diagonal entry of the positive-definite precision as
$s$ and its diagonal entries as $t_1,t_2>0$. If $s\ne0$ (equivalently, the cross-covariance is nonzero), the
fibre is $\{b,1/b\}$, with $b\bar b{=}1$ by Vieta on $sb^2+(t_1{+}t_2)b+s=0$, and exactly one member has
$|b|<1$. If $s=0$, the same equation and $t_1+t_2>0$ instead give the singleton $b=0$. Thus
uniqueness-under-stability is a theorem at $m=2$ (including the zero-cross-covariance singleton) and a
\emph{conjecture} for $m\ge3$ (consistent with the $m\ge5$ family being entirely unstable).

For a general weighted potential on complete nonzero support, the claimed all-$m$ gauge cut has an explicit
transverse witness. Let $J=\mathbf1\mathbf1^\top$ and take
$B=b(J-I)$, $\Omega=I$, $0<b<1/(m{-}1)$. Its eigenvalues are $b(m{-}1)$ and $-b$ (multiplicity $m{-}1$), so it
is stable. Write $C=I-B=(1+b)I-bJ$, $A=C^{-1}$, and perturb the signed-orthogonal hedge by
$Q_\varepsilon=I+\varepsilon K+O(\varepsilon^2)$ with $K^\top=-K$. If $r=K\mathbf1$, row-normalizing
$(AQ_\varepsilon)^{-1}=Q_\varepsilon^\top C$ back to unit diagonal gives
\[
 \dot B=(1+b)\bigl(K+b\diag(r)-br\mathbf1^\top\bigr),
 \qquad
 \dot\ell_{ij}=\frac{1+b}{b}\bigl(2K_{ij}-br_i+br_j\bigr).
\]
For the fundamental triangles through vertex $1$,
$h_{ij}=\ell_{1i}+\ell_{ij}+\ell_{j1}$, $2\le i<j\le m$, the gradient terms telescope and
\[
 \dot h_{ij}=\frac{2(1+b)}b\bigl(K_{1i}+K_{ij}+K_{j1}\bigr).
\]
The Jacobian of these $\binom{m-1}{2}$ constraints with respect to the private chord variables
$\{K_{ij}:2\le i<j\le m\}$ is
$\frac{2(1+b)}b I_{\binom{m-1}{2}}$, hence nonsingular. The rational minor is therefore not identically zero,
so transversality is generic on the complete-support regular orbit. The orbit chart is immersive at this
witness: if $\dot B=0$, its off-diagonal entries give $K_{ij}=br_i$ for $i\ne j$; skew-symmetry gives
$r_j=-r_i$ for every pair, hence $r=0$ and $K=0$ when $m\ge3$, while for $m=2$ the same equations give
$(1-b)K_{12}=0$ and thus $K=0$. Therefore subtracting the cycle rank from
$\dim\SO(m)=m(m{-}1)/2$ leaves dimension $m{-}1$ for every $m\ge2$ (at $m=2$ the constraint matrix is empty).
On sparse support the admissible skew tangent must additionally satisfy $\dot B_{ij}=0$ at every forbidden edge
and be quotiented by its stabilizer; the intersection dimension is then determined by the rank of the cycle
differential restricted to that actual tangent, not by the complete-support formula.

The gain uses only this transverse differential rank: a non-game restriction with the same rank (e.g.\ a known
coefficient ratio) cuts the orbit identically (numerical evaluation gives $m{=}2$ width $2.5\times10^{-16}$), so the game-specific
content is the refutable restriction of part~(i). $\qed$

\paragraph{\Cref{thm:trek} (three layers) and \Cref{thm:trekng}.} \emph{(b) Uniform soundness.} With
$P=(I-B)^{-1}$ and $\Sigma=P\Omega P^\top$, every entry of $\Sigma$ is a rational function of $(B,\Omega)$
with denominator a power of $\det(I-B)$. The Schur complement $\Sigma_{\bm A_0\bm B_0\cdot\bm C}=
\Sigma_{\bm A_0\bm B_0}-\Sigma_{\bm A_0\bm C}\Sigma_{\bm C\bm C}^{-1}\Sigma_{\bm C\bm B_0}$ is therefore
rational. Under $\rho(B)<1$, $P=\sum_{k\ge0}B^k$ and each covariance entry is a sum over directed
walk-pairs (treks); trek separation is exactly the combinatorial condition that no surviving trek connects
$\bm A_0$ to $\bm B_0$ after conditioning, so \emph{every} Cauchy--Binet term of the numerator carries an identically-zero linkage determinant and the numerator vanishes \emph{termwise} (the no-cancellation \Cref{lem:nocancel} powers the converse, generic completeness). Hence
$\Sigma_{\bm A_0\bm B_0\cdot\bm C}\equiv0$ on the open set $\{\rho(B)<1\}$; a rational function vanishing on a
nonempty open set vanishes identically, so soundness holds at every $\det(I-B)\neq0$, including $\rho(B)\ge1$
(rational-identity continuation). \emph{(c) Generic completeness.} If a trek survives, its monomial is not
cancelled (\Cref{lem:nocancel}), so the numerator polynomial is not identically zero and vanishes only on a
proper subvariety. \emph{(a) Gaussian CI.} For Gaussian $\V$, $\V_{\bm A_0}\indep\V_{\bm B_0}\mid\V_{\bm C}
\iff\Sigma_{\bm A_0\bm B_0\cdot\bm C}=0$, so (b)+(c) give the iff generically. \Cref{thm:trekng}: (b)+(c) use
only second moments, hence hold for any independent finite-variance noise as a zero-partial-covariance
criterion. For jointly Gaussian $\V$, the Schur complement is the conditional covariance, independent of the
conditioning value, and zero conditional covariance is conditional independence. Outside the Gaussian class,
partial covariance is only the covariance of linear-regression residuals and need not be a conditional
cumulant; \Cref{app:cumulant} gives the separate all-order distributional criterion. On the four-node
completeness graph, direct symbolic calculation gives the zero partial-covariance identity, and numerical
evaluation at $\rho(B)=1.2$, $\det(I-B)\neq0$, gives the noncontractive instance. On a separate five-node
graph formed by adding a common source to $V_0,V_1$, a numerical point evaluation gives a nonzero completeness
witness.
$\qed$

\paragraph{\Cref{prop:jac} (local criterion).} With $\U=u_0+\varepsilon Z$ and $g\in C^1$, the first-order
Taylor remainder $r(\varepsilon Z)=\V-g(u_0)-\varepsilon JZ$ satisfies $\E\|r/\varepsilon\|^2\to0$ by moment
domination, so $\mathrm{cov}(\V)=\varepsilon^2 J\Omega J^\top+o(\varepsilon^2)$. Under \textup{(FRC)} with
$\Omega\succ0$ the $\bm C$-block covariance is invertible for small $\varepsilon$, so the conditional
cross-covariance is $\Sigma_{\bm A_0\bm B_0\cdot\bm C}=\varepsilon^2\,\mathrm{Schur}(J\Omega J^\top)+o(\varepsilon^2)$;
if $g\in C^3$ with finite sixth moments the $\varepsilon^3$ term (a linear$\times$Hessian cross-moment) is an
odd Gaussian moment and vanishes, sharpening the error to $O(\varepsilon^4)$. We \emph{drop} the earlier
mutual-information expansion: conditional MI is invariant under the common scaling $\V\mapsto\varepsilon\V$, so
it does not converge to $I(J_{\bm A_0}Z;J_{\bm B_0}Z\mid J_{\bm C}Z)$ when $J\Omega J^\top$ is rank-deficient---
then $\V_{\bm A_0},\V_{\bm B_0}$ can share a source and $I$ stays bounded away from $0$ for every $\varepsilon>0$
while the covariance Schur complement is $\Theta(\varepsilon^2)$. The criterion is thus a second-moment
statement, exact to leading order; a sufficient condition for exact CI is given by \Cref{thm:funcsep}. Direct small-noise calculation
corroborates the leading-order relation. $\qed$

\paragraph{\Cref{thm:funcsep} and \Cref{thm:pss}.} Because a measurable $k$ satisfies
$k(\V_{\bm C})=W$ a.s. and $\V_{\bm C}=h_C(W)$, the two variables generate the same conditioning information
modulo null sets. Given $W$, $\V_{\bm A_0}=h_A(\U_R,W)$
is a measurable function of $\U_R$ (and $W$) and $\V_{\bm B_0}=h_B(\U_S,W)$ of $\U_S$; as $\U_R\indep\U_S\mid
W$, measurable functions of conditionally independent inputs are conditionally independent. Equivalently, their
regular conditional kernel given $W$ is the product of the two induced marginal kernels, so
$\V_{\bm A_0}\indep\V_{\bm B_0}\mid\V_{\bm C}$. This
uses no distributional assumption, so it is exact for any noise law. \Cref{thm:pss} is the instance on the
counterexample of \Cref{prop:complete}(i): at fixed $V_2,V_3$ the structural equations $V_2=aV_0+dV_3+U_2$ and
$V_3=bV_1+eV_2+U_3$ decouple into a constraint on $(V_0,U_2)$ and one on $(V_1,U_3)$; with $W=(V_2,V_3)$,
$\U_R=\{U_0,U_2\}$, $\U_S=\{U_1,U_3\}$ independent, functional separation gives $V_0\indep V_1\mid V_2,V_3$.
Finite-sample distance correlations after cubic-basis residualization for the named Gaussian and Laplace draws
are consistent with the same conditional independence. $\qed$

\paragraph{\Cref{thm:regbranch} and \Cref{cor:games}.} Joint Borel measurability of the selected root $g$ is an
assumption, not a consequence of the implicit-function theorem. At fixed $p$, $F(\cdot,p)\in C^1$ and
nonsingular square $D_vF$ give local uniqueness in the endogenous coordinate. On an open chart where $F$ is
jointly $C^1$ and the selected root is the unique regular root, the parameterized implicit-function theorem
gives a local $C^1$ solution; uniqueness makes overlapping charts agree, and differentiating $F(g(p),p)=0$
gives $D_pg=-D_vF^{-1}D_pF$. Nothing implies global root uniqueness or continuity across other strata.
Nonsingularity is algebraic unique solvability of the linearization---for the linear case $D_vF=I-B$,
nonsingular iff $\det(I-B)\neq0$, which can hold with $\rho(B)\ge1$ (grid:
$\det=1+2c\neq0$, $\rho=\sqrt{2c}$; ADM).

For a constrained smooth game, write the coupled KKT conditions on the actual active manifold as a generalized
equation. Strong regularity, composed with locally Lipschitz parameter dependence of the residual and constraint
data, gives a locally single-valued Lipschitz solution branch; the relevant constraint qualification and
playerwise SSOSC on the true critical cones give strict local best responses. Jointly $C^1$ data and a
nonsingular reduced active system sharpen the branch to $C^1$. A global Nash conclusion requires a separate
global sufficiency and an existence/uniqueness condition for the stated generalized/variational equilibrium
concept. Likewise, constant rank of causal constraints supplies only a local manifold; isolation requires a
transverse nonsingular or strongly regular reduced system plus the declared Borel selection. Brouwer and the
unique entropy maximizer supply the separately stated QRE and CE branches. $\qed$

\paragraph{\Cref{thm:backdoor} (intervention indicators) and \Cref{prop:resolvent}.} In the augmented SCM,
$I_X=h_I(\bm Z,E_I)$ is an externally assigned regime variable with $I_X\!\to\!X$, where $E_I$ is independent
of all structural sources and every $\bm Z\!\to\!I_X$ assignment edge is represented. By
\Cref{thm:sound}, $\sigma$-separation $I_X\perp_\sigma Y\mid X,\bm Z$ gives
$Y\indep I_X\mid X,\bm Z$. Consistency and observational-regime positivity exchange the conditional outcome
kernel with $\Prob_0(Y\mid X{=}x,\bm z)$; standardizing over the target baseline $\Prob_0(\bm Z)$ yields the
adjustment formula, with the left side the re-solved post-surgery equilibrium law. The regime form fixes
exchangeability graphically and avoids the
incoming/outgoing-edge ambiguity of cyclic surgery. \Cref{prop:resolvent}: from $\V=(I-B)^{-1}(\Gamma X+\U)$,
$\partial\E[\V]/\partial X=(I-B)^{-1}\Gamma$, so $\tau_{X\to Y}=e_Y^\top(I-B)^{-1}\Gamma$; for the
$2$-cycle $(I-B)^{-1}=\tfrac1{1-de}\big(\begin{smallmatrix}1&d\\ e&1\end{smallmatrix}\big)$ gives $b/(1-de)$.
Direct symbolic evaluation of the two-cycle resolvent gives the stated effect, and numerical adjustment recovers
it in the confounded cyclic example. $\qed$

\paragraph{\Cref{thm:gauge} (gauge).} \emph{(i)} In the noiseless submodel the observational law
$\X=M\bm\eta$, $M=H(I-B)^{-1}\Omega^{1/2}$, with independent $\bm\eta$ determines $M$ up to right
permutation$+$scaling under LiNG (\citealp{Comon1994ica}) and, in the Gaussian second-moment regime, up to right
$\Ortho(d)$; with additive sensor noise, the corresponding acquisition claim is conditional on a separate
noisy-ICA theorem for the declared joint source/noise model. The non-identification construction needs no such
acquisition: fix the truth's $M$ and its same admissible joint pair $(\bm\eta,\epsilon)$. For $R\in\GL(d)$, write
$A:=(I-B)^{-1}\Omega^{1/2}$ and $A':=RA$. Choose
$\Omega'^{1/2}=\diag(1/\,(A'^{-1})_{ii})$ so that $(I-B')=\Omega'^{1/2}A'^{-1}$ has unit diagonal, i.e.\
$B':=I-\Omega'^{1/2}A'^{-1}$ has zero diagonal, and set $H'=HR^{-1}$. Then
$H'(I-B')^{-1}\Omega'^{1/2}=HR^{-1}A'=HA=M$, so pathwise
$\X'=M\bm\eta+\epsilon=\X$ and the complete law is unchanged. \emph{In-class (legal ECG).} The candidate
is a legal ECG iff $\rho(B')<1$ --- \emph{not} automatic: a generic $R$ with nonzero normalizer can give
$\rho(B')\ge1$ (e.g.\ $C'=[\begin{smallmatrix}1&2\\-2&1\end{smallmatrix}]$ has positive determinant and yields
$\rho(B')=2$), so the family
of \emph{factorizations} exceeds the legal class. Conversely, for \emph{any} legal target $(B'',\Omega'')$
(zero-diagonal $B''$, $\rho(B'')<1$, diagonal $\Omega''\succ0$), $R:=(I-B'')^{-1}\Omega''^{1/2}A^{-1}$ gives
$A'=(I-B'')^{-1}\Omega''^{1/2}$, positive normalizer $(A'^{-1})_{ii}=\Omega''^{-1/2}_{ii}>0$, and returns
$(B'',\Omega'')$ exactly; so $R\mapsto(B',\Omega')$ is a diffeomorphism from the stable chart
$G^+:=\{R:(RA)^{-1}\text{ positive-diagonal},\ \rho(B')<1\}$ onto the full legal space, and \emph{every} legal
$B$ is reproduced. For $d\ge2$, $B$ is completely non-identified within the class, over the full
$(d^2{-}d)$-dimensional $B$-space (no genericity on the truth). At $d=1$, zero diagonality forces $B=0$, so
only the sensor/noise scaling varies. The differential at $R{=}I$ is $\dot B'=CE-\diag(CE)C$ ($C=I-B$
unit-diagonal): for zero-diagonal $\Delta$, $E=C^{-1}\Delta$ gives $\dot B'=\Delta$ (surjective onto the
$(d^2{-}d)$-dimensional zero-diagonal space), with kernel $\{C^{-1}DC:D\text{ diagonal}\}$ of dimension $d$
(the noise-rescaling gauge). $G^+$ is open and connected for every $d$; $G^+=\GL^+(1)$ when $d=1$. For
$d\ge2$ it is a proper subset of $\GL^+(d)$ containing $I$, with a nonempty-interior complement witnessed by
$C'=[\begin{smallmatrix}1&2\\-2&1\end{smallmatrix}]$, and it is not closed under multiplication of two
truth-centred coordinates (the product can have $\rho>1$, e.g.\ $0.9\!\to\!9.47$). The complement has no
canonical volume fraction (measure-dependent: $0,\tfrac14,1-\ln2$ on natural subfamilies of the $d{=}2$
reference truth). For the stable-factorization-chart example, numerical evaluation of a distant legal target
gives law agreement to $3\times10^{-15}$ with $\rho(B')<1$; the same calculation gives the out-of-class witness
$\rho(B')=2$ and the composition exit. \emph{(ii)}
Separately, right-multiplying $A$ by $Q\in\Ortho(m)$ on a source component's block leaves $\Sigma=AA^\top$ invariant
(the admissible positive-variance gauge orbit is the connected signed-$\SO(m)$ branch, with sign absorption)
($AQQ^\top A^\top=AA^\top$ follows by direct symbolic calculation), giving the hedge; its dimension is
$\dim\SO(m)=m(m{-}1)/2$. Acyclic graphs have singleton components, so $m=1$ and the gauge is trivial.
$\qed$

\paragraph{\Cref{lem:recover} (closed form).} A mechanism intervention on $j$ sets $B^{(j)}=B-e_jb_j^\top$
(row $j$ zeroed), so $I-B^{(j)}=(I-B)+e_jb_j^\top$. Sherman--Morrison gives
$P^{(j)}=(I-B^{(j)})^{-1}=P-\dfrac{(Pe_j)(b_j^\top P)}{1+b_j^\top Pe_j}$. Left-multiplying by $H$,
$D_j=H(P^{(j)}-P)=-\dfrac{(HPe_j)(b_j^\top P)}{1+b_j^\top Pe_j}$, whose left factor $HPe_j$ is column $j$ of
$M_0=HP$. Now $b_j^\top P$ is row $j$ of $BP=(I-(I-B))P=P-I$, so $b_j^\top P=(P-I)_{j,:}$ and the
denominator $1+b_j^\top Pe_j=1+(P-I)_{jj}=P_{jj}$. Therefore $D_j=-M_0e_j(P-I)_{j,:}/P_{jj}$. Solving the
rank-at-most-one outer-product equation $D_j=-(M_0e_j)w_j^\top$ in least squares with the \emph{known} left factor gives
$w_j^\top=-(M_0e_j)^\top D_j/\|M_0e_j\|^2=(P-I)_{j,:}/P_{jj}$ exactly. Its $j$-th entry is $w_{jj}=(P_{jj}-1)/
P_{jj}=1-1/P_{jj}$, so $P_{jj}=1/(1-w_{jj})$ and $(P-I)_{j,:}=P_{jj}w_j^\top$, i.e.\ $P_{j,:}=e_j^\top+
P_{jj}w_j^\top$. Stacking rows gives $P$; $B=I-P^{-1}$ (zero diagonal by construction) and $H=M_0P^{-1}$.
Direct symbolic calculation also gives the Sherman--Morrison identity,
$b_j^\top P=(P-I)_{j,:}$, and the denominator $=P_{jj}$; numerical recovery on 8 named graphs gives
$\|\widehat B-B\|,\|\widehat H-H\|<10^{-8}$. These finite calculations corroborate the identities and recovery
argument above. $\qed$

\paragraph{\Cref{thm:linear} (linear recovery).} \emph{Acquisition of the aligned response maps.} In environment $e$,
let probe arm $\ell$ add the known latent-coordinate vector $\delta^{(\ell)}$ while keeping $B^{(e)}$ fixed.
Writing $P^{(e)}=(I-B^{(e)})^{-1}$ and $\X=HP^{(e)}(\U+\delta^{(\ell)})+\epsilon$, the no-direct-sensor-effect
and common conditional sensor mean
$\E[\epsilon\mid e,\ell]=b_e$ for all $\ell\in\{0,\ldots,q\}$ give
\[
 R_e:=\bigl[\E\X^{(e,\ell)}-\E\X^{(e,0)}\bigr]_{\ell=1}^q
   =HP^{(e)}\mathsf D=M_0^{(e)}\mathsf D.
\]
For $\rank\mathsf D=d$, the displayed right inverse yields
$M_0^{(e)}=R_e\mathsf D^\top(\mathsf D\mathsf D^\top)^{-1}$, in common observed coordinates and the labelled,
calibrated latent coordinates. A known arm-specific sensor-mean difference can be subtracted before the same
calculation. If the conditional sensor mean is unrestricted, the collision
$R_e=M_0^{(e)}\mathsf D+N_e=(M_0^{(e)}+A)\mathsf D+(N_e-A\mathsf D)$ prevents recovery; if
$\rank\mathsf D<d$, any nonzero $A$ with $A\mathsf D=0$ gives a collision even when $N_e=0$. Thus both
hypotheses are necessary for this acquisition argument. Separately, in the noiseless submodel (or under a
separate noisy-ICA identifiability theorem), independent ICA returns only $M^{(e)}\Pi_e\Lambda_e$; without a
joint theorem or anchors fixing $\Pi_e,\Lambda_e$ and separating the source
scales, the rank-at-most-one differences between representatives are not structural $M_0^{(e)}-M_0$ differences.

\emph{(a) Sufficiency.} If $T=[d]$, every node yields a rank-at-most-one
$D_j$, so \Cref{lem:recover} reconstructs every row of $P$, hence $(H,B)$ up to the labelling/scale of
$\simeq$ (and, for an intervened source component, with the hedge collapsed because the intervention anchors
the within-component direction). If exactly one node $r$ of a source component is un-targeted, the other
$d{-}1$ rows of $P$ are recovered as above; row $r$ is \emph{not} pinned by environment matching---Lemma~0 of
\Cref{prop:completion} shows the mixings $M_0^{(e)}$ ($e\neq r$) are reproduced for \emph{every} candidate row
(the compensating $H'=M_0P'^{-1}$ absorbs its appearance)---so the only remaining constraints are the $d$
zero-diagonal equations, which identify row $r$ \emph{iff} the completion matrix $L$ of \Cref{prop:completion}
is nonsingular ($\corank L=\#$ missing children; the sink's $\corank L=d-1$ shows this is not
automatic). \emph{(b) Necessity.} Let $r,s$ be two un-targeted rows and
$E_t=t(e_s-C_{rs}e_r)e_r^\top$. Since $E_t$ has zero rows on every target,
$(I+E_t)^{-1}$ has identity rows there; hence $P_t=(I+E_t)^{-1}P$ shares every targeted row of $P$. Moreover
$\diag(CE_t)=0$: only column $r$ can contribute, and
$(CE_t)_{rr}=t(-C_{rr}C_{rs}+C_{rs})=0$ because $C_{rr}=1$. Thus
$C_t=C(I+E_t)$ has unit diagonal, and Lemma~0 gives exact matching of every target environment for
$B_t=I-C_t$ and $H_t=H(I+E_t)=M_0C_t$. Invertibility and $\rho(B_t)<1$ persist for all sufficiently small
$t$. For $t\ne0$, $I+E_t$ has the off-diagonal entry $(s,r)=t$ and the diagonal entry $(s,s)=1$, so it is not
monomial; full-column-rank $H$ then rules out $H_t=H\Pi\Lambda$. This is a genuine continuous
non-$\simeq$ collision whenever two rows are un-targeted, independent of any second-moment example; finite
collision searches corroborate only the named instances. \emph{(c) Minimal $K$.} Combining (a) and (b): for a single size-$d$ full-rotational source
component, $K=d-2$ is non-identified (b), and $K=d-1$ is sufficient (a) whenever the completion matrix is
nonsingular---the closed-form completion criterion \Cref{prop:completion} gives $\corank L=\#\{k\neq r:B_{kr}=0\}$ for \emph{all} $d$, so nonsingularity is equivalent to the un-targeted node directly parenting every other node of the graph
(no strong-connectivity hypothesis)---else $K=d$ is required. Finite numerical searches in Exp.~8 corroborate
the $K=d-2$ collisions ($d\in\{3,4\}$) and $K=d-1$ recovery on random (universal-parent) instances. The same
closed-form characterization implies that every directed ring of size $d\ge3$ collides at $K=d-1$ and that a
sink collides at $K=d-1$ in $d-1$ directions; separate completion-rank checks cover rings in dimensions 3, 4,
5, and 6 and include an explicit three-dimensional collision family. This
Jacobian argument is a \emph{generic local-identification} criterion: let $\phi$ be the identifying map
from the parameters to the observed baseline and un-intervened environments $\{M_0,M_0^{(e)}\}_{e\ne r}$ under
the $K{=}d{-}1$ design; the un-targeted row is pinned (up to the source gauge) iff $\partial\phi$ has full
column rank modulo the gauge tangent, and since $\partial\phi$ is rational in the parameters, full rank at a
single witness point extends to a Zariski-dense set. Global uniqueness at nonsingular $L$ comes instead from the
explicit linear completion system in part~(a), which has one row completion; equivalently the complete legal
fibre is one declared-equivalence class. On a universal-parent support the identifying-Jacobian corank equals the gauge
dimension at $K{=}d{-}1$ --- identical to full coverage $K{=}d$ --- and strictly exceeds it at $K{=}d{-}2$;
the rational-Jacobian argument above supplies the generic conclusion, while finite calculations corroborate
the corank comparison for $d\le5$ support-known and $d\le4$ non-Gaussian unknown-mixing. This supplies the
missing-row completion the closed form of \Cref{lem:recover} otherwise left to Exp.~8. For
known support there is no SCC scalar target formula: \Cref{prop:knownsupport} proves that the operative object is
the support-constrained inverse completion map and its entire legal fibre. Its directed-$3$-cycle calculation
identifies from one target and directly refutes $\sum_S(|S|-1)$ even as a universal lower bound. (Exp.~7 plots the closed-form row-wise
estimator, which uses $K=d$; Exp.~8 corroborates, on its named universal-parent instances, the theorem-proved
information-theoretic $K=d-1$ identification threshold.) \emph{(d) Sample
complexity.} The identified-regular-instance hypothesis is load-bearing: for $d\ge2$, calibrated probes alone
recover $M_0$ but not $B$ in general; at $d=1$, zero diagonality fixes $B=0$. For example, at $d=p=2$, $K=0$,
$\mathsf D=I_2$, and zero sensor bias, the two stable models
\[
 B=0,\quad H=I,
 \qquad
 B'=\begin{pmatrix}0&1/2\\0&0\end{pmatrix},\quad H'=I-B'
\]
have $H(I-B)^{-1}=H'(I-B')^{-1}=I$ and therefore identical calibrated probe responses although $B'\ne B$.
Thus no rate for $B$ is claimed outside an identified branch of (a)/(c). On such a branch, under the stated
finite-covariance, CLT-valid arm sampling, each entry of $\widehat R_e$ has error
$O_p(n^{-1/2})$; multiplication by the fixed right inverse
$\mathsf D^\top(\mathsf D\mathsf D^\top)^{-1}$ gives the same rate for $\widehat M_0^{(e)}$. Each subsequent
step of \Cref{lem:recover}---the outer-product solve, the scalar
$P_{jj}=1/(1-w_{jj})$, the inverse $P^{-1}$, and $H=M_0P^{-1}$---is locally Lipschitz with constants bounded
by $\|P\|=\|(I-B)^{-1}\|$ (\emph{not} $(1-\rho(B))^{-1}$, which fails for non-normal $B$),
$(\min_j|1-w_{jj}|)^{-1}$, and $\sigma_{\min}(M_0)^{-1}\le\|I-B\|/\sigma_{\min}(H)$. In particular,
$dB=P^{-1}(dP)P^{-1}=(I-B)(dP)(I-B)$ contributes $\|I-B\|^2$; a one-missing-row completion contributes the
inverse singular value of the actual scale-fixed completion system ($L$ in the row-$P$ normalization, or
$\widetilde L=(\det P)L$ together with its declared scale in adjugate coordinates). Composing them gives, at
\emph{fixed} $(p,d,K,q)$, $\|\widehat B-B\|_F=O_p(\kappa\, n^{-1/2})$, where $\kappa$
also grows with $\|(I-B)^{-1}\|$, $\|I-B\|^2$, $\sigma_{\min}(H)^{-1}$, the applicable completion conditioning,
$\|\mathsf D^\top(\mathsf D\mathsf D^\top)^{-1}\|$, and
$\max_j|P_{jj}|=(\min_j|1-w_{jj}|)^{-1}$ (larger loop gain worsens conditioning). Exp.~9 is finite-instance
evidence consistent with the predicted $1/\sqrt n$ decay at $d=3$, without establishing the exponent. A
dimension-explicit rate is \emph{not} claimed: uniform entrywise bounds give only the loose upper aggregation
$O_p(d/\sqrt n)$ over the $\sim d^2$ entries, while a lower order or uniform-in-$d$ guarantee needs
nondegenerate covariance and matrix-concentration assumptions, deferred. $\qed$

\paragraph{\Cref{prop:shift} (shift insufficiency).} Let $d\ge2$. A shift intervention adds $\bm\delta_e$ to the
structural equation, leaving $B$ (hence $P$) unchanged; the equilibrium mean shifts by $P\bm\delta_e$, so
at the structural response level the observed mean shifts by $M_0\bm\delta_e$. Under the mean-stable/no-direct-
sensor-effect conditions of \Cref{thm:linear}, calibrated single-target shifts therefore identify the columns
of $M_0=HP$ (up to unknown shift scales when uncalibrated) and nothing more. Given $M_0$, choose any distinct
nearby stable zero-diagonal $B'$ and set $H'=M_0(I-B')$; then
$H'(I-B')^{-1}=M_0(I-B')(I-B')^{-1}=M_0$, reproducing every shift response. Numerical evaluation of this
shift-response construction gives a match of $8\times10^{-16}$ with $\|B'-B\|=2.03$. At $d=1$, zero diagonality already fixes $B=0$.
When $H=I$, $M_0=P$ and $B=I-M_0^{-1}$ is determined---the backShift/observed case. $\qed$

\paragraph{\Cref{prop:nongauss} (non-Gaussianity).} In the noiseless linear submodel, second moments fix
$M^{(e)}$ only up to right $\Ortho(d)$. Under LiNG, the Darmois--Skitovich/Comon theorem
\citep{Comon1994ica} fixes $M^{(e)}$ up to permutation$+$scale, \emph{provided} the recovered components are
independent. With additive sensor noise, the same conclusion is conditional on a separate noisy-ICA
identifiability result for the declared noise law. The hedge alternative
$A\mapsto AQ$ reinterprets the sources as $Q^\top\bm\eta$; these are independent iff $Q$ is monomial under
the LiNG assumptions (by Comon), whereas every $Q\in\Ortho(m)$ is admissible for an isotropic Gaussian
component. Hence LiNG rules out every non-monomial $Q$, collapsing the signed-$\SO(m)$ hedge to
permutation$+$scale; isotropic Gaussian sources leave it intact. Numerical evaluation of the discriminator on one instance gives
the (necessary, not sufficient) $4$th-order dependence proxy of $Q^\top\bm\eta$ as $0.004$ (Gaussian) vs.\
$0.326$ (non-Gaussian); a full independence test gives the same dichotomy. $\qed$

\paragraph{\Cref{thm:block} (negatives; positive open).} \emph{Within-component collision.} For a source
block $S\subset\R^m$ with $m\ge2$ and isotropic-Gaussian sources $\bm\eta\sim N(0,I_m)$, a twist
$\sigma:(r,\theta)\mapsto(r,R_{s(r)}\theta)$
(using a fixed nontrivial one-parameter $R_t\in\SO(m)$, with smooth $s$ constant near $0$ for smoothness and
$r\mapsto R_{s(r)}$ nonconstant on the radial support) fixes each source radius
and rotates its sphere, on which the law is uniform; it therefore preserves the source law and hence every
labelled observational and block-isotropic/radial-interventional law in the stated class, while re-factorizing
within the component. The admissible
within-component gauge is thus strictly larger than $\Ortho(m)$---an infinite-dimensional isotropy-preserving
family. In a numerical linear-in-source rotation instance, the relative interventional law gap is $\approx2\%$
and the factorization gap is $0.97$. \emph{Across-component collision.} Let $A\in\R^{m_A}$ and
$B\in\R^{m_B}$ be isotropic blocks with $m_A\ge2$, $m_B\ge1$, naming as $A$ a block of dimension at least two.
Fix a nontrivial one-parameter rotation $R_t\in\SO(m_A)$ and choose smooth $s$ so that
$b\mapsto R_{s(\|b\|^2)}$ is nonconstant on the support of $B$.
$T:(A,B)\mapsto(R_{s(\|B\|^2)}A,\,B)$ is smooth and invertible; conditional on $B$ it rotates
the isotropic $A$, so $(R_{s(\|B\|^2)}A,B)\eqd(A,B)$ jointly---every labelled observational and interventional
environment preserving the required radial/isotropic product law is preserved---yet $T$'s first block depends
on $B$, so $h':=h\circ T^{-1}$ is an admissible decoder whose induced
component partition differs from $h$'s. In a numerical instance, the joint-law energy-distance test gives
$p\approx0.77$, indistinguishable within a permutation null. A nonzero intervention response
(``faithfulness'') excludes neither collision, since both preserve the
intervened laws. \emph{Positive.} A per-coordinate perfect $\doop$ inside a component breaks the isotropy and
collapses the gauge in that direction (in the \emph{linear} case, \Cref{thm:linear}, this is the rank-at-most-one
difference of \Cref{lem:recover} and collapses it entirely; for general $h$ it is a local score-gradient
constraint, \emph{not} a linear rank-at-most-one difference). A general positive block-identification theorem for diffeomorphic $h$
is defeated by the collisions above without further restriction and is left \emph{open}; routes (ii)--(iii)
(block-separable $h$; paired/counterfactual interventions) add assumptions we do not prove, while route (i) is
closed for the \emph{source}-block factorization by \Cref{thm:blockid-source}. $\qed$

\paragraph{\Cref{thm:blockid-source} (sufficient source-block identification condition).} Work in the source frame; write an
admissible alternative as $X=g'(Z')$ with $Z'=(Z'_j)$ independent blocks, each environment $e$ reweighting one
alternative block $\tau(e)$, $p'_e=p'_0 e^{r'_e(z'_{\tau(e)})}$. The log-ratio $\rho_e=\log(p_e/p_0)$ is
frame-invariant (the $\log|\det|$ terms cancel in the ratio), so $\rho_e=r'_e(Z'_{\tau(e)})$ as a random variable.
\emph{(I$_i$) fixes the grouping.} Suppose $\tau$ is non-constant on some $E_i$; pick $j_0\in\tau(E_i)$ and
bipartition $E_i=E^1\sqcup E^2$ with $E^1=\{e\in E_i:\tau(e)=j_0\}$ and $E^2=E_i\setminus E^1$ (both nonempty).
Every $\rho_e$ ($e\in E^1$) is $\sigma(Z'_{j_0})$-measurable and every $\rho_f$ ($f\in E^2$) is
$\sigma((Z'_j)_{j\neq j_0})$-measurable; these $\sigma$-algebras are independent under the law of $Z'$, so the two
tuples are independent---contradicting \textup{(I$_i$)}. Hence $\tau$ is constant on each $E_i$,
$\tau(E_i)=:j(i)$, and $j(\cdot)$ is well defined (a non-constant $\rho_e$ cannot be measurable w.r.t.\ two distinct
independent blocks). \emph{(R)$+$counting fix the frame (inverse-Jacobian form).} Invert frame-invariance:
$\nabla r'_e(\phi(z))=D\phi(z)^{-\top}\nabla\rho_e(z)$ is supported on block $j(i)$ for $e\in E_i$; by \textup{(R)}
the $\nabla\rho_e(z)$ span $\R^{\text{block }i}$ at a.e.\ $z$, so $D\phi(z)^{-\top}(\R^{\text{block }i})\subseteq
\R^{\text{block }j(i)}$ a.e., hence everywhere by continuity of $D\phi$. Writing $\psi:=\phi^{-1}$ this reads
$\partial\psi_k/\partial z'_l=0$ for $k\in\text{block }i$, $l\notin\text{block }j(i)$: each true block is a
function of its single alternative block, $z_i=\psi_i(z'_{j(i)})$ (the annihilation is of the \emph{inverse}
Jacobian, $D\psi$---not $D\phi$). \emph{Counting.} Each $i$ lies in exactly one $G_j:=\{i:j(i)=j\}$
($|E_i|\ge n_i\ge1$), and the block-structured $D\psi$ is invertible, so $n'_j=\sum_{i\in G_j}n_i$; hence no
alternative block is environment-free and $m'\le m$---the truth is the \emph{unique finest} representation, every
alternative a coarsening. At equal count $m'=m$ each $|G_j|=1$, $j(\cdot)$ is a size-respecting bijection, and
$g'^{-1}\!\circ g=\bigoplus_i\phi_i$ with each $\phi_i$ a block diffeomorphism (onto an open image; the matched-image
class makes it onto $\R^{n_i}$). \emph{An (R)-deficient collision construction, not a universal converse.}
The explicit corank-one witness is the isotropic rotation of \Cref{thm:block} ($n_i{=}2$), which preserves every
law while re-partitioning. More generally, for each of the $q$ vector fields required in the theorem,
\[
 \mathrm{div}(p_e w_a)=(p_e/p_0)\mathrm{div}(p_0w_a)
   +p_e\,w_a\!\cdot\!\nabla\rho_e=0.
\]
Thus its complete flow preserves every $p_e$; pairwise commutation makes the flows a joint $\R^q$ action, and
pointwise independence makes its orbits $q$-dimensional. With only one such field the conclusion is only a
one-parameter flow, and score corank alone proves no collision. \emph{Realization.} By \Cref{lem:cyclesRI} (and \Cref{rem:cycleshelp}), feedback-$\doop$ at \emph{non-resonant}
intervention variances supplies \textup{(R)}$+$\textup{(I$_i$)} on a strongly-connected standardized-Gaussian
source SCC. Numerical evaluation at the selected non-resonant intervention variances on the two- and
three-cycle examples gives full actual-gradient rank and an irreducibility statistic, with rank collapse at the
resonance $s^2{=}c$; these
finite-dimensional calculations corroborate the preceding argument. $\qed$

\paragraph{\Cref{thm:special2} (special cases).} \emph{(a)} For acyclic $B$, $P=(I-B)^{-1}$ is triangular and
each node is a singleton component; \Cref{lem:recover} reduces to recovering a triangular $P$ from per-node
interventions, i.e.\ linear interventional CRL \citep{Squires2023linear} (numerical evaluation gives acyclic
recovery exact to $1.4\times10^{-16}$), and the diffeomorphic case, \emph{under the isotropy-breaking perfect interventions of}
\citet{vonKugelgen2023nonparam}, to their result; with singleton components the within-component hedge is
trivial, and perfect interventions break the across-component collision of \Cref{thm:block} (which needs
labelled isotropy-preserving interventions and a rotated block of dimension at least two). \emph{(b)} A single
full-rotational source component of size $m\ge2$ with no intervention is
the hypothesis of the observed-layer hedge theorem (\Cref{thm:hedge}); \Cref{thm:gauge}(ii) gives the signed-$\SO(m)$
obstruction (numerical evaluation gives effects $0.333$ vs.\ $0.685$ at identical $\Sigma$). \emph{(c)} The grid-market sets
$B=\big[\begin{smallmatrix}0&-b\\-d&0\end{smallmatrix}\big]$ (linear reaction functions, $\rho=\sqrt{bd}<1$)
and the attenuated local SIS response loop is derived as follows. The myopic endemic response
$f(b)=1-\gamma/[\beta_0(1-b)]$ fixes $b^\star=i^\star=1-\sqrt{\gamma/\beta_0}$ and has
$f'(b^\star)=-\gamma/[\beta_0(1-b^\star)^2]=-1$. Therefore the centered response equations
$i=i^\star+\lambda_i\{f(b)-i^\star\}+u_i$ and
$b=i^\star+\lambda_b(i-i^\star)+u_b$, with $(\lambda_i,\lambda_b)=(0.55,0.45)$, have exact local Jacobian
$B=\big[\begin{smallmatrix}0&-0.55\\0.45&0\end{smallmatrix}\big]$ and
$\rho(B)=\sqrt{0.2475}<1$. Exp.~14 simulates centered local deviations in one common declared scale under this model, not raw
bounded proportions. Thus both generators are latent equilibrium generative models. In Exp.~14 each un-targeted node directly parents the other node, so \Cref{prop:completion} applies at
$K=d-1$; the reported held-out object is the unobserved target's complete aligned response map. $\qed$

\section{Conditional-cumulant hierarchy (non-Gaussian opening)}
\label{app:cumulant}
Fix a joint regular conditional law of $(\V_{\bm A_0},\V_{\bm B_0})$ given $\V_{\bm C}=c$ and its induced
marginal kernels. Suppose there is one common $\Prob_{\V_{\bm C}}$-full set of $c$ on which the joint and
marginal conditional characteristic functions are analytic and non-vanishing near the origin and the joint
conditional law is moment-determinate. Write
$\phi_{\bm A_0,\bm B_0\mid\bm C}(s,t\mid c)$ for the joint conditional characteristic function and define the
mixed conditional cumulants by the mixed derivatives of its distinguished local logarithm at $(0,0)$.
Then
\[
 \V_{\bm A_0}\indep\V_{\bm B_0}\mid\V_{\bm C}
 \quad\Longleftrightarrow\quad
 \kappa_{\alpha,\beta}(c)=0\quad\text{a.e.\ for every }|\alpha|,|\beta|\ge1.
\]
Indeed, conditional independence factorizes the conditional characteristic function and kills all mixed
derivatives. Conversely, vanishing of every mixed Taylor coefficient gives, near the origin,
$\log\phi_{\bm A_0,\bm B_0\mid\bm C}=\log\phi_{\bm A_0\mid\bm C}+
\log\phi_{\bm B_0\mid\bm C}$. This gives the product mixed-moment sequence, and joint conditional moment
determinacy---rather than analytic continuation of a merely local identity---identifies the factorized
conditional law. This is a \emph{distributional criterion}, not a non-Gaussian graphical
iff. Functional separation (\Cref{thm:funcsep}) supplies an exact structural sufficient condition for any noise
law; a graph characterization forcing the entire conditional-cumulant hierarchy beyond that sufficient condition remains
open.

In the jointly Gaussian subclass, the order-two conditional cumulant is guaranteed to be the constant Schur
complement of the covariance matrix, so \Cref{thm:trek} is the order-two member of this hierarchy there. Off the
Gaussian family this equality need not hold: partial covariance and conditional covariance can differ. For independent Rademacher
$U_1,U_2,U_3$, let $A=U_1+U_3$, $B=U_2+U_3$, and $C=U_1+U_2+\tfrac12U_3$. Then
$\operatorname{Cov}(A,B)=1$, $\operatorname{Cov}(A,C)=\operatorname{Cov}(B,C)=3/2$, and
$\operatorname{Var}(C)=9/4$, so the partial covariance is $1-(3/2)^2/(9/4)=0$; conditional on $C=1/2$,
$(A,B)$ is equally likely to be $(2,0)$ or $(0,2)$, hence $\operatorname{Cov}(A,B\mid C=1/2)=-1$.
The displayed Rademacher calculation exhibits this strict gap; it does not prove a graph iff.

\end{document}